\documentclass[11pt]{article}

\usepackage{amssymb,amsfonts,amsmath, cite}
\usepackage{enumerate}

\usepackage[english]{babel}
\usepackage[cp1250]{inputenc}

\usepackage{algorithm2e}

\usepackage{longtable}

\usepackage{graphicx}
\usepackage{index}
\usepackage{fancyhdr}
\usepackage{lhelp}
\usepackage{optparams}
\usepackage{psfrag}
\usepackage{forloop}
\usepackage{setspace}
\usepackage{tikz}
\usepackage{subcaption}
\usepackage{pgffor}
\usepackage{xifthen}

\definecolor{lgray}{gray}{0.75}

\usetikzlibrary{decorations.pathreplacing,automata,calc,positioning}

\newcommand{\pch}{\chi_{\rho}}
\newcommand{\diam}{{\rm diam}}

\newcommand{\qed}{\hfill $\square$ \bigskip}

\newcommand{\mptt}[1]{}

\newtheorem{theorem}{Theorem} %ce pise na koncu Chapter, bo potem vse ?tevil?eno 1.15 in podobno%
\newtheorem{corollary}[theorem]{Corollary}
\newtheorem{lemma}[theorem]{Lemma}

\newtheorem{problem}{Problem}

\newcommand{\NN}{\mathbb N}

\begin{document}

\title{\bf Packing coloring of Sierpi\'{n}ski-type graphs}

\author{
Bo\v{s}tjan Bre\v{s}ar $^{a,b}$  \and Jasmina Ferme $^{c,a}$ 
 }

\date{}

\maketitle

\begin{center}
$^a$ Faculty of Natural Sciences and Mathematics, University of Maribor, Slovenia\\
\medskip

$^b$ Institute of Mathematics, Physics and Mechanics, Ljubljana, Slovenia\\
\medskip

$^c$ Faculty of Education, University of Maribor, Slovenia\\

\end{center}

\begin{abstract}
The packing chromatic number $\pch(G)$ of a graph $G$ is the smallest integer $k$ such that the vertex set of $G$ can be partitioned into sets $V_i$,  $i\in \{1,\ldots,k\}$, where each $V_i$ is an $i$-packing. In this paper, we consider the packing chromatic number of several families of Sierpi\' nski-type graphs. While it is known that this number is bounded from above by $8$ in the family of Sierpi\' nski graphs with base $3$, we prove that it is unbounded in the families of Sierpi\' nski graphs with bases greater than $3$. On the other hand, we prove that the packing chromatic number in the family of Sierpi\' nski triangle graphs $ST^n_3$ is bounded from above by $31$. Furthermore, we establish or provide bounds for the packing chromatic numbers of generalized Sierpi\' nski graphs $S^n_G$ with respect to all connected graphs $G$ of order 4.
\end{abstract}


\medskip\noindent
{\bf AMS Subj.\ Class:} 05C15, 05C70, 05C12

%%%%%%%%%%%%%%%%%%%%%%%%%%%%%%%%%%%%%%%%%%%%%%%%%%%%%%%%%%%%%%%%
%%%%%%%%%%%%%%%%%%%%%%%%%%%%%%%%%%%%%%%%%%%%%%%%%%%%%%%%%%%%%%%%
\section{Introduction}
%%%%%%%%%%%%%%%%%%%%%%%%%%%%%%%%%%%%%%%%%%%%%%%%%%%%%%%%%%%%%%%%
%%%%%%%%%%%%%%%%%%%%%%%%%%%%%%%%%%%%%%%%%%%%%%%%%%%%%%%%%%%%%%%%
Given a graph $G$ and a positive integer $i$, an {\em $i$-packing} in $G$ is a subset $W$ of the vertex set of $G$ such that the distance between any two distinct vertices from $W$ is greater than $i$. This generalizes the notion of an independent set, which is equivalent to a $1$-packing. The {\em packing chromatic number} of $G$ is the smallest integer $k$ such that the vertex set of $G$ can be partitioned into sets $V_1,\ldots, V_k$, where $V_i$ is an $i$-packing for each $i\in [k]$ (where $[k] = \{1,\ldots, k\}$). This invariant is well defined in any graph $G$ and is denoted by $\pch(G)$.
The corresponding mapping $c:V(G)\longrightarrow [k]$ having the property that 
$c(u)=c(v)=i$ implies $d(u, v) > i$, where $d(u,v)$ is the usual shortest-path distance between $u$ and $v$, is called a {\em $k$-packing coloring}. The packing chromatic number was introduced in~\cite{goddard-2008} (the name broadcast chromatic number was used in the first paper, indicating potential applications of the concept), and was subsequently studied in a number of paper, see~\cite{argiroffo-2014, balogh-2017+, barnaby-2017, bkr-2007, bkr-2016, bkrw-2017a, bkrw-2017b, ekstein-2014, fiala-2010, fiala-2009, finbow-2010, goddard-2012, jacobs-2013, korze-2014, lbe-2016, shao-2015, togni-2014, torres-2015}.

One of the main areas of investigation has been to determine the packing chromating numbers of infinite graphs such as infinite grids, lattices, distance graphs, etc.~\cite{barnaby-2017,bkr-2007,ekstein-2014,fiala-2009,finbow-2010,korze-2014}. For instance, the question of what is the packing chromatic number of the infinite square grid was initiated already in the seminal paper~\cite{goddard-2008}, and in some further papers a lower and an upper bound for this number were improved;  currently, the packing chromatic number of the square grid is known to lie between 13 and 15~\cite{barnaby-2017}. While the packing chromatic number of the hexagonal lattice was shown to be 7~ (joint efforts in the papers~\cite{bkr-2007,fiala-2009,korze-2014} give this value),  that of the triangular lattice is infinite~\cite{finbow-2010}. 

The results of Sloper~\cite{sloper-2004} imply that in the infinite 3-regular tree the packing chromatic number is 7, while in the infinite $k$-regular tree with $k>3$ the packing chromatic number is infinite. From this, one can infer that in the class of graphs with degree bounded by $k$ the packing chromatic number is unbounded as soon as $k\ge 4 $. Several recent papers considered the question of whether the packing chromatic number in the class of graphs with maximum degree 3 (i.e., subcubic graphs) is bounded by some constant \cite{bkr-2016,gt-2016,bkrw-2017a}, but in the very recent, not yet published paper~\cite{balogh-2017+} the authors prove that this is not the case. Nevertheless, the question of boundedness of the packing chromatic number in some natural infinite classes of graphs remains interesting. For instance, for the Sierpi\' nski graphs with base $3$, $S^n_3$, it was proven that their packing chromatic number is between $8$ and $9$ as soon as $n\ge 5$~\cite{bkr-2016}. (In the recent manuscript~\cite{vesel} it was shown that in fact $\pch(S^n_3)=8$ for $n\ge 5$.) As the Sierpi\' nski graphs $S^n_k$ form a fractal-like class of graphs, which can be built by a recursive procedure in which the graphs from the class of smaller dimension are used as building blocks, it seems particularly interesting to study the packing coloring in these graphs. Let us recall the definition of this class of graphs.

Let $[k]_0=\{0, 1, 2, \ldots, k-1\}$. The \textit{Sierpi\'nski graphs} $S^n_k$ of {\em dimension $n$}, $n \geq 1$, and {\em base $k$} have $[k]_0^n$ as the vertex set, and the edge set is defined recursively as $$ E(S^n_k)=\{\{is, it\}: i \in [k]_0; \{s, t\} \in E(S^{n-1})\} \cup \{\{ij^{n-1}, ji^{n-1}\}; i, j \in [k]_0, i  \neq j \}. $$
In other words, $S^n_k$ can be constructed from $k$ copies of $S^{n-1}_k$ in the following way. For each $j \in [k]_0$ concatenate $j$ to the left of the vertices in a copy of $S^{n-1}_k$ and denote the obtained graph by $jS^{n-1}$. Next for each $i \neq j$ join copies $iS^{n-1}$ and $jS^{n-1}$ by the single edge $\{ij^{n-1}, ji^{n-1}\}$. 
Note that the diameter of $S^n_k$ is $2^{n}-1$.

A natural question arises whether the packing chromatic number in the class of Sierpi\' nski graphs with bases greater than 3 is also bounded from above as is the case with base 3 Sierpi\' nski graphs. In Section~\ref{sec:base4} we prove that this is not the case, even when the class $S^{n}_4$ is considered. In the proof we use the recursive structure of these graphs, and combine a lower bound on $\pch(S^n_4)$ with the diameter of $S^{n+1}_4$ to obtain a lower bound for $\pch(S^{n+1}_4)$; the resulting sequence of lower bounds turns out to be positive and increasing, which suffices for the proof. This result, which is in a sense negative, motivated us to consider other variations of Sierpi\' nski-type graphs. The first such variation are generalized Sierpi\' nski graphs, as introduced in~\cite{gkp-11}. They have similarly defined recursive structure as Sierpi\' nski graphs, but  the edges, loosely speaking, reflect the structure of a graph $G$ with respect to which the generalized Sierpi\' nski graph $S^n_G$ is defined (in particular, the generalized Sierpi\' nski graph $S^n_{K_k}$ coincides with $S^n_k$). Total colorings and standard colorings of the generalized Sierpi\' nski graphs have already been considered, cf.~\cite{gs-15,rre-17+}. The formal definition of these graphs is presented in Section~\ref{sec:base4}. In Section~\ref{sec:generalized} we study the packing chromatic numbers of the generalized Sierpi\' nski graphs $S^n_G$ for all connected graphs $G$ of order $4$. As $S^n_G$ is a spanning subgraph of $S^{n}_4$, its packing chromatic number is bounded from above by $\pch(S^{n}_4)$, which gives the possibility that the set of values $\{\pch(S^n_G)\,|\,n\in \NN\}$ is bounded for any such $G$ different from $K_4$. As it turns out, this is indeed the case, and for each of these classes of graphs we present either the exact values of their packing chromatic numbers, or we give upper and lower bounds for these numbers.

Finally, in Section~\ref{sec:trikotnik}, we consider yet another variation of Sierpi\' nski graphs, called the Sierpi\' nski triangle graphs, $ST^n_3, n\in \NN_0$. They were studied in a number of papers, but were also called by several different names such as Sierpi\' nski gasket graphs; e.g.~see~\cite{jk-09} where different types of vertex and edge colorings were considered for this class of graphs (cf. also~\cite{hkz-17} for a survey on Sierpi\' nski-type graphs in which one can also find an explanation of different terminology issues). Interestingly, the graphs in this class are subgraphs of the infinite triangular lattice, hence in the limit they yield a subgraph of the infinite triangular lattice. As mentioned above, it is known that the packing chromatic number of the triangular lattice is infinite~\cite{finbow-2010}. Thus, it is particularly interesting that this number is bounded in the class of Sierpi\' nski triangle graphs,  where our construction gives $\pch(ST^n_3)\le 31$.

%%%%%%%%%%%%%%%%%%%%%
%%%%%%%%%%%%
\section{Sierpi\' nski graphs $S^n_k$ with $k$ greater than 3}
\label{sec:base4}
%%%%%%%%%%%%%%
%%%%%%%%%%%%%%%%%%%%%
%
%
% 1. izrek
%

In this section we prove one of our main results that the packing chromatic number in the class of graphs $S^n_k$, where $k$ is a fixed integer greater than 3, is not bounded from above. This result motivates all further investigations in the subsequent sections.

\begin{theorem}
Let $k \geq 4$ and let $S^n_k$ be the Sierpi\'nski graph of dimension $n$ and base $k$. The sequence $(\chi_\rho(S^n_k))_{n \in \NN}$ is unbounded from above. 
\label{prvi_izrek}
\end{theorem}
\begin{proof}
Let $k$ be an arbitrary and fixed integer such that $k\ge 4$. In this proof, we build a sequence $(a_{n})_{n \in \NN}$, where each $a_n$ presents a lower bound for the packing chromatic number of the Sierpi\'nski graph $S^n_k$. Clearly,
as $S^1_k$ coincides with $K_k$, and $\pch(K_k)=k$, we can set $a_1=k$. 
 
Note that $S^{n+1}_k$ contains a subgraph, isomorphic to $S^n_k$. Hence for a packing coloring $c$ of graph $S^{n+1}_k$ at least $a_{n}$ colors are used. But, of these colors used by $c$, the colors $\diam(S^{n+1}_k),\diam(S^{n+1}_k)+1,\ldots, a_{n}$ can be used only once in the entire $S^{n+1}_k$. 
Each of the mentioned colors can be used only in one copy of $S^n_k$. Since $S^{n+1}_k$ contains $k$ disjoint copies of subgraphs isomorphic to $S^n_k$, the packing coloring $c$ of $S^{n+1}_k$ uses, beside the mentioned (at least $a_n$) colors, at least $(a_{n} - \diam(S^{n+1}_k)+1)\cdot(k-1)$ additional colors. 

Therefore, $$a_{n+1}=a_{n}+(a_{n} - \diam(S^{n+1}_k)+1)\cdot(k-1),$$ 
and note that $a_{n+1}$ is a lower bound for $\pch(S^{n+1}_k)$, 
since we assumed that $a_n$ is a lower bound for $\pch(S^n_k)$.
Since $\diam(S^{n+1}_k)=2^{n+1}-1$, $$a_{n+1}=k\cdot a_{n} - 2^{n+1} \cdot (k-1) + k-1.$$ By solving this recurrence relation we obtain $$a_{n+1}=\frac{(4-k)\cdot k^{n+1} -2\cdot(2-k)-2^{n+2}\cdot(k-1)}{2-k}\,.$$ We claim that $(a_n)_{n \in \NN}$ is a strictly increasing sequence. 

Indeed, since the inequality $\frac{(4-k)\cdot k^n-2 \cdot (2-k)-2^{n+1}\cdot (k-1)}{2-k}$ $<$ $\frac{(4-k) \cdot k^{n+1}-2 \cdot (2-k)-2^{n+2}\cdot (k-1)}{2-k}$ is equivalent to the inequality $2^{n+1} > (4-k) \cdot k^n$, we get $a_{n} < a_{n+1}$ for any $k \geq 4$. This means that the sequence $(a_n)_{n \in \NN}$ is strictly increasing as soon as $k\ge 4$ (naturally, for $k=3$ this is not the case, as the result of~\cite{bkr-2016} also shows). Thus $(\chi_{\rho} (S^n_k))_{n \in \NN}$ is unbounded when $k>3$. 
\qed
\end{proof}

As the situation of the packing chromatic number of the standard Sierpi\'{n}ski graphs is in some sense resolved by Theorem~\ref{prvi_izrek} and the results of~\cite{bkr-2016}, it is interesting to consider natural generalizations of this class of graphs. One such class are the generalized Sierpi\'{n}ski graphs, whose definition we now recall. 

Let $G$ be an undirected graph on the vertex set $[k]_0$. The \textit{generalized Sierpi\'nski graph} $S^n_G$ {\em of $G$ of dimension $n$} is the graph with vertex set $[k]^{n}_{0}$ and edge set defined as follows: vertices $u, v \in V(S^n_G)$ are adjacent if there exists $i \in \{1, 2, 3, \ldots, n\}$ such that: 
\begin{enumerate}[(i)]
\item $u_j=v_j$ if $j<i$; 
\item $u_i \neq v_i$ and $\{u_i, v_i\} \in E(G)$; 
\item $u_j=v_i$ and $v_j=u_i$ if $j > i$. 
\end{enumerate}

\noindent In other words, for edges of $S^n_G$ the following holds: if $\{u, v\}$ is an edge of $S^n_G$, there is an edge $\{x, y\}$ of $G$ such that the labels of $u$ and $v$ are: $u=\underline{w}xyy\dots y$, $v=\underline{w}yxx\dots x$, where $\underline{w}\in[k]_0^{\ell}$, $0 \leq \ell \leq n-1$. 

The generalized Sierpi\'nski graph $S^n_G$ can also be constructed recursively from $k$ copies of $S^{n-1}_G$. Analogously as for the Sierpi\'nski graphs, for each $j \in [k]_0$  add the label $j$ in front of the labels of all vertices in $S^{n-1}_G$, and denote the obtained graph by $jS^{n-1}_G$. 
Then for any edge $\{a, b\}$ of $G$, add an edge between the vertices $abb \ldots b$ and $baa\ldots a$. 
From this perspective it is clear that only the {\em extreme vertices} (i.e. the vertices $i^{n-1}$) of any $S^{n-1}_G$ can be endvertices of edges between distinct copies of $S^{n-1}_G$ in $S^n_G$.

If $1 \leq d < n$ and $\underline{w} \in [k]_0^d$, then the subgraph of $S^n_G$, induced by the vertices, whose labels begin with $\underline{w}$, is isomorphic to $S^{n-d}_G$. This subgraph is denoted by $\underline{w}S^{n-d}_G$. Note that $S^n_G$ contains $k^d$ pairwise distinct subgraphs $\underline{w}S^{n-d}_G$, for $\underline{w} \in [k]_0^d$.  
In the case $k=4$, a subgraph $i^{n-1}S^1_G$ is called an {\em extreme square}. 
The graph $G$ in the definition of $S^n_G$ is also called the {\em base graph} of the generalized Sierpi\'{n}ski graph.

From Theorem~\ref{prvi_izrek} we derive the following consequence for the generalized Sierpi\'nski graphs $S^n_G$ with graphs $G$ that contains cliques of size at least 4. 
%%% še nova posledica
\begin{corollary}
\label{cor:K4}
If a graph $G$ contains a subgraph isomorphic to $K_k$, where $k \geq 4$, then the sequence $(\chi_\rho(S^n_G))_{n \in \NN}$ is unbounded from above.
\end{corollary}
\begin{proof}
Let $G$ be a graph and $K_k$ its subgraph, where $k \geq 4$ is a fixed integer, and let the vertices of $K_k$ in $G$ be denoted by $[k]_0$.

Using induction we first show that $S^n_G$ contains a subgraph isomorphic to $S^n_k$ for any $n \geq 1$. 
Since $S^1_G$ is isomorphic to $G$ and $S^1_k$ is isomorphic to $K_k$, it is clear that $S^1_G$ contains a subgraph isomorphic to $S^1_k$. 
Next, by induction hypothesis $S^n_G$ contains a subgraph isomorphic to $S^n_k$. Recall that $S^{n+1}_G$ consists of $|V(G)|$ copies of $S^n_G$ and thus $S^{n+1}_G$ contains at least $k$ copies of subgraphs, isomorphic to $S^n_k$. Consider all such copies of $S^n_k$, which are subgraphs of $0S^n_G$, $1S^n_G$, \ldots, $(k-2)S^n_G$ and $(k-1)S^n_G$. We only need to show that there exist edges in $S^{n+1}_G$ between any two mentioned copies, i.e. between any two vertices $ijj\ldots j$ and $jii\ldots i$, $i \neq j$, $0 \leq i, j\leq k-1$. By the definiton of edges of generalized Sierpi\'nski graph such vertices are adjacent. Therefore $S^{n+1}_G$ contains a subgraph isomorphic to $S^{n+1}_k$ and our claim holds. 

By the previous paragraph, the packing chromatic number of $S^{n}_G$ is at least $\chi_\rho(S^{n}_k)$, $k \geq 4$, and hence by Theorem \ref{prvi_izrek} the sequence $(\chi_\rho(S^n_G))_{n \in \NN}$ is unbounded from above.
\qed
\end{proof}

%
%
%%%%%%%%%%%%%%%%%%%%%%%
%%%%%%%%%%%%%%%%
\section{Generalized Sierpi\' nski graphs with base graphs on 4 vertices} 
\label{sec:generalized}  
%%%%%%%%%%%%%%%
%%%%%%%%%%%%%%%%%%%%%%%

In this section, we study the packing chromatic numbers of $S^n_G$ for all connected graphs $G$ on 4 vertices. There are altogether six such graph, and are shown in Fig.~\ref{fig:grafi}  . Note that the case of $S^n_{K_4}$ was already considered in Section~\ref{sec:base4}.

%%%% 6 grafov 
%
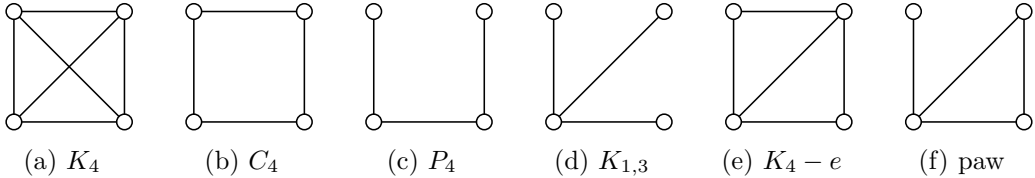
\begin{figure}[htb!]
%
% K_4
\begin{subfigure}[b]{0.13\textwidth}
       \centering
        \resizebox{\linewidth}{!}{
\begin{tikzpicture}%[scale=0.95, style=thick]
\def\vr{2pt}
\def\len{1}
\coordinate (x_0) at (0, 1);\coordinate (x_1) at (0, 0);\coordinate (x_2) at (1, 0);\coordinate (x_3) at (1, 1);
\draw (x_1) -- (x_2) -- (x_3) -- (x_0) -- (x_1); \draw (x_0) -- (x_2); \draw (x_1) -- (x_3);
\foreach \i in {0, 1, 2, 3}{
\draw (x_\i)[fill=white]circle(\vr);}
\end{tikzpicture}}
\caption{$K_4$}
%\end{center}
\end{subfigure}
\hspace{0.25cm}
%
% C_4
\begin{subfigure}[b]{0.13\textwidth}
       \centering
        \resizebox{\linewidth}{!}{
\begin{tikzpicture}%[scale=0.95, style=thick]
\def\vr{2pt}
\def\len{1}
\coordinate (x_0) at (0, 1);\coordinate (x_1) at (0, 0);\coordinate (x_2) at (1, 0);\coordinate (x_3) at (1, 1);
\draw (x_1) -- (x_2) -- (x_3) -- (x_0) -- (x_1);
\foreach \i in {0, 1, 2, 3}{
\draw (x_\i)[fill=white]circle(\vr);}
\end{tikzpicture}}
\caption{$C_4$}
%\end{center}
\end{subfigure}
%
% P_4
%
\hspace{0.25cm}
\begin{subfigure}[b]{0.13\textwidth}
       \centering
        \resizebox{\linewidth}{!}{
\begin{tikzpicture}%[scale=0.95, style=thick]
\def\vr{2pt}
\def\len{1}
\coordinate (x_0) at (0, 1);\coordinate (x_1) at (0, 0);\coordinate (x_2) at (1, 0);\coordinate (x_3) at (1, 1);
\draw (x_0)--(x_1) -- (x_2) -- (x_3);
\foreach \i in {0, 1, 2, 3}{
\draw (x_\i)[fill=white]circle(\vr);}
\end{tikzpicture}}
\caption{$P_4$}
%\end{center}
\end{subfigure}
\hspace{0.25cm}
% K_1,3
%
\begin{subfigure}[b]{0.13\textwidth}
       \centering
        \resizebox{\linewidth}{!}{
\begin{tikzpicture}%[scale=0.95, style=thick]
\def\vr{2pt}
\def\len{1}
\coordinate (x_0) at (0, 1);\coordinate (x_1) at (0, 0);\coordinate (x_2) at (1, 0);\coordinate (x_3) at (1, 1);
\draw (x_1) -- (x_2); \draw (x_1) -- (x_3); \draw (x_0) -- (x_1);
\foreach \i in {0, 1, 2, 3}{
\draw (x_\i)[fill=white]circle(\vr);}
\end{tikzpicture}}
\caption{$K_{1,3}$}
%\end{center}
\end{subfigure}
%
%
% K_4-e
\hspace{0.25cm}
\begin{subfigure}[b]{0.13\textwidth}
       \centering
        \resizebox{\linewidth}{!}{
\begin{tikzpicture}%[scale=0.95, style=thick]
\def\vr{2pt}
\def\len{1}
\coordinate (x_0) at (0, 1);\coordinate (x_1) at (0, 0);\coordinate (x_2) at (1, 0);\coordinate (x_3) at (1, 1);
\draw (x_1) -- (x_2) -- (x_3) -- (x_0) -- (x_1); \draw (x_1) -- (x_3);
\foreach \i in {0, 1, 2, 3}{
\draw (x_\i)[fill=white]circle(\vr);}
\end{tikzpicture}}
\caption{$K_4-e$}
%\end{center}
\end{subfigure}
\hspace{0.25cm}
% paw
\begin{subfigure}[b]{0.13\textwidth}
       \centering
        \resizebox{\linewidth}{!}{
\begin{tikzpicture}%[scale=0.95, style=thick]
\def\vr{2pt}
\def\len{1}
\coordinate (x_0) at (0, 1);\coordinate (x_1) at (0, 0);\coordinate (x_2) at (1, 0);\coordinate (x_3) at (1, 1);
\draw (x_1) -- (x_2) -- (x_3) -- (x_1); \draw (x_0) -- (x_1);
\foreach \i in {0, 1, 2, 3}{
\draw (x_\i)[fill=white]circle(\vr);}
\end{tikzpicture}}
\caption{paw}
%\end{center}
\end{subfigure}
\caption{All connected graphs on $4$ vertices.}
\label{fig:grafi}
\end{figure}

Before determining $\pch(S^n_{C_4})$ we need two lemmas.  

%% lemma 1
\begin{lemma}
The packing chromatic number of the graph $H$ in Fig.~\ref{fig:H*} is at least 5. 
\label{lemmaH}
\end{lemma}
\begin{proof} Suppose to the contrary that $\chi_{\rho}(H) \leq 4$. Let be $c$ be an arbitrary $4$-packing coloring of a graph $H$. Consider the following possibilities for vertices $a_1$ and $c_1$ of $H$ to be colored by color $1$ or not. 

%% GRAF H 
\begin{figure}[htb!]
\begin{center}
\begin{tikzpicture}%[scale=0.95, style=thick]
\def\vr{3pt}
\def\len{1}
	
\coordinate(x_1) at (0, 0); 
\coordinate(a_1) at (1, 0); 
\coordinate(b_1) at (2, 0); 
\coordinate(c_1) at (3, 0); 
\coordinate(c_2) at (4, 0); 
\coordinate(b_2) at (5, 0); 
\coordinate(a_2) at (6, 0); 
\coordinate(x_2) at (7, 0);

\coordinate(x_3) at (1, 1); 
\coordinate(x_4) at (3, 1); 
\coordinate(x_5) at (4, 1); 
\coordinate(x_6) at (6, 1);

\coordinate(x_7) at (1, 2); 
\coordinate(x_8) at (2, 2); 
\coordinate(x_14) at (2, 3); 
\coordinate(x_13) at (1, 3); 
\coordinate(x_9) at (3, 2);

\coordinate(x_10) at (4, 2); 
\coordinate(x_12) at (6, 2); 
\coordinate(x_15) at (5, 3); 
\coordinate(x_16) at (6, 3); 
\coordinate(x_11) at (5, 2);

\draw (x_1) -- (a_1) -- (b_1) -- (c_1) -- (c_2) -- (b_2) -- (a_2) -- (x_2);
\draw (a_1) -- (x_3) -- (x_7);
\draw (c_1) -- (x_4) -- (x_8) -- (x_13);
\draw (x_8) -- (x_14);
\draw (x_4) -- (x_9);

\draw (a_2) -- (x_6) -- (x_12);
\draw (c_2) -- (x_5) -- (x_11) -- (x_16);
\draw (x_11) -- (x_15);
\draw (x_5) -- (x_10);

\foreach \i in {1, 2, 3, 4, 5, 6, 7, 8, 9, 10, 11, 12, 13, 14, 15, 16} {\draw(x_\i)[fill=white] circle(\vr);}

\foreach \i in {1, 2} {\draw(a_\i)[fill=white] circle(\vr);
\draw(a_\i) node[below] {$a_\i$};
}
\foreach \i in {1, 2} {\draw(b_\i)[fill=white] circle(\vr);
\draw(b_\i) node[below] {$b_\i$};
}
\foreach \i in {1, 2} {\draw(c_\i)[fill=white] circle(\vr);
\draw(c_\i) node[below] {$c_\i$};
}

\draw(x_1) node[below] {$x_1$}; 
\draw(x_2) node[below] {$x_2$};

\draw(x_3) node[left] {$x_3$}; 
\draw(x_4) node[left] {$x_4$}; 
\draw(x_7) node[left] {$x_7$}; 
\draw(x_8) node[left] {$x_8$}; 
\draw(x_9) node[left] {$x_9$}; 
\draw(x_13) node[left] {$x_{13}$}; 
\draw(x_14) node[left] {$x_{14}$};

\draw(x_5) node[right] {$x_5$};
\draw(x_6) node[right] {$x_6$};
\draw(x_10) node[right] {$x_{10}$};
\draw(x_11) node[right] {$x_{11}$};
\draw(x_12) node[right] {$x_{12}$};
\draw(x_15) node[right] {$x_{15}$};
\draw(x_16) node[right] {$x_{16}$};

\end{tikzpicture}
\end{center}
\caption{Graph $H$}
\label{fig:H*}
\end{figure}

\textbf{Case 1.} $c(a_1)=c(c_1)=1$. \\
Since $\chi_{\rho}(H) \leq 4$, all five vertices in $N(a_1) \cup N(c_1)$ can be colored only by colors $2$, $3$ or $4$. Considering the distances between these vertices, colors $2$ and $3$ can be used only twice and color $4$ only once in any packing coloring. Actually, colors $2$ and $3$ are certainly used exactly twice and color $4$ exactly once, but this is possible only if $c(b_1)=4$. 
If $c(c_2)=2$, then $c(x_4)=3$ and $c(x_5)=1$. Only one of the vertices $x_{10}$ and $x_{11}$ can be colored by color $3$, but then for the other there is no available color from $\{1, 2, 3, 4\}$. Hence $c(c_2) \neq 2$, but $c(c_2)=3$, and then $c(x_4)=2$. Vertices $x_8$ and $x_9$ are colored by color 1 and only one of the vertices $x_{13}$ or $x_{14}$ can get color $3$, but the other cannot be colored by colors from $\{1, 2, 3, 4\}$. Therefore $c(c_2) \neq 3$ and thus $c_2$ cannot be colored by any of the colors from $\{1, 2, 3, 4\}$. Therefore $c(a_1)=c(c_1)=1$ yields a contradiction to the assumption, that $\chi_{\rho}(H) \leq 4$. 

\textbf{Case 2.} $c(c_1)=1$, $c(a_1) \neq 1$.\\
Since $\chi_{\rho}(H) \leq 4$ and $c_1$ has $3$ neighbours, each of them gets one of the colors $2$, $3$ or $4$. 
With respect to the Case $1$ and considering the distances between vertex $a_1$ and neighbours of vertex $c_1$, vertex $a_1$ cannot be colored by any of the available colors, if $c(b_1)=2$. 
Suppose then $c(b_1)=4$. If $c(x_4)=2$, $x_8$ and $x_9$ are both colored by color $1$, and then one of the vertices $x_{13}$ or $x_{14}$ cannot be colored by any of the available colors. If $c(x_4)=3$, then $c(c_2)=2$ and $c(x_5)=1$. But then there is no available color for one of the vertices $x_{10}$ and $x_{11}$. 
Therefore $c(b_1)=3$. If $c(x_4)=2$, we have the analogous situation as above (when $c(b_1)=4$ and $c(x_4)=2$), so this is not possible. Then $c(x_4)=4$ and $c(c_2)=2$, but then we have the same situation as in the case when $c(b_1)=4$ and $c(c_2)=2$. 
Again in each case we get a contradition to our assumption, so $c(c_1) \neq 1$. \\
By symmetry, we derive from the above two cases, that also $c(c_2) \neq 1$. 

\textbf{Case 3.} $c(a_1)=1$, $c(c_1) \neq 1$. \\
In this case each of the vertices $x_1$, $x_3$ and $b_1$ gets one of the colors $2$, $3$, or $4$ and the same holds also for vertices $c_1$ and $c_2$ by the Case $2$. For the packing coloring of vertices $x_1$, $x_3$, $b_1$ and $c_1$ color $2$ is used exactly twice and hence $c(c_1)=2$. Therefore vertices $b_1$ and $c_2$ get colors $3$ and $4$. Then $c(x_4)=1$. But there are no available colors for vertices $x_8$ and $x_9$. Therefore $c(a_1) \neq 1$. \\
By symmetry and by the Case $2$, also $c(a_2) \neq 1$. 

\textbf{Case 4.} $c(a_1) \neq 1$, $c(c_1) \neq 1$, $c(c_2) \neq 1$, $c(a_2) \neq 1$. 	\\	
Suppose $c(c_1)=4$ and $c(c_2)=2$. Then $c(a_2)=3$ and  $c(x_5)=1$. Only one of the vertices $x_{10}$ and $x_{11}$ can be colored by color $3$, but then the other cannot get any of the available colors, so in this case there does not exist $4$-packing coloring of $H$.
Let $c(c_1)=4$ and $c(c_2)=3$. Vertex $x_5$ can get colors $1$ or $2$. If it is colored by color $1$, then one of the vertices $x_{10}$ or $x_{11}$ is colored by color $2$, but the other cannot get any of the available colors. If $c(x_5)=2$, then $c(x_{11})=1$ and there is no available colors for vertices $x_{15}$ and $x_{16}$. 
Since $c(c_1)$ and $c(c_2)$ cannot be colored by color 4 (for $c_2$ the proof is analogous), without loss of generality we may assume that $c(c_1)=2$ and $c(c_2)=3$. Then $c(a_1)=4$ and $c(x_4)=1$. Vertices $x_8$ and $x_9$ cannot be colored by colors from $\{1, 2, 3, 4\}$. Hence also in this case $c$ cannot be $4$-packing coloring of $H$. 

In conclusion, regardless of which vertices of graph $H$ are colored by color $1$, $H$ cannot be properly $4$-packing colored, hence $\chi_{\rho}(H) >4$. 
\qed
\end{proof}

%% Graf H'

\begin{figure}[h]
\begin{center}
\begin{tikzpicture}%[scale=0.95, style=thick]
\def\vr{3pt}
\def\len{1}
\foreach \i in {0, 1, 2, 3, 4, 5, 6, 7}{
\coordinate(x_\i) at (\i, 0);}

\draw (x_0) -- (x_3);
\draw (x_4) -- (x_7);
\draw (5, 0) arc  (0:180:1.5);

\foreach \i in {0, 1, 2, 3, 4, 5, 6, 7}{
\draw(x_\i)[fill=white] circle(\vr);}

\foreach \i in {0, 1, 2, 3, 4, 5, 6, 7} 
{\draw(x_\i) node[below] {$x_\i$};}

\end{tikzpicture}
\end{center}
\caption{Graph $H'$} 
\label{fig:subgraph2}
\end{figure}

% lemma 2

\begin{lemma}
The packing chromatic number of the graph $H'$, shown in Fig.~\ref{fig:subgraph2}, is at least 4. 
\label{lemmaH'}
\end{lemma}

\begin{proof} Suppose to the contrary that $c$ is a $3$-packing coloring of $H'$. 
If $c(x_2)=1$, then all three of its neighbours are colored by colors $2$ and $3$, which is not possible since the distance between any two of them is $2$. Therefore $c(x_2) \neq 1$ and by symmetry also $c(x_5) \neq 1$. Without loss of generality suppose $c(x_2)=2$ and $c(x_5)=3$. Then only one of the vertices $x_0$ and $x_1$ can be colored by color $1$, but the other cannot get any of the available colors. Hence $\chi_{\rho}(H')>3$.
\qed
\end{proof}

%% C_4 in H  

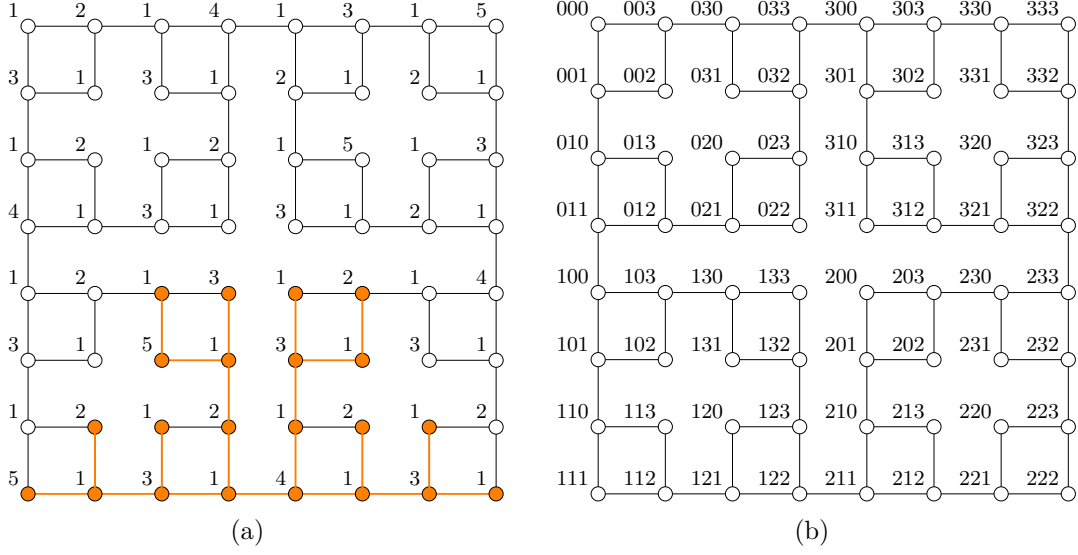
\begin{figure}[h]
\begin{subfigure}[b]{0.47\textwidth}
        \centering
        \resizebox{\linewidth}{!}{
\begin{tikzpicture}%[scale=0.95, style=thick]
\def\vr{3pt}
\def\len{1}
\foreach \i in {0, 1, 2, 3, 4, 5, 6, 7}{
\coordinate(x_\i) at (\i, 0);
\coordinate(y_\i) at (\i, 1);
\coordinate(z_\i) at (\i, 2);
\coordinate(w_\i) at (\i, 3);
\coordinate(u_\i) at (\i, 4);
\coordinate(v_\i) at (\i, 5);
\coordinate(o_\i) at (\i, 6);
\coordinate(p_\i) at (\i, 7);
}
\foreach \i in {0, 1, 2, 3, 4, 5, 6, 7}{
\draw (u_\i) -- (v_\i);
\draw (o_\i) -- (p_\i);
\draw (x_\i) -- (y_\i);
\draw (z_\i) -- (w_\i);
}
\draw (y_0) -- (z_0); \draw (y_3) -- (z_3); \draw (y_4) -- (z_4);\draw (y_7) -- (z_7);
\draw (v_0) -- (o_0); \draw (v_3) -- (o_3); \draw (v_4) -- (o_4); \draw (v_7) -- (o_7);
\draw (w_0) -- (u_0); \draw (w_7) -- (u_7);
\draw[orange, thick] (x_0) -- (x_1); \draw[orange, thick] (x_1) -- (x_2); \draw[orange, thick] (x_2) -- (x_3); \draw[orange, thick] (x_3) -- (x_4); \draw[orange, thick] (x_4) -- (x_5); \draw[orange, thick] (x_5) -- (x_6); \draw[orange, thick] (x_6) -- (x_7);
\draw (p_0) -- (p_1); \draw (p_1) -- (p_2); \draw (p_2) -- (p_3); \draw (p_3) -- (p_4); \draw (p_4) -- (p_5); \draw (p_5) -- (p_6); \draw (p_6) -- (p_7);
\draw (z_0) -- (z_1); \draw[orange, thick]  (z_2) -- (z_3); \draw[orange, thick]  (z_4) -- (z_5); \draw (z_6) -- (z_7);
\draw (y_0) -- (y_1);\draw (y_2) -- (y_3); \draw (y_4) -- (y_5); \draw (y_6) -- (y_7);
\draw (v_0) -- (v_1); \draw (v_2) -- (v_3); \draw (v_4) -- (v_5); \draw (v_6) -- (v_7);
\draw (u_0) -- (u_1); \draw (u_1) -- (u_2); \draw (u_2) -- (u_3); \draw (u_4) -- (u_5); \draw (u_5) -- (u_6); \draw (u_6) -- (u_7);
\draw (w_0) -- (w_1); \draw (w_1) -- (w_2); \draw (w_2) -- (w_3); \draw (w_4) -- (w_5); \draw (w_5) -- (w_6); \draw (w_6) -- (w_7);
\draw (o_0) -- (o_1); \draw (o_2) -- (o_3); \draw (o_4) -- (o_5); \draw (o_6) -- (o_7);
\foreach \i in {0, 1, 2, 3, 4, 5, 6, 7}{
% Na koncu jih pobarvaj, da so celi beli! 
\draw(x_\i)[fill=orange] circle(\vr);
\draw(u_\i)[fill=white] circle(\vr);
\draw(v_\i)[fill=white] circle(\vr);
\draw(o_\i)[fill=white] circle(\vr);
\draw(p_\i)[fill=white] circle(\vr);
}
\draw (y_0)[fill=white] circle(\vr);
\draw (y_7)[fill=white] circle(\vr);
\foreach \i in {1, 2, 3, 4, 5, 6}{
\draw(y_\i)[fill=orange] circle(\vr);}
\draw(z_0)[fill=white] circle(\vr);
\draw(z_1)[fill=white] circle(\vr);
\draw(z_6)[fill=white] circle(\vr);
\draw(z_7)[fill=white] circle(\vr);
\draw(z_2)[fill=orange] circle(\vr);
\draw(z_3)[fill=orange] circle(\vr);
\draw(z_4)[fill=orange] circle(\vr);
\draw(z_5)[fill=orange] circle(\vr);
\draw(w_0)[fill=white] circle(\vr);
\draw(w_1)[fill=white] circle(\vr);
\draw(w_6)[fill=white] circle(\vr);
\draw(w_7)[fill=white] circle(\vr);
\draw(w_3)[fill=orange] circle(\vr);
\draw(w_4)[fill=orange] circle(\vr);
\draw(w_5)[fill=orange] circle(\vr);
\draw(w_2)[fill=orange] circle(\vr);
\draw[orange, thick] (y_3) -- (z_3);
\draw[orange, thick] (y_4) -- (z_4);
\draw[orange, thick] (z_2) -- (w_2);
\draw[orange, thick] (z_3) -- (w_3);
\draw[orange, thick] (z_4) -- (w_4);
\draw[orange, thick] (z_5) -- (w_5);
\draw[orange, thick] (x_1) -- (y_1);
\draw[orange, thick] (x_2) -- (y_2);
\draw[orange, thick] (x_3) -- (y_3);
\draw[orange, thick] (x_4) -- (y_4);
\draw[orange, thick] (x_5) -- (y_5);
\draw[orange, thick] (x_6) -- (y_6);
%
% Barvanje 5
\draw(x_0) node[above left] {\footnotesize{5}}; \draw(z_2) node[above left] {\footnotesize{5}}; \draw(v_5) node[above left] {\footnotesize{5}}; \draw(p_7) node[above left] {\footnotesize{5}};
% Barvanje 4
\draw(x_4) node[above left] {\footnotesize{4}}; \draw(w_7) node[above left] {\footnotesize{4}}; \draw(u_0) node[above left] {\footnotesize{4}}; \draw(p_3) node[above left] {\footnotesize{4}};
% Barvanje 3
\draw(x_2) node[above left] {\footnotesize{3}}; \draw(x_6) node[above left] {\footnotesize{3}}; \draw(z_0) node[above left] {\footnotesize{3}}; \draw(z_4) node[above left] {\footnotesize{3}}; \draw(z_6) node[above left] {\footnotesize{3}}; \draw(w_3) node[above left] {\footnotesize{3}}; \draw(u_2) node[above left] {\footnotesize{3}}; \draw(u_4) node[above left] {\footnotesize{3}}; \draw(v_7) node[above left] {\footnotesize{3}}; \draw(o_0) node[above left] {\footnotesize{3}}; \draw(o_2) node[above left] {\footnotesize{3}}; \draw(p_5) node[above left] {\footnotesize{3}};
% Barvanje 2
\foreach \i in {1, 3, 5, 7}{
\draw(y_\i) node[above left] {\footnotesize{2}}; }
\draw(w_1) node[above left] {\footnotesize{2}};
\draw(w_5) node[above left] {\footnotesize{2}};
\draw(u_6) node[above left] {\footnotesize{2}};
\draw(v_1) node[above left] {\footnotesize{2}};
\draw(v_3) node[above left] {\footnotesize{2}};
\draw(o_4) node[above left] {\footnotesize{2}};
\draw(o_6) node[above left] {\footnotesize{2}};
\draw(p_1) node[above left] {\footnotesize{2}};
% Barvanje 1
\foreach \i in {1, 3, 5, 7}{
\draw(x_\i) node[above left] {\footnotesize{1}}; 
\draw(z_\i) node[above left] {\footnotesize{1}};
\draw(o_\i) node[above left] {\footnotesize{1}}; }
\draw(u_1) node[above left] {\footnotesize{1}}; 
\draw(u_3) node[above left] {\footnotesize{1}}; 
\draw(u_5) node[above left] {\footnotesize{1}}; 
\draw(u_7) node[above left] {\footnotesize{1}}; 
\foreach \i in {0, 2, 4, 6}{
\draw(y_\i) node[above left] {\footnotesize{1}}; 
\draw(v_\i) node[above left] {\footnotesize{1}};
\draw(p_\i) node[above left] {\footnotesize{1}}; }
\draw(w_0) node[above left] {\footnotesize{1}}; 
\draw(w_2) node[above left] {\footnotesize{1}};
\draw(w_4) node[above left] {\footnotesize{1}}; 
\draw(w_6) node[above left] {\footnotesize{1}};
\end{tikzpicture}}
\caption{}
\end{subfigure}
%
%
% Še oznake
\hspace*{0.3cm}
\begin{subfigure}[b]{0.5\textwidth}
        \centering
        \resizebox{\linewidth}{!}{
\begin{tikzpicture}%[scale=0.95, style=thick]
\def\vr{3pt}
\def\len{1}
\foreach \i in {0, 1, 2, 3, 4, 5, 6, 7}{
\foreach \j in {0, 1, 2, 3, 4, 5, 6, 7}{
\coordinate (x^\j_\i) at (\i, \j); }}
\foreach \i in {0, 1, 2, 3}{
\foreach \j in {0, 1, 2, 3, 4, 5, 6, 7}{
\draw (x^\j_\the\numexpr2*\i\relax) -- (x^\j_\the\numexpr2*\i+1\relax); 
\draw (x^\the\numexpr2*\i\relax_\j) -- (x^\the\numexpr2*\i+1\relax_\j); 
}}
\foreach \i in {0, 1}{
\foreach \j in {0, 1}{
\draw (x^\the\numexpr4*\i+1\relax_\the\numexpr3*\j\relax) -- (x^\the\numexpr4*\i+2\relax_\the\numexpr3*\j\relax); 
\draw (x^\the\numexpr4*\i+1\relax_\the\numexpr3*\j+4\relax) -- (x^\the\numexpr4*\i+2\relax_\the\numexpr3*\j+4\relax); 
\draw (x^\the\numexpr3*\i\relax_\the\numexpr4*\j+1\relax) -- (x^\the\numexpr3*\i\relax_\the\numexpr4*\j+2\relax) ; 
\draw (x^\the\numexpr3*\i+4\relax_\the\numexpr4*\j+1\relax) -- (x^\the\numexpr3*\i+4\relax_\the\numexpr4*\j+2\relax) ; 
}}
\draw(x^0_3) -- (x^0_4); \draw(x^7_3) -- (x^7_4); \draw(x^3_0) -- (x^4_0); \draw(x^3_7) -- (x^4_7);
\foreach \i in {0, 1, 2, 3, 4, 5, 6, 7}{
\foreach \j in {0, 1, 2, 3, 4, 5, 6, 7}{
\draw(x^\j_\i)[fill=white] circle(\vr); }}
%% oznake
% Oznake - najprej le 3S^2
\draw(x^4_4)node[above left]{\footnotesize{311}}; \draw(x^4_5)node[above left]{\footnotesize{312}}; \draw(x^4_6)node[above left]{\footnotesize{321}}; \draw(x^4_7)node[above left] {\footnotesize{322}};
\draw(x^5_4)node[above left]{\footnotesize{310}}; \draw(x^5_5)node[above left]{\footnotesize{313}}; \draw(x^5_6)node[above left] {\footnotesize{320}}; \draw(x^5_7)node[above left] {\footnotesize{323}};
\draw(x^6_4)node[above left]{\footnotesize{301}}; \draw(x^6_5)node[above left]{\footnotesize{302}}; \draw(x^6_6)node[above left] {\footnotesize{331}}; \draw(x^6_7)node[above left] {\footnotesize{332}};
\draw(x^7_4)node[above left]{\footnotesize{300}}; \draw(x^7_5)node[above left]{\footnotesize{303}}; \draw(x^7_6)node[above left] {\footnotesize{330}}; \draw(x^7_7)node[above left] {\footnotesize{333}};
%
%  0S^2
\draw(x^4_0)node[above left]{\footnotesize{011}}; \draw(x^4_1)node[above left]{\footnotesize{012}}; \draw(x^4_2)node[above left] {\footnotesize{021}}; \draw(x^4_3)node[above left] {\footnotesize{022}};
\draw(x^5_0)node[above left]{\footnotesize{010}}; \draw(x^5_1)node[above left]{\footnotesize{013}}; \draw(x^5_2)node[above left] {\footnotesize{020}}; \draw(x^5_3)node[above left] {\footnotesize{023}};
\draw(x^6_0)node[above left]{\footnotesize{001}}; \draw(x^6_1)node[above left]{\footnotesize{002}}; \draw(x^6_2)node[above left] {\footnotesize{031}}; \draw(x^6_3)node[above left] {\footnotesize{032}};
\draw(x^7_0)node[above left]{\footnotesize{000}}; \draw(x^7_1)node[above left]{\footnotesize{003}}; \draw(x^7_2)node[above left] {\footnotesize{030}}; \draw(x^7_3)node[above left] {\footnotesize{033}};
%
% 2S^2
\draw(x^0_4)node[above left]{\footnotesize{211}}; \draw(x^0_5)node[above left]{\footnotesize{212}}; \draw(x^0_6)node[above left] {\footnotesize{221}}; \draw(x^0_7)node[above left] {\footnotesize{222}};
\draw(x^1_4)node[above left]{\footnotesize{210}}; \draw(x^1_5)node[above left]{\footnotesize{213}}; \draw(x^1_6)node[above left] {\footnotesize{220}}; \draw(x^1_7)node[above left] {\footnotesize{223}};
\draw(x^2_4)node[above left]{\footnotesize{201}}; \draw(x^2_5)node[above left]{\footnotesize{202}}; \draw(x^2_6)node[above left] {\footnotesize{231}}; \draw(x^2_7)node[above left] {\footnotesize{232}};
\draw(x^3_4)node[above left]{\footnotesize{200}}; \draw(x^3_5)node[above left]{\footnotesize{203}}; \draw(x^3_6)node[above left] {\footnotesize{230}}; \draw(x^3_7)node[above left] {\footnotesize{233}};
%
% 1S^2
\draw(x^0_0)node[above left]{\footnotesize{111}}; \draw(x^0_1)node[above left]{\footnotesize{112}}; \draw(x^0_2)node[above left]{\footnotesize{121}}; \draw(x^0_3)node[above left] {\footnotesize{122}};
\draw(x^1_0)node[above left]{\footnotesize{110}}; \draw(x^1_1)node[above left]{\footnotesize{113}}; \draw(x^1_2)node[above left]{\footnotesize{120}}; \draw(x^1_3)node[above left] {\footnotesize{123}};
\draw(x^2_0)node[above left]{\footnotesize{101}}; \draw(x^2_1)node[above left]{\footnotesize{102}}; \draw(x^2_2)node[above left]{\footnotesize{131}}; \draw(x^2_3)node[above left] {\footnotesize{132}};
\draw(x^3_0)node[above left]{\footnotesize{100}}; \draw(x^3_1)node[above left]{\footnotesize{103}}; \draw(x^3_2)node[above left] {\footnotesize{130}}; \draw(x^3_3)node[above left]{\footnotesize{133}};
\end{tikzpicture}
}
\caption{}
\end{subfigure}
\caption{(a) A 5-packing coloring of $S^3_{C_4}$ with graph $H$ as its subgraph in orange; (b) the labeling of $S^3_{C_4}$.}
\label{fig:C_4}
\end{figure}

\begin{theorem}
If $n\geq 1$ and $S^n_{C_4}$ is the generalized Sierpi\'nski graph of $C_4$ of dimension $n$, then
$$\chi_{\rho}(S^n_{C_4}) =
\left\{\begin{array}{ll}
3; & n=1\,,\\
4; & n=2\,,\\
5; & n \geq 3\,.
\end{array}\right.$$
\label{theoremC_4}
\end{theorem}
\begin{proof} Since $S^1_{C_4}$ is isomorphic to $C_4$, it is clear that $\chi_{\rho}(S^1_{C_4})=3$. 
The graph $S^2_{C_4}$ contains a subgraph isomorphic to $H'$, which is shown in Fig.~\ref{fig:subgraph2}. Since $H'$ is not $3$-packing colorable, also $S^2_{C_4}$ is not $3$-packing colorable, thus $\chi_{\rho}(S^2_{C_4})>3$. As shown in Fig.~\ref{fig:C_4} (find the 4-packing coloring of the subgraph $2S^2_{C_4}$), there exists a $4$-packing coloring of $S^2_{C_4}$ and therefore $\chi_{\rho}(S^2_{C_4})=4$. 

First we show that $\chi_{\rho}(S^n_{C_4}) \leq 5$ holds for any $n \geq 3$. 
Recall that for each $n \geq 3$ graph $S^n_{C_4}$ contains $4^{n-3}$ distinct copies of subgraphs, isomorphic to $S^3_{C_4}$. Therefore it is enough to show that there exists a $5$-packing coloring of $S^3_{C_4}$, which can be used for packing coloring of each subgraph of $S^n_{C_4}$ that is isomorphic to $S^3_{C_4}$, and in this way forms a $5$-packing coloring of $S^n_{C_4}$ for any $n \geq 3$. \\
Color each subgraph of $S^n_{C_4}$, $n \geq 3$, which is isomorphic to $S^3_{C_4}$, using the packing coloring from Fig.~\ref{fig:C_4}. We claim that in that way we create a $5$-packing coloring of $S^n_{C_4}$. It is clear that this holds for $n=3$. 

The extreme vertices of $S^3_{C_4}$, which are colored by $1$, are only $000$ and $222$. Since $(0, 2) \notin E(C_4)$, by the definition of edges in $S^n_{C_4}$, $ n \geq 3$, vertices $\underline{u}000$ and $\underline{v}222$, where $\underline{u}, \underline{v} \in [4]_0^{n-3}$, both colored by $1$, are not adjacent in any $S^n_{C_4}$,  $n \geq 3$. Therefore the distance between any two vertices of $S^n_{C_4}$, $n \geq 3$, which are both colored by $1$, is at least $2$.

In $S^3_{C_4}$ none of the extreme vertices is colored by color $2$, hence any two vertices of $S^n_{C_4}$, $n \geq 3$, both colored by $2$, are at distance at least $3$.

Consider now any two vertices of $S^n_{C_4}$, $n \geq 3$, which are both colored by color $3$ and belong to two distinct copies $S^3_{C_4}$. 
If such two vertices do not belong to any extreme square of any copy of $S^3_{C_4}$, then they are at distance at least $5$. If one of these two vertices belongs to some extreme square (recall that this is not extreme vertex) of a copy of $S^3_{C_4}$, but the other not, then they are at distance at least $4$. 
The only pairs of problematic vertices can be those, which consist of two vertices, both colored by color $3$ and both belong to extreme squares in different copies of $S^3_{C_4}$. %(they can be at distance $3$, since they are not extreme vertices). 
Since the diameter of $S^3_{C_4}$ is $14$, any two vertices of $S^n_{C_4}$, which are both corresponding to vertex $001$ (respectively $221$) of $S^3_{C_4-e}$, are at distance more than $3$. Therefore, the only pairs of problematic vertices are actually $\underline{u}001$ and $\underline{v}221$, where $\underline{u}, \underline{v} \in [4]_0^{n-3}$, $\underline{u} \neq \underline{v}$.  But the extreme squares to which such two vertices belong ($\underline{u}00S^1_{C_4}$ and $\underline{v}22S^1_{C_4}$) are not connected by any edge in any $S^n_{C_4}$, $n\geq 4$, since $ (0,2) \notin E(C_4)$. Therefore also such vertices are at distance more than $3$. 

Each vertex in $S^3_{C_4}$, which is colored by color $4$, is at distance at least $3$ from any of the extreme vertices. 
%The sum of two numbers, which are both at least $3$, is more than $5$. 
Hence any two vertices of $S^n_{C_4}$ both colored by $4$ are at distance more than $5$.

Any non-extreme vertex of a copy of $S^3_{C_4}$, which is colored by color $5$, is in $S^n_{C_4}$, $n \geq 3$, at distance at least $7$ from any vertex, which is colored by $5$ and belong to the other copy of $S^3_{C_4}$. 
Also extreme vertices of copies of $S^3_{C_4}$, $\underline{u}111$ and $\underline{v}333$, $\underline{u}, \underline{v} \in [4]_0^{n-3}$, which are colored by color $5$ and belong to two different copies of $S^3_{C_4}$, are not adjacent in any $S^n_{C_4}$, $n \geq 3$, by the definition of edges of generalized Sierpi\'nski graphs (recall that $(1,3) \notin E(C_4)$). Hence also vertices $\underline{u}111$ and $\underline{v}333$ are at distance more than $5$.
Therefore the presented coloring is a $5$-packing coloring of $S^n_{C_4}$ for any $n \geq 3$ and thus $\chi_{\rho}(S^n_{C_4}) \leq 5$.

% >=
 
We next show that $\chi_{\rho}(S^n_{C_4}) \geq 5$ holds for any $n \geq 3$. Recall that for each $n \geq 3$ graph $S^n_{C_4}$ contains a subgraph isomorphic to $S^3_{C_4}$. Since $S^3_{C_4}$ contains a subgraph, isomorphic to $H$ from Fig.~\ref{fig:H*} as shown in Fig.~\ref{fig:C_4}, also each of the graphs $S^n_{C_4}$, $n \geq 3$, contains the subgraph, isomorphic to $H$. 
Using Lemma~\ref{lemmaH} we have $$4<\chi_{\rho}(H) \leq  \chi_{\rho}(S^n_{C_4})$$ for any $n \geq 3$ and thus $\chi_{\rho}(S^n_{C_4})=5$.
\qed
\end{proof}

% S^3_{P_4}
\begin{figure}
\begin{center}
\begin{tikzpicture}%[scale=0.95, style=thick]
\def\vr{3pt}
\def\len{1}

\foreach \i in {0, 1, 2, 3, 4, 5, 6, 7, 8 }{
\coordinate(x_\i) at (\i, 0); 
\coordinate(y_\i) at (\i, 1);
\coordinate(z_\i) at (\i, 2);
\coordinate(w_\i) at (\i, 3);
\coordinate(u_\i) at (\i, 4);
\coordinate(v_\i) at (\i, 5);
\coordinate(o_\i) at (\i, 6);
\coordinate(p_\i) at (\i, 7);
}

\foreach \i in {0, 1, 2, 3, 4, 5, 6, 7}{
\draw (u_\i) -- (v_\i);
\draw (o_\i) -- (p_\i);
}

\draw (x_0) -- (y_0); \draw[thick, orange](x_1) -- (y_1); \draw[thick, orange](x_2) -- (y_2); \draw[thick, orange](x_3) -- (y_3);
\draw[thick, orange](x_4) -- (y_4); \draw[thick, orange](x_5) -- (y_5); \draw[thick, orange](x_6) -- (y_6); \draw (x_7) -- (y_7);

\draw (z_0) -- (w_0); \draw (z_1) -- (w_1); \draw[thick, orange](z_2) -- (w_2); \draw[thick, orange](z_3) -- (w_3); \draw[thick, orange](z_4) -- (w_4); \draw[thick, orange](z_5) -- (w_5); \draw (z_6) -- (w_6); \draw (z_7) -- (w_7);

\draw (y_0) -- (z_0); \draw[thick, orange](y_3) -- (z_3); \draw[thick, orange](y_4) -- (z_4); \draw (y_7) -- (z_7);
   
\draw (v_0) -- (o_0); \draw (v_3) -- (o_3); \draw (v_4) -- (o_4); \draw (v_7) -- (o_7);
\draw (w_0) -- (u_0); \draw (w_7) -- (u_7);

\draw[thick, orange](x_0) -- (x_1); \draw[thick, orange](x_1) -- (x_2); \draw[thick, orange](x_2) -- (x_3); \draw[thick, orange](x_3) -- (x_4); \draw[thick, orange](x_4) -- (x_5); \draw[thick, orange](x_5) -- (x_6); \draw[thick, orange](x_6) -- (x_7);

\draw (z_0) -- (z_1); \draw[thick, orange](z_2) -- (z_3); \draw (z_4)[thick, orange]-- (z_5); \draw (z_6) -- (z_7);

\draw (u_0) -- (u_1); \draw (u_1) -- (u_2); \draw (u_2) -- (u_3); \draw (u_4) -- (u_5); \draw (u_5) -- (u_6); \draw (u_6) -- (u_7);

\draw (o_0) -- (o_1); \draw (o_2) -- (o_3); \draw (o_4) -- (o_5); \draw (o_6) -- (o_7);

\foreach \i in {0, 1, 2, 3, 4, 5, 6, 7}{
% Na koncu jih pobarvaj, da so celi beli! 
\draw(x_\i)[fill=orange] circle(\vr);
\ifthenelse{\i < 2}{\draw(z_\i)[fill=white] circle(\vr);}{}
\ifthenelse{\i > 5}{\draw(z_\i)[fill=white] circle(\vr);}{}
\ifthenelse{\i =2}{\draw(z_\i)[fill=orange] circle(\vr);}{}
\ifthenelse{\i =3}{\draw(z_\i)[fill=orange] circle(\vr);}{}
\ifthenelse{\i =4}{\draw(z_\i)[fill=orange] circle(\vr);}{}
\ifthenelse{\i =5}{\draw(z_\i)[fill=orange] circle(\vr);}{}
\ifthenelse{\i < 2}{\draw(w_\i)[fill=white] circle(\vr);}{}
\ifthenelse{\i > 5}{\draw(w_\i)[fill=white] circle(\vr);}{}
\ifthenelse{\i =2}{\draw(w_\i)[fill=orange] circle(\vr);}{}
\ifthenelse{\i =3}{\draw(w_\i)[fill=orange] circle(\vr);}{}
\ifthenelse{\i =4}{\draw(w_\i)[fill=orange] circle(\vr);}{}
\ifthenelse{\i =5}{\draw(w_\i)[fill=orange] circle(\vr);}{}
\ifthenelse{\i =0}{\draw(y_\i)[fill=white] circle(\vr);}{}
\ifthenelse{\i =1}{\draw(y_\i)[fill=orange] circle(\vr);}{}
\ifthenelse{\i =2}{\draw(y_\i)[fill=orange] circle(\vr);}{}
\ifthenelse{\i =3}{\draw(y_\i)[fill=orange] circle(\vr);}{}
\ifthenelse{\i =4}{\draw(y_\i)[fill=orange] circle(\vr);}{}
\ifthenelse{\i =5}{\draw(y_\i)[fill=orange] circle(\vr);}{}
\ifthenelse{\i =6}{\draw(y_\i)[fill=orange] circle(\vr);}{}
\ifthenelse{\i =7}{\draw(y_\i)[fill=white] circle(\vr);}{}
\draw(u_\i)[fill=white] circle(\vr);
\draw(v_\i)[fill=white] circle(\vr);
\draw(o_\i)[fill=white] circle(\vr);
\draw(p_\i)[fill=white] circle(\vr);
}
\end{tikzpicture}
\end{center}
\caption{Graph $S^3_{P_4}$ with graph $H$ as its induced subgraph in orange. }
\label{fig:S_3}

\end{figure}
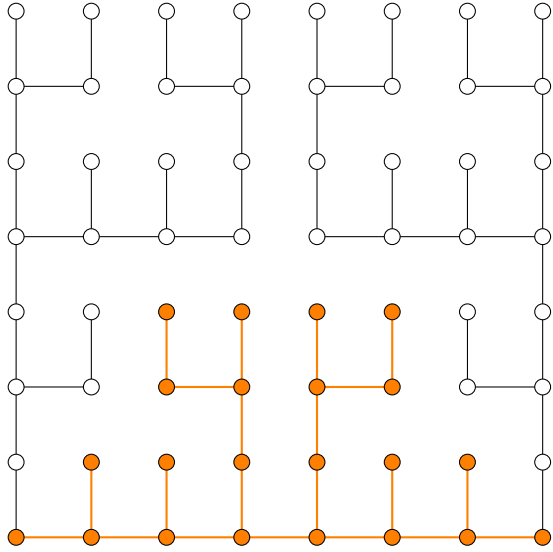

We follow with packing chromatic numbers of the generalized Sierpi\' nski graphs when $G$ is $P_4$.

\begin{theorem}
If $n\geq 1$ and $S^n_{P_4}$ is the generalized Sierpi\'nski graph of $P_4$ of dimension $n$, then
$$\chi_{\rho}(S^n_{P_4}) =
\left\{\begin{array}{ll}
3; & n=1\,,\\
4; & n=2\,,\\
5; & n \geq 3\,.
\end{array}\right.$$
\end{theorem}

\begin{proof} Since $S^1_{P_4}$ is isomorphic to $P_4$ it is clear that $\chi_{\rho}(S^1_{P_4})=3$. 
Graph $H'$ shown in Fig.~\ref{fig:subgraph2} is an induced subgraph of $S^2_{P_4}$. Hence by Lemma~\ref{lemmaH'}, $\chi_{\rho}(S^2_{P_4})>3$. Since $S^2_{P_4}$ is isomorphic to a subgraph of the graph $S^2_{C_4}$, we have $\chi_{\rho}(S^2_{P_4}) \leq 4$. Thus $\chi_{\rho}(S^2_{P_4}) = 4$. 

For any $n \geq 3$, the graph $H$ shown in Fig.~\ref{fig:H*} is an induced subgraph of $S^n_{P_4}$ (see Fig.~\ref{fig:S_3}) and graph $S^n_{P_4}$ is isomorphic to a subgraph of $S^n_{C_4}$. Therefore for any $n \geq 3$ we infer by using Lemma~\ref{lemmaH} and Theorem~\ref{theoremC_4} that
$$ 5 \leq \chi_{\rho}(S^n_{P_4}) \leq \chi_{\rho}(S^n_{C_4}) = 5.$$ Hence $\chi_{\rho}(S^n_{P_4})=5$ for any $n \geq 3$. \qed
\end{proof}

%
%% Zvezda
%
%
%
The next family of graphs $S^n_G$ for which we get a complete solution of our problem is when $G$ is the claw, $K_{1,3}$. We label the vertices of $K_{1, 3}$ with labels from the set $[3]_0$, such that vertex label by $1$ has degree $3$. This notation is then also the base for the labeling of vertices of $S^n_{K_{1, 3}}$, $n \geq 2$.

\begin{theorem}
If $n\geq 1$ and $S^n_{K_{1, 3}}$ is the generalized Sierpi\'nski graph of $K_{1,3}$ of dimension $n$, then
$$\chi_{\rho}(S^n_{K_{1, 3}}) =
\left\{\begin{array}{ll}
2; & n=1\,,\\
3; & n \geq 2\,.
\end{array}\right.$$
\end{theorem}
\begin{proof} As graph $S^1_{K_{1,3}}$ is isomorphic to $K_{1,3}$, we have $\chi_{\rho}(S^1_{K_{1, 3}})=2$. 
For each $n \geq 2$ the graph $S^n_{K_{1, 3}}$ contains a subgraph isomorphic to $P_4$, hence $\chi_{\rho}(S^n_{K_{1, 3}}) \geq 3$ for each $n \geq 2$. 

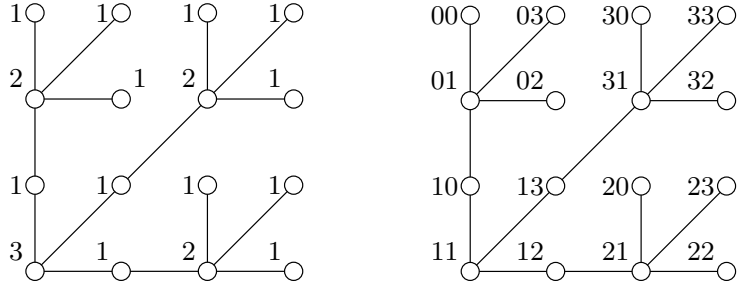
\begin{figure}[h]
\hspace*{2cm}
\begin{subfigure}[b]{0.3\textwidth}
        \centering
        \resizebox{\linewidth}{!}{
\begin{tikzpicture}%[scale=0.95, style=thick]
\def\vr{3pt}
\def\len{1}
\foreach \i in {0, 1, 2, 3}{
\foreach \j in {0, 1, 2, 3}{
\coordinate(x^\j_\i) at (\i, \j);}}
\draw (x^0_0) -- (x^0_3);
\draw (x^0_0) -- (x^3_0);
\draw (x^0_0) -- (x^3_3);
\draw (x^3_1) -- (x^2_0) -- (x^2_1);
\draw (x^3_2) -- (x^2_2) -- (x^2_3);
\draw (x^1_2) -- (x^0_2) -- (x^1_3);
\foreach \i in {0, 1, 2, 3}{
\foreach \j in {0, 1, 2, 3}{
\draw(x^\j_\i)[fill=white] circle(\vr);}
\draw(x^3_\i)node[left]{\footnotesize{1}};
\draw(x^1_\i)node[left]{\footnotesize{1}}; 
}
\draw(x^0_0) node[above left] {\footnotesize{3}};
\draw(x^0_1) node[above left] {\footnotesize{1}};
\draw(x^0_2) node[above left] {\footnotesize{2}};
\draw(x^2_3)node[above left]{\footnotesize{1}};
\draw(x^0_3)node[above left]{\footnotesize{1}};
\draw(x^2_0) node[above left] {\footnotesize{2}};
\draw(x^2_1) node[above right] {\footnotesize{1}};
\draw(x^2_2) node[above left] {\footnotesize{2}};
\end{tikzpicture}
}
\end{subfigure}
%
% oznake
%
\hspace*{1cm}
\begin{subfigure}[b]{0.3\textwidth}
        \centering
        \resizebox{\linewidth}{!}{
\begin{tikzpicture}%[scale=0.95, style=thick]
\def\vr{3pt}
\def\len{1}
\foreach \i in {0, 1, 2, 3}{
\foreach \j in {0, 1, 2, 3}{
\coordinate(x^\j_\i) at (\i, \j);}}
\draw (x^0_0) -- (x^0_3);
\draw (x^0_0) -- (x^3_0);
\draw (x^0_0) -- (x^3_3);
\draw (x^3_1) -- (x^2_0) -- (x^2_1);
\draw (x^3_2) -- (x^2_2) -- (x^2_3);
\draw (x^1_2) -- (x^0_2) -- (x^1_3);
\foreach \i in {0, 1, 2, 3}{
\foreach \j in {0, 1, 2, 3}{
\draw(x^\j_\i)[fill=white] circle(\vr);}   }
\draw(x^0_0)node[above left]{\footnotesize{11}}; \draw(x^0_1)node[above left]{\footnotesize{12}}; \draw(x^0_2)node[above left] {\footnotesize{21}}; \draw(x^0_3)node[above left] {\footnotesize{22}};
\draw(x^1_0)node[left]{\footnotesize{10}}; \draw(x^1_1)node[left]{\footnotesize{13}}; \draw(x^1_2)node[left] {\footnotesize{20}}; \draw(x^1_3)node[left] {\footnotesize{23}};
\draw(x^2_0)node[above left]{\footnotesize{01}}; \draw(x^2_1)node[above left]{\footnotesize{02}}; \draw(x^2_2)node[above left] {\footnotesize{31}}; \draw(x^2_3)node[above left] {\footnotesize{32}};
\draw(x^3_0)node[left]{\footnotesize{00}}; \draw(x^3_1)node[left]{\footnotesize{03}}; \draw(x^3_2)node[left] {\footnotesize{30}}; \draw(x^3_3)node[left] {\footnotesize{33}};
\end{tikzpicture}}
\end{subfigure}
\caption{The $3$-packing coloring of $S^2_{K_{1,3}}$ and the notation of its vertices.}
\label{fig:zvezda}
\end{figure}

Note that the coloring shown in Figure~\ref{fig:zvezda} is a 3-packing coloring of $S^2_{K_{1, 3}}$. We next show that this coloring can be used for a $3$-packing coloring of $S^n_{K_{1, 3}}$, for any $n \geq 2$. For each $w\in [4]_0^{n-2}$ color the vertices of subgraph $\underline{w}S^2_{K_{1, 3}}$ of $S^n_{K_{1, 3}}$ by using the $3$-packing coloring of $S^2_{K_{1, 3}}$ as shown in the figure.

In $S^2_{K_{1, 3}}$ the extreme vertices colored by color $1$ are $00$, $22$ and $33$. 
Since the vertices $0$, $2$ and $3$ are pairwise non-adjacent in $K_{1,3}$ by the definition of the edges of generalized Sierpi\'nski graphs, the vertices $\underline{u}00$, $\underline{v}22$ and $\underline{w}33$, where $\underline{u}, \underline{v}, \underline{w} \in [4]_0^{n-2}$, are also pairwise non-adjacent in $S^n_{K_{1, 3}}$, $n \geq 2$. Therefore any two vertices of $S^n_{K_{1, 3}}$, $n \geq 2$, colored by $1$, are at distance at least $2$. 
Each vertex of $S^n_{K_{1, 3}}$, $n \geq 2$, colored by color $2$, is at distance at least $1$ from any extreme vertex of any subgraph $\underline{w}S^2_{K_{1, 3}}$, where $\underline{w} \in [4]_0^{n-2}$, and this means, that any two vertices of $S^n_{K_{1, 3}}$, $n \geq 2$, colored $2$, are at distance at least $3$.
Since there is only one vertex of $S^2_{K_{1, 3}}$, colored $3$, and the diameter of this graph is $3$, any two vertices of a graph $S^n_{K_{1, 3}}$, $n \geq 2$, colored by $3$, are at distance at least $4$. 
Hence the described coloring is a $3$-packing coloring of the graph $S^n_{K_{1, 3}}$ for any $n \geq 2$, and so $\chi_{\rho}(S^n_{K_{1, 3}}) = 3$.
\qed
\end{proof}
%
%
%
% lemma, ki jo potebujem za pak. barvanje grafa $S(3, K_4-e)
%

The following lemma will be used for determining $\pch(S^3_{K_4-e})$. We use the notation of vertices of $K_4-e$, in the set $[3]_0$, such that $0$ and $2$ are not adjacent. This is then also reflected in the notation of vertices of $S^n_{K_4-e}$. 
 
\begin{lemma}
For any packing coloring of the subgraph of $S^3_{K_4-e}$, which is induced by the set of vertices $V(3S^2_{K_4-e}) \cup V(03S^1_{K_4-e}) \cup V(23S^1_{K_4-e})$ $($respectively $V(1S^2_{K_4-e}) \cup V(01S^1_{K_4-e}) \cup V(21S^1_{K_4-e}))$, at least $7$ colors is required. Moreover, if color $7$ is not used in $3S^2_{K_4-e}$ $($respectively $1S^2_{K_4-e})$, then it is needed in both, in $03S^1_{K_4-e}$ and also in $23S^1_{K_4-e}$ $($respectively $01S^1_{K_4-e}$ and $21S^1_{K_4-e})$.
\label{7pakirno}
\end{lemma}

%%%%%%%%% NOVA SLIKA 
%
\begin{figure}[htb!]
\begin{center}
\begin{tikzpicture}%[scale=0.95, style=thick]
\def\vr{3pt}
\def\len{1}
\foreach \i in {0, 1, 2, 3, 4, 5, 6, 7}{
\coordinate(x_\i) at (\i, 0);
\coordinate(y_\i) at (\i, 1);
\coordinate(z_\i) at (\i, 2);
\coordinate(w_\i) at (\i, 3);
\coordinate(u_\i) at (\i, 4);
\coordinate(v_\i) at (\i, 5);
\coordinate(o_\i) at (\i, 6);
\coordinate(p_\i) at (\i, 7);
}
\foreach \i in {0, 1, 2, 3, 4, 5, 6, 7}{
\ifthenelse{\i < 2}{\draw[thick, gray](u_\i) -- (v_\i);}{}
\ifthenelse{\i >3}{\draw[thick, gray](u_\i) -- (v_\i);}{}
\ifthenelse{\i =2}{\draw(u_\i) -- (v_\i);}{}
\ifthenelse{\i =3}{\draw(u_\i) -- (v_\i);}{}
\ifthenelse{\i < 2}{\draw(o_\i) -- (p_\i);}{\draw[thick, gray](o_\i) -- (p_\i);}
\ifthenelse{\i < 6}{\draw[thick, gray](x_\i) -- (y_\i);}{\draw(x_\i) -- (y_\i);}
\ifthenelse{\i < 4}{\draw[thick, gray](z_\i) -- (w_\i);}{}
\ifthenelse{\i > 5}{\draw[thick, gray](z_\i) -- (w_\i);}{}
\ifthenelse{\i=4}{\draw(z_\i) -- (w_\i);}{}
\ifthenelse{\i=5}{\draw(z_\i) -- (w_\i);}{}
}
\draw[thick, gray](y_0) -- (z_0); \draw[thick, gray](y_3) -- (z_3); \draw (y_4) -- (z_4);\draw (y_7) -- (z_7);
\draw (v_0) -- (o_0); \draw (v_3) -- (o_3); \draw[thick, gray](v_4) -- (o_4); \draw[thick, gray](v_7) -- (o_7);
\draw[thick, gray](w_0) -- (u_0); \draw[thick, gray](w_7) -- (u_7);
\draw[thick, gray](x_0) -- (x_1); \draw[thick, gray](x_1) -- (x_2); \draw[thick, gray](x_2) -- (x_3); \draw[thick, gray](x_3) -- (x_4); \draw[thick, gray](x_4) -- (x_5); \draw (x_5) -- (x_6); \draw (x_6) -- (x_7);
\draw (p_0) -- (p_1); \draw (p_1) -- (p_2); \draw[thick, gray](p_2) -- (p_3); \draw[thick, gray](p_3) -- (p_4); \draw[thick, gray](p_4) -- (p_5); \draw[thick, gray](p_5) -- (p_6); \draw[thick, gray](p_6) -- (p_7);
\draw[thick, gray](z_0) -- (z_1); \draw[thick, gray](z_2) -- (z_3); \draw (z_4) -- (z_5); \draw[thick, gray](z_6) -- (z_7);
\draw[thick, gray](y_0) -- (y_1);\draw[thick, gray](y_2) -- (y_3); \draw[thick, gray](y_4) -- (y_5); \draw (y_6) -- (y_7);
\draw[thick, gray](v_0) -- (v_1); \draw (v_2) -- (v_3); \draw[thick, gray](v_4) -- (v_5); \draw[thick, gray](v_6) -- (v_7);
\draw[thick, gray](u_0) -- (u_1); \draw (u_1) -- (u_2); \draw (u_2) -- (u_3); \draw[thick, gray](u_4) -- (u_5); \draw[thick, gray](u_5) -- (u_6); \draw[thick, gray](u_6) -- (u_7);
\draw[thick, gray](w_0) -- (w_1); \draw[thick, gray](w_1) -- (w_2); \draw[thick, gray](w_2) -- (w_3); \draw (w_4) -- (w_5); \draw (w_5) -- (w_6); \draw[thick, gray](w_6) -- (w_7);
\draw (o_0) -- (o_1); \draw[thick, gray](o_2) -- (o_3); \draw[thick, gray](o_4) -- (o_5); \draw[thick, gray](o_6) -- (o_7);
%
%
% Poševne 
\draw[thick, gray](x_0) -- (y_1);
\draw[thick, gray](x_2) -- (y_3);
\draw[thick, gray](x_4) -- (y_5);
\draw (x_6) -- (y_7);
\draw[thick, gray](z_0) -- (w_1);
\draw[thick, gray](z_2) -- (w_3);
\draw (z_4) -- (w_5);
\draw[thick, gray](z_6) -- (w_7);
\draw[thick, gray](u_0) -- (v_1);
\draw (u_2) -- (v_3);
\draw[thick, gray](u_4) -- (v_5);
\draw[thick, gray](u_6) -- (v_7);
\draw (o_0) -- (p_1);
\draw[thick, gray](o_2) -- (p_3);
\draw[thick, gray](o_4) -- (p_5);
\draw[thick, gray](o_6) -- (p_7);
\draw[thick, gray](y_1) -- (z_2);
\draw (v_1) -- (o_2);
\draw (w_3) -- (u_4);
\draw (y_5) -- (z_6);
\draw[thick, gray](v_5) -- (o_6); 
\foreach \i in {0, 1, 2, 3, 4, 5, 6, 7}{
\ifthenelse{\i < 6}{\draw(x_\i)[fill=gray] circle(\vr);}{\draw(x_\i)[fill=white] circle(\vr);}
\ifthenelse{\i < 6}{\draw(y_\i)[fill=gray] circle(\vr);}{\draw(y_\i)[fill=white] circle(\vr);}
\ifthenelse{\i < 4}{\draw(z_\i)[fill=gray] circle(\vr);}{}
\ifthenelse{\i > 5}{\draw(z_\i)[fill=gray] circle(\vr);}{}
\ifthenelse{\i =4}{\draw(z_\i)[fill=white] circle(\vr);}{}
\ifthenelse{\i =5}{\draw(z_\i)[fill=white] circle(\vr);}{}
\ifthenelse{\i < 4}{\draw(w_\i)[fill=gray] circle(\vr);}{}
\ifthenelse{\i > 5}{\draw(w_\i)[fill=gray] circle(\vr);}{}
\ifthenelse{\i =4}{\draw(w_\i)[fill=white] circle(\vr);}{}
\ifthenelse{\i =5}{\draw(w_\i)[fill=white] circle(\vr);}{}
\ifthenelse{\i < 4}{\draw(z_\i)[fill=gray] circle(\vr);}{}
\ifthenelse{\i > 5}{\draw(z_\i)[fill=gray] circle(\vr);}{}
\ifthenelse{\i =4}{\draw(z_\i)[fill=white] circle(\vr);}{}
\ifthenelse{\i =5}{\draw(z_\i)[fill=white] circle(\vr);}{}
\ifthenelse{\i < 2}{\draw(u_\i)[fill=gray] circle(\vr);}{}
\ifthenelse{\i > 3}{\draw(u_\i)[fill=gray] circle(\vr);}{}
\ifthenelse{\i =2}{\draw(u_\i)[fill=white] circle(\vr);}{}
\ifthenelse{\i =3}{\draw(u_\i)[fill=white] circle(\vr);}{}
\ifthenelse{\i < 2}{\draw(v_\i)[fill=gray] circle(\vr);}{}
\ifthenelse{\i > 3}{\draw(v_\i)[fill=gray] circle(\vr);}{}
\ifthenelse{\i =2}{\draw(v_\i)[fill=white] circle(\vr);}{}
\ifthenelse{\i =3}{\draw(v_\i)[fill=white] circle(\vr);}{}
\ifthenelse{\i > 1}{\draw(o_\i)[fill=gray] circle(\vr);}{\draw(o_\i)[fill=white] circle(\vr);}
\ifthenelse{\i > 1}{\draw(p_\i)[fill=gray] circle(\vr);}{\draw(p_\i)[fill=white] circle(\vr);}
}
%
% Oznake - najprej le 3S^2
\draw(u_4)node[left]{\footnotesize{311}}; \draw(u_5)node[below]{\footnotesize{312}}; \draw(u_6)node[below]{\footnotesize{321}}; \draw(u_7)node[right] {\footnotesize{322}};
\draw(v_4)node[left]{\footnotesize{310}}; \draw(v_5)node[right]{\footnotesize{313}}; \draw(v_6)node[above] {\footnotesize{320}}; \draw(v_7)node[right] {\footnotesize{323}};
\draw(o_4)node[left]{\footnotesize{301}}; \draw(o_5)node[below]{\footnotesize{302}}; \draw(o_6)node[left] {\footnotesize{331}}; \draw(o_7)node[right] {\footnotesize{332}};
\draw(p_4)node[above]{\footnotesize{300}}; \draw(p_5)node[above]{\footnotesize{303}}; \draw(p_6)node[above] {\footnotesize{330}}; \draw(p_7)node[above right] {\footnotesize{333}};
%
%  0S^2
\draw(u_0)node[above left]{\footnotesize{011}}; \draw(u_1)node[above left]{\footnotesize{012}}; 
\draw(u_2)node[above left] {\footnotesize{021}}; \draw(u_3)node[above left] {\footnotesize{022}};
\draw(v_0)node[above left]{\footnotesize{010}}; \draw(v_1)node[above left]{\footnotesize{013}}; 
\draw(v_2)node[above left] {\footnotesize{020}}; \draw(v_3)node[above left] {\footnotesize{023}};
\draw(o_0)node[above left]{\footnotesize{001}}; \draw(o_1)node[above left]{\footnotesize{002}}; 
\draw(o_2)node[above left] {\footnotesize{031}}; \draw(o_3)node[above left] {\footnotesize{032}};
\draw(p_0)node[above left]{\footnotesize{000}}; \draw(p_1)node[above left]{\footnotesize{003}}; 
\draw(p_2)node[above left] {\footnotesize{030}}; \draw(p_3)node[above left] {\footnotesize{033}};
%
% 2S^2
\draw(x_4)node[below right]{\footnotesize{211}}; \draw(x_5)node[below right]{\footnotesize{212}}; 
\draw(x_6)node[below right] {\footnotesize{221}}; \draw(x_7)node[below right] {\footnotesize{222}};
\draw(y_4)node[below right]{\footnotesize{210}}; \draw(y_5)node[below right]{\footnotesize{213}}; 
\draw(y_6)node[below right] {\footnotesize{220}}; \draw(y_7)node[below right] {\footnotesize{223}};
\draw(z_4)node[below right]{\footnotesize{201}}; \draw(z_5)node[below right]{\footnotesize{202}}; 
\draw(z_6)node[below right] {\footnotesize{231}}; \draw(z_7)node[below right] {\footnotesize{232}};
\draw(w_4)node[below right]{\footnotesize{200}}; \draw(w_5)node[below right]{\footnotesize{203}}; 
\draw(w_6)node[below right] {\footnotesize{230}}; \draw(w_7)node[below right] {\footnotesize{233}};
%
% 1S^2
\draw(x_0)node[below left]{\footnotesize{111}}; \draw(x_1)node[below]{\footnotesize{112}}; \draw(x_2)node[below]{\footnotesize{121}}; \draw(x_3)node[below] {\footnotesize{122}};
\draw(y_0)node[left]{\footnotesize{110}}; \draw(y_1)node[right]{\footnotesize{113}}; \draw(y_2)node[above]{\footnotesize{120}}; \draw(y_3)node[right] {\footnotesize{123}};
\draw(z_0)node[left]{\footnotesize{101}}; \draw(z_1)node[below]{\footnotesize{102}}; \draw(z_2)node[left]{\footnotesize{131}}; \draw(z_3)node[right] {\footnotesize{132}};
\draw(w_0)node[left]{\footnotesize{100}}; \draw(w_1)node[above]{\footnotesize{103}}; \draw(w_2)node[above] {\footnotesize{130}}; \draw(w_3)node[right]{\footnotesize{133}};
\draw[thick, orange](3.3, 3.6) -- (7.7, 3.6) -- (7.7, 7.5) -- (3.3, 7.5) -- (3.3, 3.6);
\draw[orange](7.7, 7.5) node[right]{$3S^2_{K_4-e}$};
\draw[thick, orange](1.3, 5.8) -- (3.2, 5.8) -- (3.2, 7.5) -- (1.3, 7.5) -- (1.3, 5.8);
\draw[orange](2, 7.7) node[above]{$03S^1_{K_4-e}$};
\draw[thick, orange](5.8, 1.6) -- (7.7, 1.6) -- (7.7, 3.5) -- (5.8, 3.5) -- (5.8, 1.6);
\draw[orange](7.7, 2) node[right]{$23S^1_{K_4-e}$};
\draw[thick, red](-0.7, -0.4) -- (3.7, -0.4) -- (3.7, 3.5) -- (-0.7, 3.5) -- (-0.7, -0.4);
\draw[red](-0.5, -0.5) node[left]{$1S^2_{K_4-e}$};
\draw[thick, red](-0.7, 3.6) -- (1.2, 3.6) -- (1.2, 5.4) -- (-0.7, 5.4) -- (-0.7, 3.6);
\draw[red](-1.5, 4) node[above]{$01S^1_{K_4-e}$};
\draw[thick, red](3.8, -0.4) -- (5.7, -0.4) -- (5.7, 1.2) -- (3.8, 1.2) -- (3.8, -0.4);
\draw[red](5, -0.5) node[below]{$21S^1_{K_4-e}$};
\end{tikzpicture}
\end{center}
\caption{The graph $S^3_{K_4-e}$ with subgraphs that appear in Lemma~\ref{7pakirno} marked by (red and orange) squares.}
\label{fig:lem}
\end{figure}
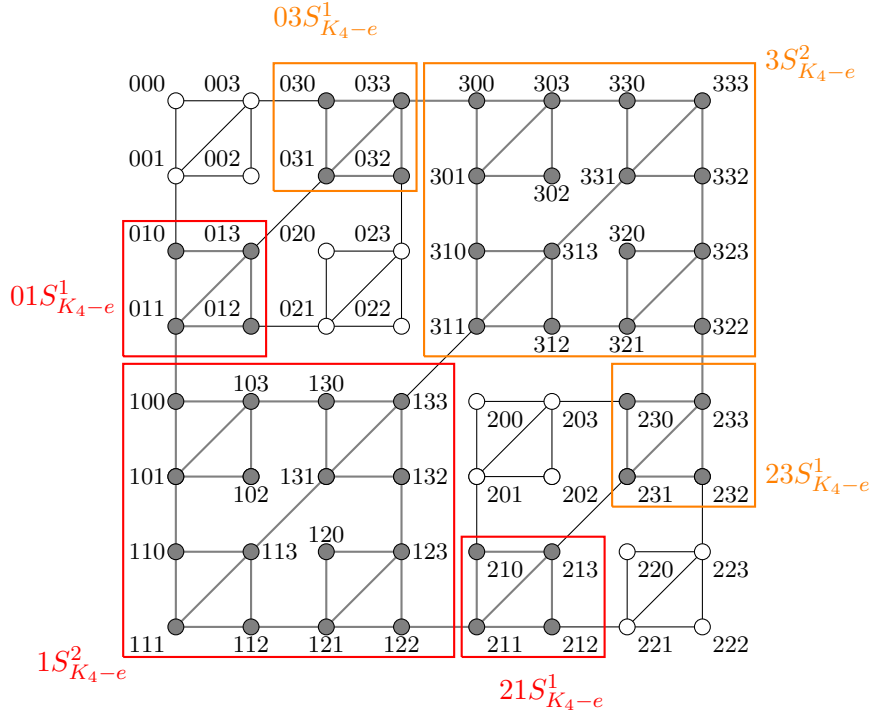

\begin{proof}
Let $c$ be a packing coloring of the subgraph of $S^3_{K_4-e}$, induced by the set of vertices $V(3S^2_{K_4-e}) \cup V(03S^1_{K_4-e}) \cup V(23S^1_{K_4-e})$. See Fig.~\ref{fig:lem} in which the corresponding subgraphs are marked. We may assume that $3S^2_{K_4-e}$ is $6$-packing colorable, for if not, then $c$ uses color $7$ already in $3S^2_{K_4-e}$, and the proof is done. For $c$ restricted to $3S^2_{K_4-e}$ we have: $|c^{-1}(1) \cap V(3S^2_{K_4-e})| \leq 8$, $|c^{-1}(2) \cap V(3S^2_{K_4-e})| \leq 4 $, $|c^{-1}(3) \cap V(3S^2_{K_4-e})| \leq 2 $, $|c^{-1}(4) \cap V(3S^2_{K_4-e})| \leq 2 $, $|c^{-1}(5) \cap V(3S^2_{K_4-e})| \leq 2$ and $|c^{-1}(6) \cap V(3S^2_{K_4-e})| \leq 1$. We distinguish five cases with respect to $|c^{-1}(1) \cap V(3S^2_{K_4-e})|$.

% manj kot 5 enkic
\textbf{Case 1.} $|c^{-1}(1) \cap V(3S^2_{K_4-e})| < 5$. \\
Since $|V(3S^2_{K_4-e})|=16$ and $\sum_{i=1}^6 |c^{-1}(i) \cap V(3S^2_{K_4-e})| \leq 15$, color $7$ is needed in $V(3S^2_{K_4-e})$ and we are done.

% 5 enkic
\textbf{Case 2.} $|c^{-1}(1) \cap V(3S^2_{K_4-e})| = 5$. \\
In this case $\sum_{i=1}^6 |c^{-1}(i) \cap V(3S^2_{K_4-e})| \leq 16$ and hence each of the numbers $|c^{-1}(i) \cap V(3S^2_{K_4-e})|$, $1 \leq i \leq 6$, must reach the established upper bound. 
Therefore in $3S^2_{K_4-e}$ there exist two vertices colored by $5$, and they are $302$ (or $300$) and $320$ (or $322$); we may assume without loss of generality the first option. Then we derive that $c(300)=c(311)=c(322)=c(333)=2$.
Without loss of generality we may assume that $c(301)=c(323)=4$ and $c(321)=3$. But then at most four vertices can receive color $1$, which is a contradiction. 
The described partial packing coloring of $3S^2_{K_4-e}$ and the notation of its vertices is shown in Fig.~\ref{fig:3S^2}.

%%%%%%%%%%% Slika: Sierpinski oznake 3S^2 in barvanje za primer 2 
%
\begin{figure}[h]
\hspace*{3cm}
\begin{subfigure}[b]{0.29\textwidth}
        \centering
        \resizebox{\linewidth}{!}{
\begin{tikzpicture}%[scale=0.95, style=thick]
\def\vr{3pt}
\def\len{1}
\foreach \i in {0, 1, 2, 3}{
\coordinate(x_\i) at (\i, 0);
\coordinate(y_\i) at (\i, 1);
\coordinate(z_\i) at (\i, 2);
\coordinate(w_\i) at (\i, 3); }
\draw (x_0) -- (x_3); 
\draw (w_0) -- (w_3); 
\draw (x_0) -- (w_0);
\draw (x_3) -- (w_3);
\draw (x_0) -- (w_3);
\draw (x_1) -- (y_1) -- (y_0);
\draw (z_3) -- (z_2) -- (w_2);
\draw (z_0) -- (z_1) -- (w_1) -- (z_0);
\draw (x_2) -- (y_3) -- (y_2) -- (x_2);
\foreach \i in {0, 1, 2, 3}{
\draw(x_\i)[fill=white] circle(\vr); 
\draw(y_\i)[fill=white] circle(\vr); 
\draw(z_\i)[fill=white] circle(\vr); 
\draw(w_\i)[fill=white] circle(\vr); 
}
\draw(x_0)node[below]{\footnotesize{2}};
\draw(x_2)node[above left] {\footnotesize{3}};
\draw(x_3)node[below] {\footnotesize{2}};
\draw(y_2)node[left] {\footnotesize{5}};
\draw(y_3)node[above left] {\footnotesize{4}};
\draw(z_0)node[below right] {\footnotesize{4}};
\draw(z_1)node[below] {\footnotesize{5}};
\draw(w_0)node[above] {\footnotesize{2}};
\draw(w_3)node[above] {\footnotesize{2}};
\end{tikzpicture}}
\end{subfigure}
%
% oznake
%
%
\hspace*{1cm}
\begin{subfigure}[b]{0.32\textwidth}
        \centering
        \resizebox{\linewidth}{!}{
\begin{tikzpicture}%[scale=0.95, style=thick]
\def\vr{3pt}
\def\len{1}
\foreach \i in {0, 1, 2, 3}{
\foreach \j in {0, 1, 2, 3}{
\coordinate(x^\j_\i) at (\i, \j);}}
\draw (x^0_0) -- (x^0_3) -- (x^3_3) -- (x^3_0) -- (x^0_0); \draw (x^0_0) -- (x^3_3);
\draw (x^1_0) -- (x^1_1) -- (x^0_1);
\draw (x^1_2) -- (x^1_3) -- (x^0_2) -- (x^1_2);
\draw (x^2_0) -- (x^3_1) -- (x^2_1) -- (x^2_0);
\draw (x^3_2) -- (x^2_2) -- (x^2_3);
\foreach \i in {0, 1, 2, 3}{
\foreach \j in {0, 1, 2, 3}{
\draw(x^\j_\i)[fill=white] circle(\vr);}   }
\draw(x^0_0)node[below]{\footnotesize{311}}; \draw(x^0_1)node[below]{\footnotesize{312}}; \draw(x^0_2)node[below] {\footnotesize{321}}; \draw(x^0_3)node[below] {\footnotesize{322}};
\draw(x^1_0)node[below right]{\footnotesize{310}}; \draw(x^1_1)node[below right]{\footnotesize{313}}; \draw(x^1_2)node[below right] {\footnotesize{320}}; \draw(x^1_3)node[above left] {\footnotesize{323}};
\draw(x^2_0)node[below right]{\footnotesize{301}}; \draw(x^2_1)node[above left]{\footnotesize{302}}; \draw(x^2_2)node[left] {\footnotesize{331}}; \draw(x^2_3)node[above left] {\footnotesize{332}};
\draw(x^3_0)node[above]{\footnotesize{300}}; \draw(x^3_1)node[above]{\footnotesize{303}}; \draw(x^3_2)node[above] {\footnotesize{330}}; \draw(x^3_3)node[above] {\footnotesize{333}};
\end{tikzpicture}}
\end{subfigure}
\caption{The subgraph of $S^3_{K_{4}-e}$ induced by the set of vertices $V(3S^2_{K_4-e})$, its partial packing coloring, and the notation of its vertices on the right-hand side.}
\label{fig:3S^2}
\end{figure}
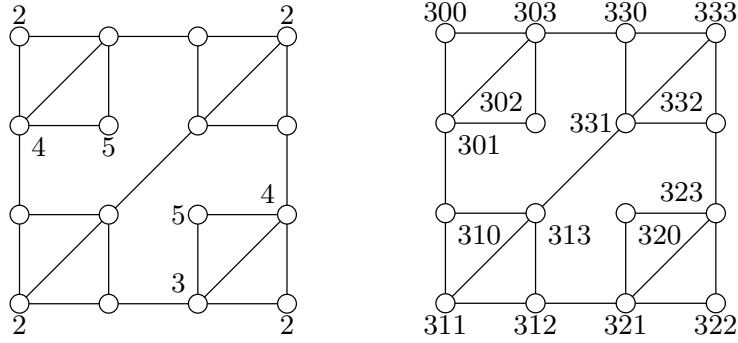

%%%%%%%%%%% 5 malih slikic za proof leme {7pakirno}
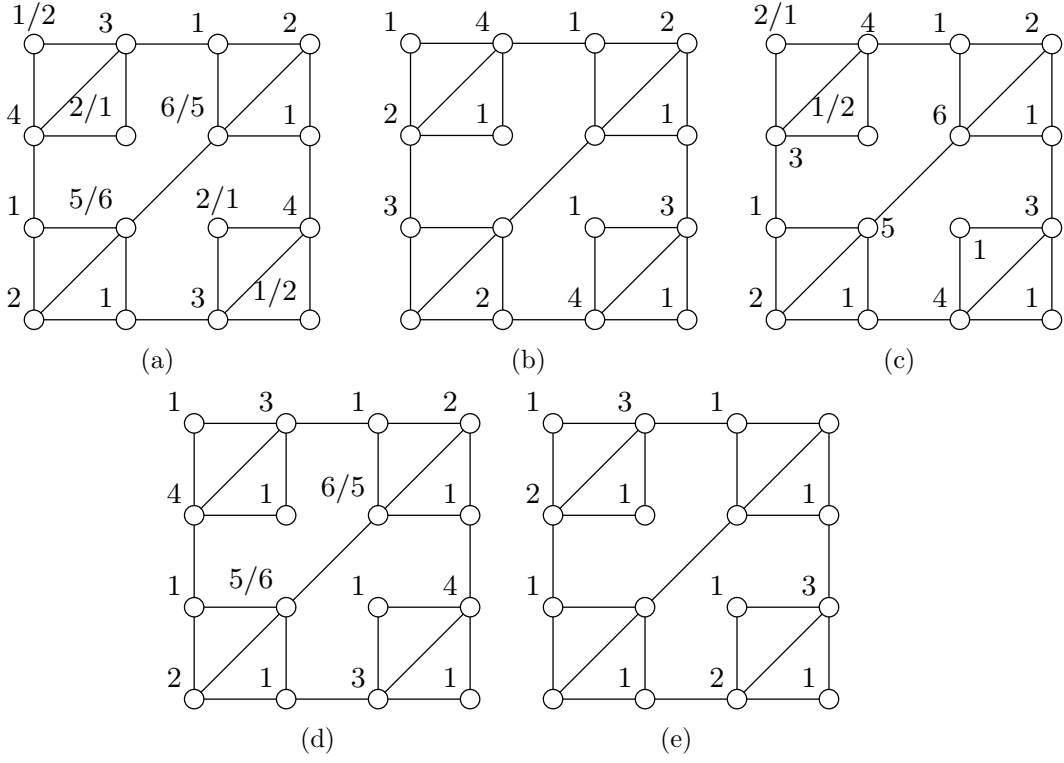
\begin{figure}[h]
\begin{subfigure}[b]{0.32\textwidth}
        \centering
        \resizebox{\linewidth}{!}{
\begin{tikzpicture}%[scale=0.95, style=thick]
\def\vr{3pt}
\def\len{1}
\foreach \i in {0, 1, 2, 3}{
\coordinate(x_\i) at (\i, 0);
\coordinate(y_\i) at (\i, 1);
\coordinate(z_\i) at (\i, 2);
\coordinate(w_\i) at (\i, 3);
}
% vodoravno
\draw (x_0) --(x_1) -- (x_2) -- (x_3);
\draw (w_0) --(w_1) -- (w_2) -- (w_3);
\draw (y_0) -- (y_1);
\draw (y_2) -- (y_3);
\draw (z_0) -- (z_1);
\draw (z_2) -- (z_3);
% navpično
\draw (x_0) -- (y_0) -- (z_0) -- (w_0);
\draw (x_1) -- (y_1);
\draw (x_2) -- (y_2);
\draw (z_1) -- (w_1);
\draw (z_2) -- (w_2);
\draw (x_3) -- (y_3) -- (z_3) -- (w_3);
% poševno
\draw (x_0) -- (y_1) -- (z_2) -- (w_3);
\draw (z_0) -- (w_1);
\draw (x_2) -- (y_3);
% krogeci
\foreach \i in {0, 1, 2, 3}{
\draw(x_\i)[fill=white]circle(\vr); 
\draw(y_\i)[fill=white]circle(\vr); 
\draw(z_\i)[fill=white]circle(\vr); 
\draw(w_\i)[fill=white]circle(\vr); 
}
% barvamo
\draw(x_0)node[above left] {\footnotesize{2}};
\draw(x_1)node[above left] {\footnotesize{1}};
\draw(x_2)node[above left] {\footnotesize{3}};
\draw(x_3)node[above left] {\footnotesize{1/2}};
\draw(y_0)node[above left] {\footnotesize{1}};
\draw(y_1)node[above left] {\footnotesize{5/6}};
\draw(y_2)node[above] {\footnotesize{2/1}};
\draw(y_3)node[above left] {\footnotesize{4}};
\draw(z_0)node[above left] {\footnotesize{4}};
\draw(z_1)node[above left] {\footnotesize{2/1}};
\draw(z_2)node[above left] {\footnotesize{6/5}};
\draw(z_3)node[above left] {\footnotesize{1}};
\draw(w_0)node[above] {\footnotesize{1/2}};
\draw(w_1)node[above left] {\footnotesize{3}};
\draw(w_2)node[above left] {\footnotesize{1}};
\draw(w_3)node[above left] {\footnotesize{2}};
\end{tikzpicture} }
\caption{} 
\label{fig:A}
\end{subfigure}
%
%
% B
%
    \begin{subfigure}[b]{0.31\textwidth}
        \centering
        \resizebox{\linewidth}{!}{			
\begin{tikzpicture}%[scale=0.95, style=thick]
\def\vr{3pt}
\def\len{1}
\foreach \i in {0, 1, 2, 3}{
\coordinate(x_\i) at (\i, 0);
\coordinate(y_\i) at (\i, 1);
\coordinate(z_\i) at (\i, 2);
\coordinate(w_\i) at (\i, 3); }
% vodoravno
\draw (x_0) --(x_1) -- (x_2) -- (x_3);
\draw (w_0) --(w_1) -- (w_2) -- (w_3);
\draw (y_0) -- (y_1);
\draw (y_2) -- (y_3);
\draw (z_0) -- (z_1);
\draw (z_2) -- (z_3);
% navpično
\draw (x_0) -- (y_0) -- (z_0) -- (w_0);
\draw (x_1) -- (y_1);
\draw (x_2) -- (y_2);
\draw (z_1) -- (w_1);
\draw (z_2) -- (w_2);
\draw (x_3) -- (y_3) -- (z_3) -- (w_3);
% poševno
\draw (x_0) -- (y_1) -- (z_2) -- (w_3);
\draw (z_0) -- (w_1);
\draw (x_2) -- (y_3);
% krogeci
\foreach \i in {0, 1, 2, 3}{
\draw(x_\i)[fill=white]circle(\vr); 
\draw(y_\i)[fill=white]circle(\vr); 
\draw(z_\i)[fill=white]circle(\vr); 
\draw(w_\i)[fill=white]circle(\vr); 
}
% barvamo
\draw(x_1)node[above left] {\footnotesize{2}};
\draw(x_2)node[above left] {\footnotesize{4}};
\draw(x_3)node[above left] {\footnotesize{1}};
\draw(y_0)node[above left] {\footnotesize{3}};
\draw(y_2)node[above left] {\footnotesize{1}};
\draw(y_3)node[above left] {\footnotesize{3}};
\draw(z_0)node[above left] {\footnotesize{2}};
\draw(z_1)node[above left] {\footnotesize{1}};
\draw(z_3)node[above left] {\footnotesize{1}};
\draw(w_0)node[above left] {\footnotesize{1}};
\draw(w_1)node[above left] {\footnotesize{4}};
\draw(w_2)node[above left] {\footnotesize{1}};
\draw(w_3)node[above left] {\footnotesize{2}};
\end{tikzpicture}}
\caption{} 
\label{fig:B}
\end{subfigure}
%
% C
%
\begin{subfigure}[b]{0.32\textwidth}
        \centering
        \resizebox{\linewidth}{!}{			
\begin{tikzpicture}%[scale=0.95, style=thick]
\def\vr{3pt}
\def\len{1}
\foreach \i in {0, 1, 2, 3}{
\coordinate(x_\i) at (\i, 0);
\coordinate(y_\i) at (\i, 1);
\coordinate(z_\i) at (\i, 2);
\coordinate(w_\i) at (\i, 3);
}
% vodoravno
\draw (x_0) --(x_1) -- (x_2) -- (x_3);
\draw (w_0) --(w_1) -- (w_2) -- (w_3);
\draw (y_0) -- (y_1);
\draw (y_2) -- (y_3);
\draw (z_0) -- (z_1);
\draw (z_2) -- (z_3);
% navpično
\draw (x_0) -- (y_0) -- (z_0) -- (w_0);
\draw (x_1) -- (y_1);
\draw (x_2) -- (y_2);
\draw (z_1) -- (w_1);
\draw (z_2) -- (w_2);
\draw (x_3) -- (y_3) -- (z_3) -- (w_3);
% poševno
\draw (x_0) -- (y_1) -- (z_2) -- (w_3);
\draw (z_0) -- (w_1);
\draw (x_2) -- (y_3);
% krogeci
\foreach \i in {0, 1, 2, 3}{
\draw(x_\i)[fill=white]circle(\vr); 
\draw(y_\i)[fill=white]circle(\vr); 
\draw(z_\i)[fill=white]circle(\vr); 
\draw(w_\i)[fill=white]circle(\vr); 
}
% barvamo
\draw(x_0)node[above left] {\footnotesize{2}};
\draw(x_1)node[above left] {\footnotesize{1}};
\draw(x_2)node[above left] {\footnotesize{4}};
\draw(x_3)node[above left] {\footnotesize{1}};
\draw(y_0)node[above left] {\footnotesize{1}};
\draw(y_1)node[right] {\footnotesize{5}};
\draw(y_2)node[below right] {\footnotesize{1}};
\draw(y_3)node[above left] {\footnotesize{3}};
\draw(z_0)node[below right] {\footnotesize{3}};
\draw(z_1)node[above left] {\footnotesize{1/2}};
\draw(z_2)node[above left] {\footnotesize{6}};
\draw(z_3)node[above left] {\footnotesize{1}};
\draw(w_0)node[above] {\footnotesize{2/1}};
\draw(w_1)node[above] {\footnotesize{4}};
\draw(w_2)node[above left] {\footnotesize{1}};
\draw(w_3)node[above left] {\footnotesize{2}};
\end{tikzpicture}
}
\caption{} 
\label{fig:C}
\end{subfigure}
\hspace*{2cm}
\begin{subfigure}[b]{0.32\textwidth}
        \centering
        \resizebox{\linewidth}{!}{			
\begin{tikzpicture}%[scale=0.95, style=thick]
\def\vr{3pt}
\def\len{1}
\foreach \i in {0, 1, 2, 3}{
\coordinate(x_\i) at (\i, 0);
\coordinate(y_\i) at (\i, 1);
\coordinate(z_\i) at (\i, 2);
\coordinate(w_\i) at (\i, 3);
}
% vodoravno
\draw (x_0) --(x_1) -- (x_2) -- (x_3);
\draw (w_0) --(w_1) -- (w_2) -- (w_3);
\draw (y_0) -- (y_1);
\draw (y_2) -- (y_3);
\draw (z_0) -- (z_1);
\draw (z_2) -- (z_3);
% navpično
\draw (x_0) -- (y_0) -- (z_0) -- (w_0);
\draw (x_1) -- (y_1);
\draw (x_2) -- (y_2);
\draw (z_1) -- (w_1);
\draw (z_2) -- (w_2);
\draw (x_3) -- (y_3) -- (z_3) -- (w_3);
% poševno
\draw (x_0) -- (y_1) -- (z_2) -- (w_3);
\draw (z_0) -- (w_1);
\draw (x_2) -- (y_3);
% krogeci
\foreach \i in {0, 1, 2, 3}{
\draw(x_\i)[fill=white]circle(\vr); 
\draw(y_\i)[fill=white]circle(\vr); 
\draw(z_\i)[fill=white]circle(\vr); 
\draw(w_\i)[fill=white]circle(\vr); 
}
% barvamo
\draw(x_0)node[above left] {\footnotesize{2}};
\draw(x_1)node[above left] {\footnotesize{1}};
\draw(x_2)node[above left] {\footnotesize{3}};
\draw(x_3)node[above left] {\footnotesize{1}};
\draw(y_0)node[above left] {\footnotesize{1}};
\draw(y_1)node[above left] {\footnotesize{5/6}};
\draw(y_2)node[above left] {\footnotesize{1}};
\draw(y_3)node[above left] {\footnotesize{4}};
\draw(z_0)node[above left] {\footnotesize{4}};
\draw(z_1)node[above left] {\footnotesize{1}};
\draw(z_2)node[above left] {\footnotesize{6/5}};
\draw(z_3)node[above left] {\footnotesize{1}};
\draw(w_0)node[above left] {\footnotesize{1}};
\draw(w_1)node[above left] {\footnotesize{3}};
\draw(w_2)node[above left] {\footnotesize{1}};
\draw(w_3)node[above left] {\footnotesize{2}};
\end{tikzpicture} }
\caption{} 
\label{fig:D}
\end{subfigure}
%
%
% E
\begin{subfigure}[b]{0.32\textwidth}
        \centering
        \resizebox{\linewidth}{!}{			
\begin{tikzpicture}%[scale=0.95, style=thick]
\def\vr{3pt}
\def\len{1}
\foreach \i in {0, 1, 2, 3}{
\coordinate(x_\i) at (\i, 0);
\coordinate(y_\i) at (\i, 1);
\coordinate(z_\i) at (\i, 2);
\coordinate(w_\i) at (\i, 3);
}
% vodoravno
\draw (x_0) --(x_1) -- (x_2) -- (x_3);
\draw (w_0) --(w_1) -- (w_2) -- (w_3);
\draw (y_0) -- (y_1);
\draw (y_2) -- (y_3);
\draw (z_0) -- (z_1);
\draw (z_2) -- (z_3);
% navpično
\draw (x_0) -- (y_0) -- (z_0) -- (w_0);
\draw (x_1) -- (y_1);
\draw (x_2) -- (y_2);
\draw (z_1) -- (w_1);
\draw (z_2) -- (w_2);
\draw (x_3) -- (y_3) -- (z_3) -- (w_3);
% poševno
\draw (x_0) -- (y_1) -- (z_2) -- (w_3);
\draw (z_0) -- (w_1);
\draw (x_2) -- (y_3);
% krogeci
\foreach \i in {0, 1, 2, 3}{
\draw(x_\i)[fill=white]circle(\vr); 
\draw(y_\i)[fill=white]circle(\vr); 
\draw(z_\i)[fill=white]circle(\vr); 
\draw(w_\i)[fill=white]circle(\vr); 
}
% barvamo
\draw(x_1)node[above left] {\footnotesize{1}};
\draw(x_2)node[above left] {\footnotesize{2}};
\draw(x_3)node[above left] {\footnotesize{1}};
\draw(y_0)node[above left] {\footnotesize{1}};
\draw(y_2)node[above left] {\footnotesize{1}};
\draw(y_3)node[above left] {\footnotesize{3}};
\draw(z_0)node[above left] {\footnotesize{2}};
\draw(z_1)node[above left] {\footnotesize{1}};
\draw(z_3)node[above left] {\footnotesize{1}};
\draw(w_0)node[above left] {\footnotesize{1}};
\draw(w_1)node[above left] {\footnotesize{3}};
\draw(w_2)node[above left] {\footnotesize{1}};
\end{tikzpicture} }
\caption{} 
\label{fig:E}
\end{subfigure}
\caption{$6$-packing colorings of $S^2_{K_4-e}$.}
\label{slika}
\end{figure}

% 6 enkic
\textbf{Case 3.} $|c^{-1}(1)\cap V(3S^2_{K_4-e})| = 6 $.  \\
We must color exactly two vertices in (at least) two distinct subgraphs of $3S^2_{K_4-e}$, which are isomorphic to $K_4-e$, by color $1$. 
If $1$ is used twice in $30S^1_{K_4-e}$ (respectively in $32S^1_{K_4-e}$ or in both), then $|c^{-1}(2)\cap V(3S^2_{K_4-e})| \leq 3$ and $|c^{-1}(5)\cap V(3S^2_{K_4-e})| \leq 1$. Thus $\sum_{i=1}^6 |c^{-1}(i) \cap V(3S^2_{K_4-e})| \leq 15$, which implies the claimed result. 
Therefore color $1$ appears exactly once in $30S^1_{K_4-e}$ and also exactly once in $32S^1_{K_4-e}$. Since $|c^{-1}(1)\cap V(3S^2_{K_4-e})| = 6 $, $c(310)=c(312)=c(330)=c(332)=1$ and we have to color exactly two vertices from $\{300, 302, 320, 322\}$ with color $1$.\\
If two of the colors from $\{3, 4, 5\}$ are used by $c$ to color vertices from $\{311$, $313$, $331$, $333\}$, then each of them can be used in $3S^2_{K_4-e}$ only once (and the other at most twice). But then $c$ is not a $6$-packing coloring of $3S^2_{K_4-e}$, since $\sum_{i=1}^6 |c^{-1}(i) \cap V(3S^2_{K_4-e})| \leq 15$. Therefore vertices from $\{311$, $313$, $331$, $333\}$ can be colored by colors $2$, $6$ and one of the colors $3, 4, 5$, which means that two of the listed vertices get color $2$, namely $311$ and $333$. \\
If $|c^{-1}(5) \cap V(3S^2_{K_4-e})| = 2$, then $|c^{-1}(2) \cap V(3S^2_{K_4-e})| = 2$, but then again we get the same conclusion, $\sum_{i=1}^6 |c^{-1}(i) \cap V(3S^2_{K_4-e})| \leq 14$. Hence $|c^{-1}(5) \cap V(3S^2_{K_4-e})| = 1$. Thus $|c^{-1}(4) \cap V(3S^2_{K_4-e})|=2$ and $|c^{-1}(2) \cap V(3S^2_{K_4-e})|=4$. Then vertices $300, 302, 320, 322$ get colors $1$ or $2$ and without loss of generality we may assume $c(301)=c(323)=4$. Therefore $c(303)=c(321)=3$. The remaining two vertices are colored by colors $5$ and $6$. The described coloring, which is clearly a $6$-packing coloring of $3S^2_{K_4-e}$, is shown in Fig.~\ref{fig:A}. \\
Now, consider the coloring $c$ in the remaining part of the subgraph, that is, in $03S^1_{K_4-e}$ and $23S^1_{K_4-e}$. By using the described coloring, vertex $033$ can be colored by one of the colors from $\{1, 2, 7\}$. Suppose that $c(033) \neq 7$. Then one of the vertices from $\{030, 031, 032\}$ cannot be colored by any of the colors from $\{1, 2, 3, 4, 5, 6\}$ and therefore must get color $7$. In either case, color $7$ is used in $03S^1_{K_4-e}$.
By symmetry one of the vertices from $\{230, 231, 232, 233\}$ is colored by $7$. In conclusion, if $|c^{-1}(1)\cap V(3S^2_{K_4-e})| = 6 $, color $7$ is used in both $03S^1_{K_4-e}$ and $23S^1_{K_4-e}$.

% Še za 7 enkic
\textbf{Case 4.} $|c^{-1}(1)\cap V(3S^2_{K_4-e})|=7$. \\
Since $3S^2_{K_4-e}$ contains $4$ distinct copies of subgraphs $K_4-e$, in exactly one of them exist only one vertex colored by color $1$, and the others contain two such vertices. Consider the position of this copy of $K_4-e$ in $3S^2_{K_4-e}$. \\
Assume first that the subgraph isomorphic to $K_4-e$ with only one vertex colored by $1$ is $31S^1_{K_4-e}$. Then $c(300)=c(302)=c(320)=c(322)=c(330)=c(332)=1$, and thus $|c^{-1}(2)\cap V(3S^2_{K_4-e})| \leq 3 $, $|c^{-1}(5)\cap V(3S^2_{K_4-e}) | \leq 1$. Then $\sum_{i=1}^6 |c^{-1}(i) \cap V(3S^2_{K_4-e})| \leq 16$, and hence each of the numbers $|c^{-1}(i) \cap V(3S^2_{K_4-e})|$, $1 \leq i \leq 6$, must reach the established upper bound, otherwise we are done in this case. Without loss of generality we may assume that $c(303)=c(321)=4$. Then $c(323)=3$. 
If $c(301)$ is also $3$, then $|c^{-1}(2)\cap V(3S^2_{K_4-e})| \leq 2$, and in this case color $7$ is required ($\sum_{i=1}^6 |c^{-1}(i) \cap V(3S^2_{K_4-e})| \leq 15$). 
Hence $c(310)=3$, which readily implies $c(301)=c(312)=c(333)=2$. 
The remaining three vertices are colored by colors $1, 5$ and $6$. Clearly, this is a $6$-packing coloring of $3S^2_{K_4-e}$, partially shown in Fig.~\ref{fig:B} (for the last three vertices colors are omitted). 
Vertex $033$ must be colored by color $7$ and vertex $233$ can be colored only by colors $2$ or $7$. If $c(233) \neq 7$, then one of the vertices from $\{230, 231, 232\}$ is colored by $7$.  Again color $7$ appears in both, $03S^1_{K_4-e}$ and $23S^1_{K_4-e}$. Analogous arguments work if there is only one vertex colored by $1$ in $33S^1_{K_4-e}$.\\
For the second case, we may assume without loss of generality that there is only one vertex colored by $1$ in $30S^1_{K_4-e}$. 
Then $c(310)=c(312)=c(320)=c(322)=c(330)=c(332)=1$. Consequently $|c^{-1}(2)\cap V(3S^2_{K_4-e})| \leq 3 $, $|c^{-1}(5)\cap V(3S^2_{K_4-e})| \leq 1$ and thus $\sum_{i=1}^6 |c^{-1}(i) \cap V(3S^2_{K_4-e})| \leq 16$. Then in $3S^2_{K_4-e}$ there are two vertices colored by $4$, and also two colored by $3$. Let $c(321)=4$ and $c(323)=3$ (respectively $c(321)=3$ and $c(323)=4$).
Further, three vertices of $3S^2_{K_4-e}$ must be colored by $2$, two of them are certainly $311$ and $333$. Vertices $300$ and $302$ are colored by colors from $\{1, 2\}$, $c(301)=3$, $c(303)=4$ (respectively $c(301)=4$ and $c(303)=3$), and then $c(313)=5$, $c(331)=6$ (respectively $c(313)=6$, $c(331)=5$). The described coloring, which is clearly a $6$-packing coloring of $3S^2_{K_4-e}$, is the second possible $6$-packing coloring of $3S^2_{K_4-e}$, if $|c^{-1}(1) \cap V(3S^2_{K_4-e})|=7$. It is shown in Fig.~\ref{fig:C}.
But analogously as above we find that if $7$ is not used in $3S^2_{K_4-e}$, it appears in $03S^1_{K_4-e}$ and also in $23S^1_{K_4-e}$.

%  8 enic
%
\textbf{Case 5.} $|c^{-1}(1)\cap V(3S^2_{K_4-e})|=8$. \\
All vertices, which get color $1$, are uniquely determined. In this case the upper bounds for the numbers $|c^{-1}(i) \cap V(3S^2_{K_4-e})|$ clearly decrease from the initial ones in the cases $i=2$ and $i=5$. We thus have $|c^{-1}(2)\cap V(3S^2_{K_4-e})| \leq 3 $ and $|c^{-1}(5)\cap V(3S^2_{K_4-e})| \leq 1$. 
If $|c^{-1}(4)\cap V(3S^2_{K_4-e})| =2 $, then $|c^{-1}(2)\cap V(3S^2_{K_4-e})| \leq 2 $ and $\sum_{i=1}^6 |c^{-1}(i) \cap V(3S^2_{K_4-e})| \leq 16$. Therefore each of the numbers $|c^{-1}(i) \cap V(3S^2_{K_4-e})|$, $1 \leq i \leq 6$, must reach the established upper bound. Without loss of generality we may assume that $c(323)=c(301)=4$. Then $c(303)=c(321)=3$ and $c(311)=c(333)=2$. The other two vertices get colors $5$ and $6$. By the same consideration as in previous cases, it follows that $7$ is used in $03S^1_{K_4-e}$ and also in $23S^1_{K_4-e}$, if it is not used in $3S^2_{K_4-e}$. \\
The second possibility is that $|c^{-1}(4)\cap V(3S^2_{K_4-e})| =1 $. 
Then $|c^{-1}(2)\cap V(3S^2_{K_4-e})| = 3$. 
Since the vertices colored by $1$ are determined, without loss of generality we may assume that $c(301)=c(321)=2$ and then $c(303)=c(323)=3$. The remaining vertices get colors $2, 4, 5$ and $6$. This is $6$-packing coloring of $3S^2_{K_4-e}$, but we can see that both vertices, $033$ and $233$, can get only color $7$. 
Both possibilities for a $6$-packing coloring of $3S^2_{K_4-e}$, if $|c^{-1}(1)\cap V(3S^2_{K_4-e})|=8$, are (partially) shown in Fig.~\ref{fig:D} and~\ref{fig:E}.

In conclusion, if $7$ is not used in $3S^2_{K_4-e}$, it appears in $03S^1_{K_4-e}$ and also in $23S^1_{K_4-e}$.  
Due to symmetry the analogous observation can be proven for $1S^2_{K_4-e}$.
\qed
\end{proof}

%% K_4-e

\begin{theorem}
If $n\geq 1$ and $S^n_{K_4-e}$ is the generalized Sierpi\'nski graph of $K_4-e$ of dimension $n$, then
$$\chi_{\rho}(S^n_{K_4-e}) =
\left\{\begin{array}{ll}
3; & n=1\,,\\
6; & n=2\,,\\
8; & n=3\,.
\end{array}\right.$$
Furthermore, if $n \geq 4$, then $8 \leq \chi_\rho(S^n_{K_4-e}) \leq 11 $. 
\label{izrek_K4-e}
\end{theorem}
\begin{proof} 
The graph $S^1_{K_4-e}$ is isomorphic to $K_4-e$ and it is clear that $\chi_\rho(K_4-e)=3$. 

Since the graph $S^2_{K_4-e}$ contains a subgraph isomorphic to $S^2_3$ (shown in Fig.~\ref{fig:S2K4-e}), we derive by using~\cite[Theorem 1.1]{bkr-2016} that $\chi_\rho(S^2_{K_4-e}) \geq 5$.
Suppose that there exists a $5$-packing coloring $c$ of $S^2_{K_4-e}$. The following inequalities hold for the number of vertices that can be colored by a given color:  $$|c^{-1}(1)| \leq 8,\,\, |c^{-1}(2)| \leq 4,\,\, |c^{-1}(3)| \leq 2 , \,\,|c^{-1}(4)| \leq 2,\,\, |c^{-1}(5)| \leq 2 .$$ 
The equality $|c^{-1}(2)| = 4 $ holds only if two of the vertices $00$, $02$, $20$ and $22$ are colored by $2$, thus we can color at most $11$ vertices by colors $1$ and $2$.

\begin{figure}[h]
\hspace*{3cm}
\begin{subfigure}[b]{0.35\textwidth}
        \centering
        \resizebox{\linewidth}{!}{
\begin{tikzpicture}%[scale=0.95, style=thick]
\def\vr{3pt}
\def\len{1}

\foreach \i in {0, 1, 2, 3}{
\coordinate(x_\i) at (\i, 0);
\coordinate(y_\i) at (\i, 1);
\coordinate(z_\i) at (\i, 2);
\coordinate(w_\i) at (\i, 3); }
% S(2,3)
\draw[orange, thick] (x_0) --(x_1) -- (x_2) -- (x_3);
\draw[orange, thick] (z_2) -- (z_3);
\draw[orange, thick] (x_1) -- (y_1);
\draw[orange, thick] (x_3) -- (y_3) -- (z_3) -- (w_3);
\draw[orange, thick] (x_2) -- (y_3);
\draw[orange, thick] (x_0) -- (y_1) -- (z_2) -- (w_3);
% in ostali del
\draw (x_2) -- (y_2) -- (y_3);
\draw (x_0) -- (y_0) -- (y_1);
\draw (y_0) -- (z_0) -- (w_0) -- (w_1) -- (w_2) -- (w_3);
\draw (z_0) -- (z_1) -- (w_1);
\draw (z_0) -- (w_1);
\draw (z_2) -- (w_2);
% S(2,3)
\foreach \i in {0, 1, 2, 3}{
\draw(x_\i)[fill=orange] circle(\vr); }
\draw(y_1)[fill=orange] circle(\vr);  \draw(y_3)[fill=orange] circle(\vr);  \draw(z_2)[fill=orange] circle(\vr);  \draw(z_3)[fill=orange] circle(\vr);  \draw(w_3)[fill=orange] circle(\vr); 
\draw(y_0)[fill=white] circle(\vr);  \draw(y_2)[fill=white] circle(\vr); \draw(z_0)[fill=white] circle(\vr); \draw(z_1)[fill=white] circle(\vr); \draw(w_0)[fill=white] circle(\vr); \draw(w_1)[fill=white] circle(\vr); \draw(w_2)[fill=white] circle(\vr); 
\foreach \i in {1, 3}{
\draw(x_\i)node[below] {\footnotesize{1}};
\draw(z_\i)node[right] {\footnotesize{1}}; }
\foreach \i in {0, 2}{
\draw(y_\i)node[left] {\footnotesize{1}};
\draw(w_\i)node[above] {\footnotesize{1}}; 
\draw(z_\i)node[left] {\footnotesize{2}};
}
\draw(x_2)node[below] {\footnotesize{2}};
\draw(y_3)node[right] {\footnotesize{3}};
\draw(w_1)node[above] {\footnotesize{3}};
\draw(x_0)node[below] {\footnotesize{6}};
\draw(y_1)node[above] {\footnotesize{5}};
\draw(w_3)node[above] {\footnotesize{4}};
\end{tikzpicture}
}
\end{subfigure}
%
% oznake
%
\begin{subfigure}[b]{0.37\textwidth}
        \centering
        \resizebox{\linewidth}{!}{
\begin{tikzpicture}%[scale=0.95, style=thick]
\def\vr{3pt}
\def\len{1}
\foreach \i in {0, 1, 2, 3}{
\foreach \j in {0, 1, 2, 3}{
\coordinate(x^\j_\i) at (\i, \j);}}
\draw (x^0_0) -- (x^0_3) -- (x^3_3) -- (x^3_0) -- (x^0_0); \draw (x^0_0) -- (x^3_3);
\draw (x^1_0) -- (x^1_1) -- (x^0_1);
\draw (x^1_2) -- (x^1_3) -- (x^0_2) -- (x^1_2);
\draw (x^2_0) -- (x^3_1) -- (x^2_1) -- (x^2_0);
\draw (x^3_2) -- (x^2_2) -- (x^2_3);
\foreach \i in {0, 1, 2, 3}{
\foreach \j in {0, 1, 2, 3}{
\draw(x^\j_\i)[fill=white] circle(\vr);}   }
\draw(x^0_0)node[below]{\footnotesize{11}}; \draw(x^0_1)node[below]{\footnotesize{12}}; \draw(x^0_2)node[below] {\footnotesize{21}}; \draw(x^0_3)node[below] {\footnotesize{22}};
\draw(x^1_0)node[left]{\footnotesize{10}}; \draw(x^1_1)node[above left]{\footnotesize{13}}; \draw(x^1_2)node[left] {\footnotesize{20}}; \draw(x^1_3)node[right] {\footnotesize{23}};
\draw(x^2_0)node[left]{\footnotesize{01}}; \draw(x^2_1)node[above left]{\footnotesize{02}}; \draw(x^2_2)node[left] {\footnotesize{31}}; \draw(x^2_3)node[right] {\footnotesize{32}};
\draw(x^3_0)node[above]{\footnotesize{00}}; \draw(x^3_1)node[above]{\footnotesize{03}}; \draw(x^3_2)node[above] {\footnotesize{30}}; \draw(x^3_3)node[above] {\footnotesize{33}};
\end{tikzpicture}}
\end{subfigure}
\caption{Graph $S^2_{K_4-e}$, its $6$-packing coloring, and the notation of its vertices.}
\label{fig:S2K4-e}
\end{figure}
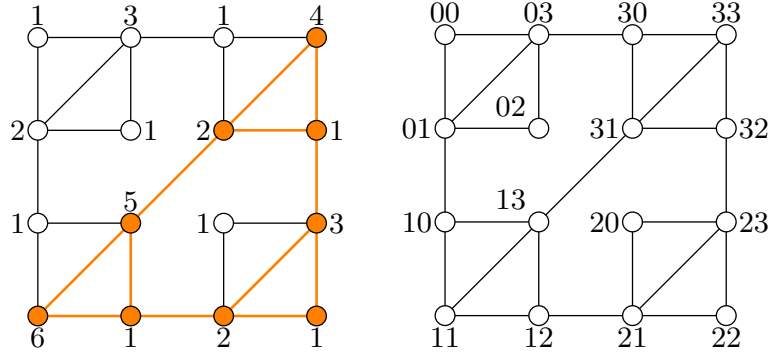

Suppose that $|c^{-1}(5)|=2$. The graph $S^2_{K_4-e}$ contains only four pairs of vertices ($00$ and $20$, $00$ and $22$, $02$ and $20$, $02$ and $22$), which are at distance $6$, and we have to color one pair of them by 5. But then there can be only $6$ vertices (two of them belong to \{$02$, $20$, $00$, $22$\}) colored by $1$. Therefore we can color by $2$ at most $3$ vertices. This is a contradiction, since $S^2_{K_4-e}$ has $16$ vertices, but we can color only $15$ of them. The same result follows, if $|c^{-1}(1)| \leq 5$. 

Alternatively, let $|c^{-1}(5)| = 1$. Hence $|c^{-1}(1)| \geq 7$, which means that at least three of the vertices $00$, $02$, $20$, $22$ get color $1$. Without loss of generality we may assume that $c(02)=c(20)=c(22)=1$.  
Suppose that $|c^{-1}(4)| =2 $. This is possible only if one vertex colored by $4$ is $21$ (or symmetrically $23$), and the other $03$ (respectively $01$). But then only $10$ vertices are colored by colors $1$ and $2$, and we get the same contradiction as above. 
Now, if $|c^{-1}(5)| = 1 $ and  $|c^{-1}(4)| \leq 1 $, we cannot color all $16$ vertices by colors from $\{1, 2, 3, 4, 5 \}$, since only $11$ vertices can be colored by colors $1$ and $2$. 

Therefore $\chi_\rho(S^2_{K_4-e}) \geq 6$. In Fig.~\ref{fig:S2K4-e} a 6-packing coloring of $S^2_{K_4-e}$ is shown, implying $\chi_{\rho}(S^2_{K_4-e})=6$.  

Next, we consider the packing coloring of $S^3_{K_4-e}$. Suppose that $\chi_\rho(S^3_{K_4-e})=7$. By Lemma~\ref{7pakirno} this is a contradiction, since we cannot form even a $7$-packing coloring of the subgraph induced by vertices $V(3S^2_{K_4-e}) \cup V(1S^2_{K_4-e})$ $\cup V(03S^1_{K_4-e}) \cup V(01S^1_{K_4-e}) \cup V(23S^1_{K_4-e}) \cup V(21S^1_{K_4-e})$. Notably, regardless of which vertices are colored by color $7$, they will be too close. Hence $\chi_\rho(S^3_{K_4-e})\geq 8$.

In Fig.~\ref{fig:8pakirno} an $8$-packing coloring of $S^3_{K_4-e}$ is shown. It is easy to verify that for any $i \in [8]$ the distance between any two vertices colored by color $i$ is bigger than $i$, and consequently $\chi_\rho(S^3_{K_4-e})=8$. 

% 8-pakirno barvanje grafa S^3_{K_4-e}
%
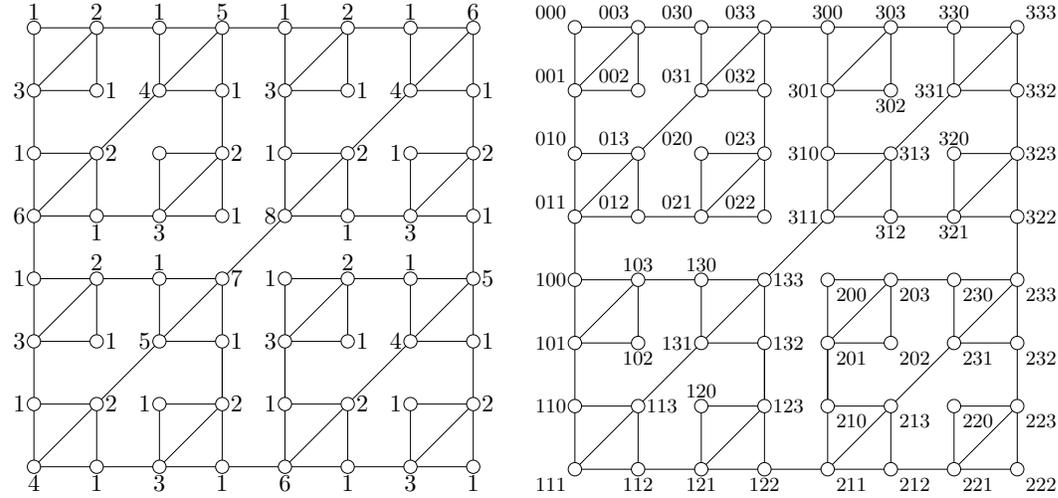
\begin{figure}[htb!]
\begin{subfigure}[b]{0.47\textwidth}
        \centering
        \resizebox{\linewidth}{!}{
\begin{tikzpicture}%[scale=0.95, style=thick]
\def\vr{3pt}
\def\len{1}
\foreach \i in {0, 1, 2, 3, 4, 5, 6, 7}{
\coordinate(x_\i) at (\i, 0);
\coordinate(y_\i) at (\i, 1);
\coordinate(z_\i) at (\i, 2);
\coordinate(w_\i) at (\i, 3);
\coordinate(u_\i) at (\i, 4);
\coordinate(v_\i) at (\i, 5);
\coordinate(o_\i) at (\i, 6);
\coordinate(p_\i) at (\i, 7);
}
\foreach \i in {0, 1, 2, 3, 4, 5, 6, 7}{
\draw (u_\i) -- (v_\i); \draw (o_\i) -- (p_\i); \draw (x_\i) -- (y_\i); \draw (z_\i) -- (w_\i); }
\draw (y_0) -- (z_0); \draw (y_3) -- (z_3); \draw (y_4) -- (z_4);\draw (y_7) -- (z_7);
\draw (v_0) -- (o_0); \draw (v_3) -- (o_3); \draw (v_4) -- (o_4); \draw (v_7) -- (o_7);
\draw (w_0) -- (u_0); \draw (w_7) -- (u_7);
\draw (x_0) -- (x_1); \draw (x_1) -- (x_2); \draw (x_2) -- (x_3); \draw (x_3) -- (x_4); \draw (x_4) -- (x_5); \draw (x_5) -- (x_6); \draw (x_6) -- (x_7);
\draw (p_0) -- (p_1); \draw (p_1) -- (p_2); \draw (p_2) -- (p_3); \draw (p_3) -- (p_4); \draw (p_4) -- (p_5); \draw (p_5) -- (p_6); \draw (p_6) -- (p_7);
\draw (z_0) -- (z_1); \draw (z_2) -- (z_3); \draw (z_4) -- (z_5); \draw (z_6) -- (z_7);
\draw (y_0) -- (y_1);\draw (y_2) -- (y_3); \draw (y_4) -- (y_5); \draw (y_6) -- (y_7);
\draw (v_0) -- (v_1); \draw (v_2) -- (v_3); \draw (v_4) -- (v_5); \draw (v_6) -- (v_7);
\draw (u_0) -- (u_1); \draw (u_1) -- (u_2); \draw (u_2) -- (u_3); \draw (u_4) -- (u_5); \draw (u_5) -- (u_6); \draw (u_6) -- (u_7);
\draw (w_0) -- (w_1); \draw (w_1) -- (w_2); \draw (w_2) -- (w_3); \draw (w_4) -- (w_5); \draw (w_5) -- (w_6); \draw (w_6) -- (w_7);
\draw (o_0) -- (o_1); \draw (o_2) -- (o_3); \draw (o_4) -- (o_5); \draw (o_6) -- (o_7);
\draw(y_3) -- (z_3); \draw(y_4) -- (z_4);
% Poševne 
\draw (x_0) -- (y_1); \draw (x_2) -- (y_3); \draw (x_4) -- (y_5); \draw (x_6) -- (y_7);
\draw (z_0) -- (w_1); \draw (z_2) -- (w_3); \draw (z_4) -- (w_5); \draw (z_6) -- (w_7);
\draw (u_0) -- (v_1); \draw (u_2) -- (v_3); \draw (u_4) -- (v_5); \draw (u_6) -- (v_7);
\draw (o_0) -- (p_1); \draw (o_2) -- (p_3); \draw (o_4) -- (p_5); \draw (o_6) -- (p_7);
\draw (y_1) -- (z_2); \draw (v_1) -- (o_2); \draw (w_3) -- (u_4); \draw (y_5) -- (z_6); \draw (v_5) -- (o_6); 
\foreach \i in {0, 1, 2, 3, 4, 5, 6, 7}{
\draw(x_\i)[fill=white] circle(\vr);
\draw(y_\i)[fill=white] circle(\vr);
\draw(z_\i)[fill=white] circle(\vr);
\draw(w_\i)[fill=white] circle(\vr);
\draw(u_\i)[fill=white] circle(\vr);
\draw(v_\i)[fill=white] circle(\vr);
\draw(o_\i)[fill=white] circle(\vr);
\draw(p_\i)[fill=white] circle(\vr);
}
%
%Barvanje 1
\foreach \i in {1, 3, 5}{
\draw(x_\i)node[below] {1};
\draw(z_\i)node[right] {1};
\draw(o_\i)node[right] {1};
}
\draw(u_1)node[below] {1};
\draw(u_5)node[below] {1};
\draw(u_3)node[right]{1};
\draw(x_7)node[below]{1};
\draw(z_7)node[right]{1};
\draw(u_7)node[right]{1};
\draw(o_7)node[right]{1};
\foreach \i in {0, 2, 4, 6}{
\draw(p_\i)node[above]{1};
\draw(y_\i)node[left]{1};
}
\foreach \i in {2, 6}{
\draw(w_\i)node[above]{1};
}
\draw(v_6)node[left]{1};
\foreach \i in {0, 4}{
\draw(w_\i)node[left]{1};
\draw(v_\i)node[left]{1};
}
%
% Barvanje 2
\foreach \i in {1, 3, 5, 7}{
\draw(y_\i)node[right]{2};
\draw(v_\i)node[right]{2}; }
\foreach \i in {1, 5}{
\draw(w_\i)node[above]{2};
\draw(p_\i)node[above]{2}; }
%
% Barvanje 3
\foreach \i in {0, 4} {
\draw(z_\i)node[left]{3};
\draw(o_\i)node[left]{3}; }
\foreach \i in {2, 6}{
\draw(x_\i)node[below]{3};
\draw(u_\i)node[below]{3}; }
%
% Barvanje 4
\draw(x_0)node[below]{4};
\draw(z_6)node[left]{4};
\draw(o_2)node[left]{4};
\draw(o_6)node[left]{4};
%
% Barvanje 5
\draw(z_2)node[left]{5};
\draw(w_7)node[right]{5};
\draw(p_3)node[above]{5};
%
% Barvanje 6, 7 in 8
\draw(x_4)node[below]{6};
\draw(u_0)node[left]{6};
\draw(p_7)node[above]{6};
\draw(u_4)node[left]{8};
\draw(w_3)node[right]{7};
\end{tikzpicture}
}
\end{subfigure}
%
% In še oznake Sierpinskega
%
\begin{subfigure}[b]{0.5\textwidth}
        \centering
        \resizebox{\linewidth}{!}{
\begin{tikzpicture}%[scale=0.95, style=thick]
\def\vr{3pt}
\def\len{1}
\foreach \i in {0, 1, 2, 3, 4, 5, 6, 7}{
\coordinate(x_\i) at (\i, 0);
\coordinate(y_\i) at (\i, 1);
\coordinate(z_\i) at (\i, 2);
\coordinate(w_\i) at (\i, 3);
\coordinate(u_\i) at (\i, 4);
\coordinate(v_\i) at (\i, 5);
\coordinate(o_\i) at (\i, 6);
\coordinate(p_\i) at (\i, 7);
}
\foreach \i in {0, 1, 2, 3, 4, 5, 6, 7}{
\draw (u_\i) -- (v_\i);
\draw (o_\i) -- (p_\i);
\draw (x_\i) -- (y_\i);
\draw (z_\i) -- (w_\i);
}
\draw (y_0) -- (z_0); \draw (y_3) -- (z_3); \draw (y_4) -- (z_4);\draw (y_7) -- (z_7);
\draw (v_0) -- (o_0); \draw (v_3) -- (o_3); \draw (v_4) -- (o_4); \draw (v_7) -- (o_7);
\draw (w_0) -- (u_0); \draw (w_7) -- (u_7);
\draw (x_0) -- (x_1); \draw (x_1) -- (x_2); \draw (x_2) -- (x_3); \draw (x_3) -- (x_4); \draw (x_4) -- (x_5); \draw (x_5) -- (x_6); \draw (x_6) -- (x_7);
\draw (p_0) -- (p_1); \draw (p_1) -- (p_2); \draw (p_2) -- (p_3); \draw (p_3) -- (p_4); \draw (p_4) -- (p_5); \draw (p_5) -- (p_6); \draw (p_6) -- (p_7);
\draw (z_0) -- (z_1); \draw (z_2) -- (z_3); \draw (z_4) -- (z_5); \draw (z_6) -- (z_7);
\draw (y_0) -- (y_1);\draw (y_2) -- (y_3); \draw (y_4) -- (y_5); \draw (y_6) -- (y_7);
\draw (v_0) -- (v_1); \draw (v_2) -- (v_3); \draw (v_4) -- (v_5); \draw (v_6) -- (v_7);
\draw (u_0) -- (u_1); \draw (u_1) -- (u_2); \draw (u_2) -- (u_3); \draw (u_4) -- (u_5); \draw (u_5) -- (u_6); \draw (u_6) -- (u_7);
\draw (w_0) -- (w_1); \draw (w_1) -- (w_2); \draw (w_2) -- (w_3); \draw (w_4) -- (w_5); \draw (w_5) -- (w_6); \draw (w_6) -- (w_7);
\draw (o_0) -- (o_1); \draw (o_2) -- (o_3); \draw (o_4) -- (o_5); \draw (o_6) -- (o_7);
\draw(y_3) -- (z_3);
\draw(y_4) -- (z_4);
%
% Poševne 
\draw (x_0) -- (y_1);
\draw (x_2) -- (y_3);
\draw (x_4) -- (y_5);
\draw (x_6) -- (y_7);
\draw (z_0) -- (w_1);
\draw (z_2) -- (w_3);
\draw (z_4) -- (w_5);
\draw (z_6) -- (w_7);
\draw (u_0) -- (v_1);
\draw (u_2) -- (v_3);
\draw (u_4) -- (v_5);
\draw (u_6) -- (v_7);
\draw (o_0) -- (p_1);
\draw (o_2) -- (p_3);
\draw (o_4) -- (p_5);
\draw (o_6) -- (p_7);
\draw (y_1) -- (z_2);
\draw (v_1) -- (o_2);
\draw (w_3) -- (u_4);
\draw (y_5) -- (z_6);
\draw (v_5) -- (o_6); 
\foreach \i in {0, 1, 2, 3, 4, 5, 6, 7}{
\draw(x_\i)[fill=white] circle(\vr);
\draw(y_\i)[fill=white] circle(\vr);
\draw(z_\i)[fill=white] circle(\vr);
\draw(w_\i)[fill=white] circle(\vr);
\draw(u_\i)[fill=white] circle(\vr);
\draw(v_\i)[fill=white] circle(\vr);
\draw(o_\i)[fill=white] circle(\vr);
\draw(p_\i)[fill=white] circle(\vr);
}
%
% Oznake - najprej le 3S^2
\draw(u_4)node[left]{\footnotesize{311}}; \draw(u_5)node[below]{\footnotesize{312}}; \draw(u_6)node[below]{\footnotesize{321}}; \draw(u_7)node[right] {\footnotesize{322}};
\draw(v_4)node[left]{\footnotesize{310}}; \draw(v_5)node[right]{\footnotesize{313}}; \draw(v_6)node[above] {\footnotesize{320}}; \draw(v_7)node[right] {\footnotesize{323}};
\draw(o_4)node[left]{\footnotesize{301}}; \draw(o_5)node[below]{\footnotesize{302}}; \draw(o_6)node[left] {\footnotesize{331}}; \draw(o_7)node[right] {\footnotesize{332}};
\draw(p_4)node[above]{\footnotesize{300}}; \draw(p_5)node[above]{\footnotesize{303}}; \draw(p_6)node[above] {\footnotesize{330}}; \draw(p_7)node[above right] {\footnotesize{333}};
%
%  0S^2
\draw(u_0)node[above left]{\footnotesize{011}}; \draw(u_1)node[above left]{\footnotesize{012}}; \draw(u_2)node[above left] {\footnotesize{021}}; \draw(u_3)node[above left] {\footnotesize{022}};
\draw(v_0)node[above left]{\footnotesize{010}}; \draw(v_1)node[above left]{\footnotesize{013}}; \draw(v_2)node[above left] {\footnotesize{020}}; \draw(v_3)node[above left] {\footnotesize{023}};
\draw(o_0)node[above left]{\footnotesize{001}}; \draw(o_1)node[above left]{\footnotesize{002}}; \draw(o_2)node[above left] {\footnotesize{031}}; \draw(o_3)node[above left] {\footnotesize{032}};
\draw(p_0)node[above left]{\footnotesize{000}}; \draw(p_1)node[above left]{\footnotesize{003}}; \draw(p_2)node[above left] {\footnotesize{030}}; \draw(p_3)node[above left] {\footnotesize{033}};
%
% 2S^2
\draw(x_4)node[below right]{\footnotesize{211}}; \draw(x_5)node[below right]{\footnotesize{212}}; \draw(x_6)node[below right] {\footnotesize{221}}; \draw(x_7)node[below right] {\footnotesize{222}};
\draw(y_4)node[below right]{\footnotesize{210}}; \draw(y_5)node[below right]{\footnotesize{213}}; \draw(y_6)node[below right] {\footnotesize{220}}; \draw(y_7)node[below right] {\footnotesize{223}};
\draw(z_4)node[below right]{\footnotesize{201}}; \draw(z_5)node[below right]{\footnotesize{202}}; \draw(z_6)node[below right] {\footnotesize{231}}; \draw(z_7)node[below right] {\footnotesize{232}};
\draw(w_4)node[below right]{\footnotesize{200}}; \draw(w_5)node[below right]{\footnotesize{203}}; \draw(w_6)node[below right] {\footnotesize{230}}; \draw(w_7)node[below right] {\footnotesize{233}};
%
% 1S^2
\draw(x_0)node[below left]{\footnotesize{111}}; \draw(x_1)node[below]{\footnotesize{112}}; \draw(x_2)node[below]{\footnotesize{121}}; \draw(x_3)node[below] {\footnotesize{122}};
\draw(y_0)node[left]{\footnotesize{110}}; \draw(y_1)node[right]{\footnotesize{113}}; \draw(y_2)node[above]{\footnotesize{120}}; \draw(y_3)node[right] {\footnotesize{123}};
\draw(z_0)node[left]{\footnotesize{101}}; \draw(z_1)node[below]{\footnotesize{102}}; \draw(z_2)node[left]{\footnotesize{131}}; \draw(z_3)node[right] {\footnotesize{132}};
\draw(w_0)node[left]{\footnotesize{100}}; \draw(w_1)node[above]{\footnotesize{103}}; \draw(w_2)node[above] {\footnotesize{130}}; \draw(w_3)node[right]{\footnotesize{133}};
\end{tikzpicture}}
\end{subfigure}
\caption{Graph $S^3_{K_4-e}$, its $8$-packing coloring, and the labeling of its vertices.} 
\label{fig:8pakirno}
\end{figure}

In the remainder of the proof, we consider the general case of the packing chromatic number of $S^n_{K_4-e}$, and prove the announced lower and upper bound. 
First, since each $S^n_{K_4-e}$, when $n \geq 4$, contains a subgraph isomorphic to $S^3_{K_4-e}$, we get $\chi_\rho(S^n_{K_4-e}) \geq 8$ for any $n \geq 4$. 
    
We next show that $\chi_\rho(S^n_{K_4-e}) \leq 11$ for any $n \geq 4$. 
We start by presenting a packing coloring of the vertices of $S^n_{K_4-e}$, $n \geq 5$, and denote it by $c$.  
%
% Opišem barvanje

Recall that $S^n_{K_4-e}$, $n \geq 5$, contains $4^{n-4}$ distinct copies of $S^4_{K_4-e}$. Color vertices of each of them as shown in Fig.~\ref{fig:11pakirnoA}. 
Let $\underline{u} \in [4]_0^{n-5}$ and $a \in \{0, 1, 2, 3\}$. 
Then let $c(\underline{u}a1131)=11$ and $c(\underline{u}a3111)=9$, if $a=1$, otherwise $c(\underline{u}a1131)=9$ and $c(\underline{u}a3111)=10$. Further, let $c(\underline{u}a1111)=8$, $c(\underline{u}a1313)=11$ and $c(\underline{u}a3311)=9$, if $a=3$, and otherwise $c(\underline{u}a1111)=c(\underline{u}a3311)=6$, $c(\underline{u}a1313)=8$. Let $c(\underline{u}13131)=10$, $c(\underline{u}33131)=8$ and $c(\underline{u}a3131)=11$, if $a \in \{0, 2\}$. Finally, color each of the unlabeled vertices by color $1$. We claim that in this way we get an $11$-packing coloring of $S^n_{K_4-e}$ for any $n \geq 5$. 

Consider first any two vertices, colored by the same color $i$, $ i \in [11]$, which both belong to the same subgraph of $S^n_{K_4-e}$, isomorphic to $S^4_{K_4-e}$.
It is clear that such two vertices are at distance at least $i+1$ (see Fig.~\ref{fig:11pakirnoA}). We show that the same holds also, if such two vertices belong to two distinct subgraphs of $S^n_{K_4-e}$, which are both isomorphic to $S^4_{K_4-e}$. 

%
% 11-pakirno barvanje grafa S^n_{K_4-e}
% x, y, z, ... ti so spodaj, zgornja dva podgrafa so pa x' y', ... 
\begin{figure}[htb!]
\begin{center}
\begin{tikzpicture}%[scale=0.95, style=thick]
\def\vr{3pt}
\def\len{1}
\foreach \i in {0, 1, 2, 3, 4, 5, 6, 7, 8, 9, 10, 11, 12, 13, 14, 15}{
\coordinate(x_\i) at (\i, 0);
\coordinate(y_\i) at (\i, 1);
\coordinate(z_\i) at (\i, 2);
\coordinate(w_\i) at (\i, 3);
\coordinate(u_\i) at (\i, 4);
\coordinate(v_\i) at (\i, 5);
\coordinate(o_\i) at (\i, 6);
\coordinate(p_\i) at (\i, 7);
\coordinate(x'_\i) at (\i, 8);
\coordinate(y'_\i) at (\i, 9);
\coordinate(z'_\i) at (\i, 10);
\coordinate(w'_\i) at (\i, 11);
\coordinate(u'_\i) at (\i, 12);
\coordinate(v'_\i) at (\i, 13);
\coordinate(o'_\i) at (\i, 14);
\coordinate(p'_\i) at (\i, 15);
}
%
% Navpične
\foreach \i in {0, 1, 2, 3, 4, 5, 6, 7, 8, 9, 10, 11, 12, 13, 14, 15}{
\draw (u_\i) -- (v_\i); \draw (o_\i) -- (p_\i); \draw (x_\i) -- (y_\i); \draw (z_\i) -- (w_\i); 
\draw (u'_\i) -- (v'_\i); \draw (o'_\i) -- (p'_\i); \draw (x'_\i) -- (y'_\i); \draw (z'_\i) -- (w'_\i);
}
\foreach \i in {0, 3, 4, 7, 8, 11, 12, 15}{
\draw (y_\i) -- (z_\i);
\draw (v_\i) -- (o_\i); 
\draw (y'_\i) -- (z'_\i);
\draw (v'_\i) -- (o'_\i); }
\foreach \i in {0, 7, 8, 15}{
\draw (w_\i) -- (u_\i);
\draw (w'_\i) -- (u'_\i); }
\draw (p_0) -- (x'_0);
\draw (p_15) -- (x'_15);
% 
% Vodoravne
\foreach \i in {0, 2, 4, 6, 8, 10, 12, 14}{
\draw (x_\i) -- (x_\the\numexpr\i+1\relax); 
\draw (y_\i) -- (y_\the\numexpr\i+1\relax); 
\draw (z_\i) -- (z_\the\numexpr\i+1\relax); 
\draw (w_\i) -- (w_\the\numexpr\i+1\relax); 
\draw (u_\i) -- (u_\the\numexpr\i+1\relax); 
\draw (v_\i) -- (v_\the\numexpr\i+1\relax); 
\draw (o_\i) -- (o_\the\numexpr\i+1\relax); 
\draw (p_\i) -- (p_\the\numexpr\i+1\relax); 
\draw (x'_\i) -- (x'_\the\numexpr\i+1\relax); 
\draw (y'_\i) -- (y'_\the\numexpr\i+1\relax); 
\draw (z'_\i) -- (z'_\the\numexpr\i+1\relax); 
\draw (w'_\i) -- (w'_\the\numexpr\i+1\relax); 
\draw (u'_\i) -- (u'_\the\numexpr\i+1\relax); 
\draw (v'_\i) -- (v'_\the\numexpr\i+1\relax); 
\draw (o'_\i) -- (o'_\the\numexpr\i+1\relax); 
\draw (p'_\i) -- (p'_\the\numexpr\i+1\relax); 
}
\foreach \i in {1, 5, 9, 13}{
\draw (x_\i) -- (x_\the\numexpr\i+1\relax); 
\draw (w_\i) -- (w_\the\numexpr\i+1\relax); 
\draw (u_\i) -- (u_\the\numexpr\i+1\relax); 
\draw (p_\i) -- (p_\the\numexpr\i+1\relax); 
\draw (x'_\i) -- (x'_\the\numexpr\i+1\relax); 
\draw (w'_\i) -- (w'_\the\numexpr\i+1\relax); 
\draw (u'_\i) -- (u'_\the\numexpr\i+1\relax); 
\draw (p'_\i) -- (p'_\the\numexpr\i+1\relax); 
}
\foreach \i in {3, 11}{
\draw (x_\i) -- (x_\the\numexpr\i+1\relax); 
\draw (p_\i) -- (p_\the\numexpr\i+1\relax); 
\draw (x'_\i) -- (x'_\the\numexpr\i+1\relax); 
\draw (p'_\i) -- (p'_\the\numexpr\i+1\relax); 
}
\draw (x_7) -- (x_8);
\draw (p'_7) -- (p'_8);
%
% Poševne
\foreach \i in {0, 2, 4, 6, 8, 10, 12, 14}{
\draw (x_\i) -- (y_\the\numexpr\i+1\relax); 
\draw (z_\i) -- (w_\the\numexpr\i+1\relax); 
\draw (u_\i) -- (v_\the\numexpr\i+1\relax); 
\draw (o_\i) -- (p_\the\numexpr\i+1\relax); 
\draw (x'_\i) -- (y'_\the\numexpr\i+1\relax); 
\draw (z'_\i) -- (w'_\the\numexpr\i+1\relax); 
\draw (u'_\i) -- (v'_\the\numexpr\i+1\relax); 
\draw (o'_\i) -- (p'_\the\numexpr\i+1\relax); 
}
\foreach \i in {1, 5, 9, 13}{
\draw (y_\i) -- (z_\the\numexpr\i+1\relax); 
\draw (v_\i) -- (o_\the\numexpr\i+1\relax); 
\draw (y'_\i) -- (z'_\the\numexpr\i+1\relax); 
\draw (v'_\i) -- (o'_\the\numexpr\i+1\relax); 
}
\draw(w_3) -- (u_4);
\draw(w_11) -- (u_12);
\draw(w'_3) -- (u'_4);
\draw(w'_11) -- (u'_12);
\draw(p_7) -- (x'_8);
\foreach \i in {0, 1, 2, 3, 4, 5, 6, 7, 8, 9, 10, 11, 12, 13, 14, 15}{
\draw(x_\i)[fill=white] circle(\vr);
\draw(y_\i)[fill=white] circle(\vr);
\draw(z_\i)[fill=white] circle(\vr);
\draw(w_\i)[fill=white] circle(\vr);
\draw(u_\i)[fill=white] circle(\vr);
\draw(v_\i)[fill=white] circle(\vr);
\draw(o_\i)[fill=white] circle(\vr);
\draw(p_\i)[fill=white] circle(\vr);
\draw(x'_\i)[fill=white] circle(\vr);
\draw(y'_\i)[fill=white] circle(\vr);
\draw(z'_\i)[fill=white] circle(\vr);
\draw(w'_\i)[fill=white] circle(\vr);
\draw(u'_\i)[fill=white] circle(\vr);
\draw(v'_\i)[fill=white] circle(\vr);
\draw(o'_\i)[fill=white] circle(\vr);
\draw(p'_\i)[fill=white] circle(\vr);
}
% Barvam
%
% 1-kic ne bomo označili 
%
% Barvanje 2
\foreach \i in {6, 10}{
\draw(x_\i)node[above left]{\footnotesize{2}};
}
\foreach \i in {1, 3, 13, 15}{
\draw(y_\i)node[above left]{\footnotesize{2}};
}
\foreach \i in {4, 6, 8, 10}{
\draw(z_\i)node[above left]{\footnotesize{2}};
}
\foreach \i in {1, 13}{
\draw(w_\i)node[above left]{\footnotesize{2}};
}
\foreach \i in {2, 6, 10}{
\draw(u_\i)node[above left]{\footnotesize{2}};
}
\foreach \i in { 13, 15}{
\draw(v_\i)node[above left]{\footnotesize{2}};
}
\foreach \i in {0, 2, 4, 8, 10}{
\draw(o_\i)node[above left]{\footnotesize{2}};
}
\foreach \i in {7, 13}{
\draw(p_\i)node[above left]{\footnotesize{2}};
}
\foreach \i in {2, 10, 14}{
\draw(x'_\i)node[above left]{\footnotesize{2}};
}
\foreach \i in {5, 7}{
\draw(y'_\i)node[above left]{\footnotesize{2}};
}
\foreach \i in {0, 2, 8, 12, 14}{
\draw(z'_\i)node[above left]{\footnotesize{2}};
}
\foreach \i in {5, 11}{
\draw(w'_\i)node[above left]{\footnotesize{2}};
}
\foreach \i in {10, 14}{
\draw(u'_\i)node[above left]{\footnotesize{2}};
}
\foreach \i in {1, 3, 5, 7}{
\draw(v'_\i)node[above left]{\footnotesize{2}};
}
\foreach \i in {8, 10, 12, 14}{
\draw(o'_\i)node[above left]{\footnotesize{2}};
}
\foreach \i in {1, 5}{
\draw(p'_\i)node[above left]{\footnotesize{2}};
}
%% Barvanje 3
%
\foreach \i in {2, 14}{
\draw(x_\i)node[above left]{\footnotesize{3}};
}
\foreach \i in {7, 11}{
\draw(y_\i)node[above left]{\footnotesize{3}};
}
\foreach \i in {0, 12}{
\draw(z_\i)node[above left]{\footnotesize{3}};
}
\foreach \i in {5, 9}{
\draw(w_\i)node[above left]{\footnotesize{3}};
}
\foreach \i in {3, 7, 11}{
\draw(v_\i)node[above left]{\footnotesize{3}};
}
\foreach \i in {1, 5, 9}{
\draw(p_\i)node[above left]{\footnotesize{3}};
}
\foreach \i in {3, 11, 15}{
\draw(y'_\i)node[above left]{\footnotesize{3}};
}
\foreach \i in {1, 9, 13}{
\draw(w'_\i)node[above left]{\footnotesize{3}};
}
\foreach \i in {2, 6}{
\draw(u'_\i)node[above left]{\footnotesize{3}};
}
\foreach \i in {11, 15}{
\draw(v'_\i)node[above left]{\footnotesize{3}};
}
\foreach \i in {0, 4}{
\draw(o'_\i)node[above left]{\footnotesize{3}};
}
\foreach \i in {9, 13}{
\draw(p'_\i)node[above left]{\footnotesize{3}};
}
\draw(w_14)node[above left]{\footnotesize{3}};
\draw(u_14)node[above left]{\footnotesize{3}};
\draw(o_12)node[above left]{\footnotesize{3}};
\draw(x'_6)node[above left]{\footnotesize{3}};
\draw(z'_4)node[above left]{\footnotesize{3}};
%
% Barvanje 4
\draw(x_4)node[above left]{\footnotesize{4}};
\draw(y_9)node[above left]{\footnotesize{4}};
\draw(u_0)node[above left]{\footnotesize{4}};
\draw(u_4)node[above left]{\footnotesize{4}};
\draw(v_9)node[above left]{\footnotesize{4}};
\draw(o_14)node[above left]{\footnotesize{4}};
\draw(y'_1)node[above left]{\footnotesize{4}};
\draw(y'_9)node[above left]{\footnotesize{4}};
\draw(y'_13)node[above left]{\footnotesize{4}};
\draw(z'_6)node[above left]{\footnotesize{4}};
\draw(v'_9)node[above left]{\footnotesize{4}};
\draw(o'_6)node[above left]{\footnotesize{4}};
\draw(p'_15)node[above left]{\footnotesize{4}};
%
% Barvanje 5
\draw(y_5)node[above left]{\footnotesize{5}};
\draw(z_14)node[above left]{\footnotesize{5}};
\draw(w_11)node[above left]{\footnotesize{5}};
\draw(v_1)node[above left]{\footnotesize{5}};
\draw(o_6)node[above left]{\footnotesize{5}};
\draw(x'_12)node[above left]{\footnotesize{5}};
\draw(w'_3)node[above left]{\footnotesize{5}};
\draw(u'_8)node[above left]{\footnotesize{5}};
\draw(v'_13)node[above left]{\footnotesize{5}};
\draw(o'_2)node[above left]{\footnotesize{5}};
%
% Barvanje 6 
\draw(w_7)node[above left]{\footnotesize{6}};
\draw(p_3)node[above left]{\footnotesize{6}};
\draw(p_15)node[above left]{\footnotesize{6}};
\draw(u_8)node[above left]{\footnotesize{6}};
\draw(x'_4)node[above left]{\footnotesize{6}};
\draw(p'_7)node[above left]{\footnotesize{6}};
%
% Barvanje 7 
\draw(x_8)node[above left]{\footnotesize{7}};
\draw(w_3)node[above left]{\footnotesize{7}};
\draw(x'_0)node[above left]{\footnotesize{7}};
\draw(w'_15)node[above left]{\footnotesize{7}};
\draw(p'_11)node[above left]{\footnotesize{7}};
%
% Barvanje 8
\draw(x_12)node[above left]{\footnotesize{8}};
\draw(p_11)node[above left]{\footnotesize{8}};
\draw(u'_0)node[above left]{\footnotesize{8}};
\draw(w'_7)node[above left]{\footnotesize{8}};
%
% Barvanje 9
\draw(u_12)node[above left]{\footnotesize{9}};
\draw(u'_4)node[above left]{\footnotesize{9}};
%
% Barvanje 10
\draw(w_15)node[above left]{\footnotesize{10}};
\draw(p'_3)node[above left]{\footnotesize{10}};
%
% Označim še tista vozlišča, ki še niso dobila barve.
\draw(x_0)[fill=orange] circle(\vr);
\draw(z_2)[fill=orange] circle(\vr);
\draw(v_5)[fill=orange] circle(\vr);
\draw(x'_8)[fill=orange] circle(\vr);
\draw(z'_10)[fill=orange] circle(\vr);
\draw(u'_12)[fill=orange] circle(\vr);
%
% In ta vozlišča še poimenujem
\draw(x_0)node[below]{\footnotesize{1111}};
\draw(z_2)node[above left]{\footnotesize{1131}};
\draw(v_5)node[above left]{\footnotesize{1313}};
\draw(x'_8)node[above left]{\footnotesize{3111}};
\draw(z'_10)node[above left]{\footnotesize{3131}};
\draw(u'_12)node[above left]{\footnotesize{3311}};
\end{tikzpicture}
\caption{A graph $S^4_{K_4-e}$, and its partial packing coloring} 
\label{fig:11pakirnoA}
\end{center}
\end{figure}
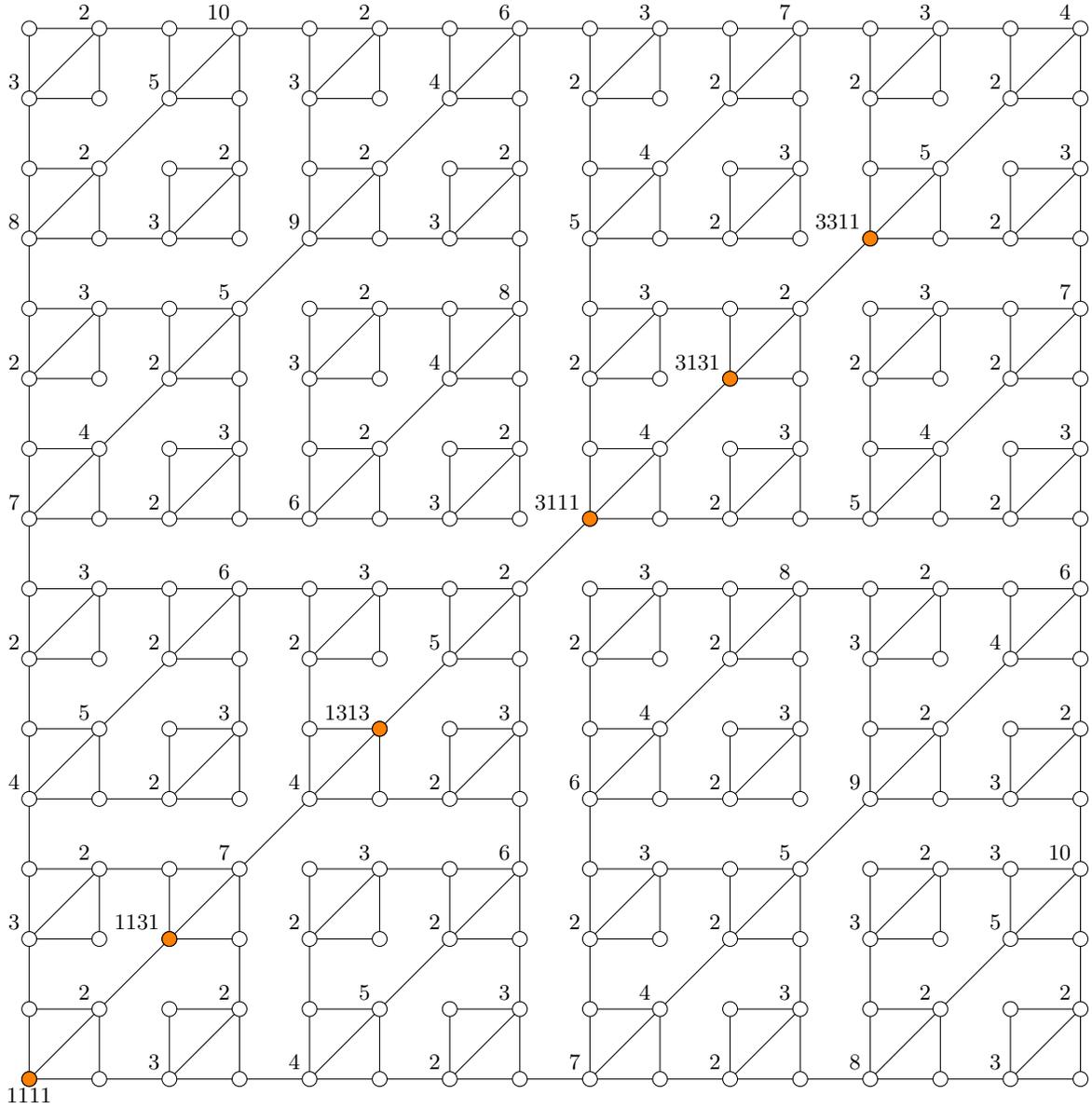
%
% 
% Za 1
%
Consider first the vertices of $S^n_{K_4-e}$, $n \geq 4$, colored by $1$. If such vertex is non-extreme vertex of the subgraph isomorphic $S^4_{K_4-e}$ to which this vertex belongs, then it is at distance at least $2$ from any vertex colored by $1$.
If considered vertex is an extreme vertex of some subgraph isomorphic to $S^4_{K_4-e}$, then it is $\underline{u}0000$ or $\underline{v}2222$, $\underline{u}, \underline{v} \in [4]_0^{n-4}$. But by the definition of edges of generalized Sierpi\'nski graphs, vertices $\underline{u}0000$ and $\underline{v}2222$, for any $\underline{u}, \underline{v} \in [4]^{n-4}_0$, are not adjacent in any $S^n_{K_4-e}$, $n \geq 4$, since $\{0, 2\} \notin E(K_4-e)$. Thus any two  vertices of $S^n_{K_4-e}$, $n \geq 4$, colored by $1$, are at distance at least $2$. 

% Za 2
%
Next, in the packing coloring $c$ of $S^4_{K_4-e}$ none of the extreme vertices is colored by $2$. Consequently, any two vertices, colored by $2$, are at distance at least $3$ in $S^n_{K_4-e}$, $n \geq 4$. 
% Za 3
%
Each copy of $S^4_{K_4-e}$ in $S^n_{K_4-e}$ contains two vertices ($\underline{u}0001$ and $\underline{u}2221$, $\underline{u} \in [4]_0^{n-4}$), which are colored by $3$, belong to some extreme square, and are at distance $1$ from the nearest extreme vertex (i.e. $\underline{u}0000$ respectively $\underline{u}2222$). 
But since vertices $\underline{u}0000$ and $\underline{v}2222$, $\underline{u}, \underline{v} \in [4]_0^{n-4}$, are not adjacent in any graph $S^n_{K_4-e}$, $n \geq 4$, the distance between vertices  $\underline{u}0001$ and $\underline{v}2221$, $ \underline{u}, \underline{v} \in [4]_0^{n-4}$, is at least $17$ in $S^n_{K_4-e}$. 
Any other vertex of any subgraph isomorphic to $S^4_{K_4-e}$, colored by $3$, is at distance at least $2$ from any extreme vertex and hence its distance to any vertex of $S^n_{K_4-e}$ colored by the same color is at least $4$. 

% Za 4
%
If a vertex of $S^n_{K_4-e}$, $n \geq 4$, colored by $4$, is at distance at least $4$ from any extreme vertex of any copy of $S^4_{K_4-e}$, then is at distance at least $5$ from any other vertex, colored also by $4$. 
Since there exists only one vertex of each copy of $S^4_{K_4-e}$ colored by $4$, which is at distance less than $4$ from some extreme vertex of this copy of $S^4_{K_4-e}$, such vertex is not problematic (it is at distance at least $5$ from any vertex with the same color). 
% Za 5
%
In $S^4_{K_4-e}$ whose coloring is shown in Fig.~\ref{fig:11pakirnoA}, there is only one vertex ($3313$) colored by $5$, which is at distance (at most) $2$ from some extreme vertex of $S^4_{K_4-e}$ (other vertices colored by $5$ are at distance at least $3$ from any extreme vertex of $S^4_{K_4-e}$). Hence in each case the sum of two distances between a vertex colored by $5$ and extreme vertex of $S^4_{K_4-e}$
%(considering two different vertices) 
is at least $5$. Therefore any two vertices of $S^n_{K_4-e}$, $n \geq 4$, colored by $5$ are at distance at least $6$. 

% Za 6
%
Consider now the vertices of $S^n_{K_4-e}$, $n \geq 4$, which are colored by $6$ using the packing coloring shown in Fig.~\ref{fig:11pakirnoA}. 
It is clear that each such vertex of any subgraph of $S^n_{K_4-e}$, isomorphic to $S^4_{K_4-e}$, is at distance at least $7$ from any extreme vertex of this copy of $S^4_{K_4-e}$. Therefore any such two vertices of $S^n_{K_4-e}$, $n \geq 5$, which belong to two distinct subgraphs isomorphic to $S^4_{K_4-e}$ and are colored by $6$, are at distance at least $7$.

We still need to check the distances between the mentioned vertices colored by $6$ and those, whose coloring is not shown in Fig.~\ref{fig:11pakirnoA} but nevertheless get color $6$ by the description of the coloring $c$ above, and, in addition, we need to check the distances between these vertices themselves. It is easy to see that the distance between any two vertices, colored by $6$, within any copy of $S^5_{K_4-e}$ is at least $7$. 
Now, in any subgraph $\underline{u}S^5_{K_4-e}$, $\underline{u} \in [4]_0^{n-5}$, each vertex colored by $6$ except $\underline{u}11111$ is at distance at least $7$ from any extreme vertex of $\underline{u}S^5_{K_4-e}$. Hence any two vertices of $S^n_{K_4-e}$, $n \geq 5$, colored by $6$ are at distance at least $7$.

% Za 7
In $S^4_{K_4-e}$, whose coloring is shown in Fig.-\ref{fig:11pakirnoA} only one vertex colored by $7$ is at distance $3$ from the nearest extreme vertex of $S^4_{K_4-e}$. 
Other such vertices are at distances more than $3$ ($4$ or $7$) from any extreme vertex of $S^4_{K_4-e}$. Hence the distance between any two vertices of $S^n_{K_4-e}$, $n \geq 4$, colored by $7$ is at least $8$. 

% Za 8
There are only two vertices of $S^4_{K_4-e}$ colored by $8$ by the use of packing coloring shown in Fig.~\ref{fig:11pakirnoA}, which are at distance (at most) $3$ from the nearest extreme vertex. These two vertices are $0011$ and $2211$, the nearest extreme vertex of the first is $0000$ and of the second is $2222$. 
But since vertices $\underline{u}0000$ and $\underline{v}2222$, $\underline{u}, \underline{v} \in [4]_0^{n-4}$, are not adjacent in any graph $S^n_{K_4-e}$, $n \geq 4$, vertices $\underline{u}0011$ and $\underline{v}2211$ are at distance more than $8$ in any graph $S^n_{K_4-e}$, $n \geq 4$.
It is clear that the distance between any two vertices of $S^n_{K_4-e}$, $n \geq 5$, which are colored by $8$, but their coloring is not shown in Fig.~\ref{fig:11pakirnoA}, is more than $8$. Any of these vertices is also at distance more than $8$ from any vertex colored by $8$. 
Vertices of any subgraph $\underline{u}S^5_{K_4-e}$ of $S^n_{K_4-e}$, $n \geq 5$, $\underline{u} \in [4]_0^{n-5}$, which are colored by $8$, are at distances $3$ ($\underline{u}00011$ and $\underline{u}22211$), $5$ ($\underline{u}11313$ and $\underline{u}33131$), or more than $8$ from any extreme vertex of $\underline{u}S^5_{K_4-e}$. 
Thus the only problematic vertices could be $\underline{u}00011$ and $\underline{v}22211$, $\underline{u}, \underline{v} \in [4]_0^{n-5}$, since the distance between any other two vertices, colored by $8$, is at least $9$ in $S^n_{K_4-e}$. But by above consideration, vertices $\underline{u}00011$ and $\underline{v}22211$ are at distance more than $8$ in any $S^n_{K_4-e}$, $n \geq 5$. 

% Za 9
In a copy of $S^5_{K_4-e}$ all vertices colored by $9$, except vertex, corresponding to $33311$, are at distance at least $7$ from any of the extreme vertices. Thus any such two vertices in $S^n_{K_4-e}$, $n \geq 5$, are at distance at least $12$ (because the distance between them within a copy of $S^5_{K_4-e}$ is at least $12$). 
Vertex $33311$ of $S^5_{K_4-e}$ is at distance $3$ from the nearest extreme vertex, but since the sum of $3$ and any number, which is at least $7$, is more than $9$, vertex $\underline{u}33311$, $\underline{u} \in [4]_0^{n-5}$, is at distance at least $10$ (actually $11$) from any vertex of $S^n_{K_4-e}$, $n \geq 5$, colored by $9$. 

% Za 10
In $S^5_{K_4-e}$ two vertices ($00033$ and $22233$) colored by $10$ are at distance $3$ from the nearest extreme vertex.
Since vertices $\underline{u}00000$ and $\underline{v}22222$, $\underline{u}, \underline{v} \in [4]_0^{n-5}$, are not adjacent for any $n \geq 5$ in a graph $S^n_{K_4-e}$, $n \geq 5$, vertices $\underline{u}00033$ and $\underline{v}22233$ are at distance at least $24$. 
The other vertices of $S^5_{K_4-e}$ colored by $10$ are at distance at least $7$ from any extreme vertex. Hence in $S^n_{K_4-e}$, $n \geq 5$, each of the vertices, corresponding to the mentioned vertices of $S^5_{K_4-e}$, is at distance more than $10$ from any other vertex colored by $10$.

% Za 11
The graph $S^5_{K_4-e}$ contains four vertices colored by $11$, which are from the nearest extreme vertices at distances $2$, $10$ or $15$. Since the sum of any two listed numbers is more than $11$, vertices colored by $11$ are in $S^n_{K_4-e}$, $n \geq 5$, at distance more than $11$.

In conclusion, coloring $c$ is indeed an $11$-packing coloring of $S^n_{K_4-e}$ for any $n \geq 5$. Thus, $\chi_\rho(S^n_{K_4-e}) \leq 11$ for any $n \geq 5$. 
Since $S^5_{K_4-e}$ contains a subgraph isomorphic to $S^4_{K_4-e}$, we also have $\chi_\rho(S^4_{K_4-e}) \leq 11$.  
\qed
\end{proof}

Recall, that the graph obtained by attaching a leaf to one of the vertices of a triangle is called the {\em paw}.

%% Še za graf 'paw'
%
%
%
\begin{corollary}
If $n\geq 1$ and $S^n_{paw}$ is the generalized Sierpi\'nski graph of the paw of dimension $n$, then
$$\chi_{\rho}(S^1_{paw})=3,$$ $$5 \leq \chi_\rho(S^2_{paw}) \leq 6,$$ $$7 \leq \chi_\rho(S^3_{paw}) \leq 8.$$ For each $n \geq 4$, $$7 \leq \chi_\rho(S^n_{paw}) \leq 11.$$ 
\end{corollary}
\begin{proof}
It is clear, that $\chi_{\rho}(S^1_{paw})=3$, since $S^1_{paw}$ is isomorphic to the graph paw. Since for each $n \geq 2$ the graph $S^n_{paw}$ contains a subgraph isomorphic to $S^n_3$, the lower bounds for $\chi_\rho(S^n_{paw})$, $n \geq 2$, follow from \cite{bkr-2016}. The upper bounds for $\chi_\rho(S^n_{paw})$ are determined by Theorem \ref{izrek_K4-e}, since each $S^n_{paw}$, $n \geq 2$, is isomorphic to a subgraph of $S^n_{K_4-e}$.  \qed 
\end{proof}
%

%%%%%%%%%%%%%%%%%%
\section{Sierpi\' nski triangle graphs}
\label{sec:trikotnik}

Sierpi\' nski triangle graphs can be defined in various ways, which originate in the Sierpi\' nski triangle fractal~\cite{hkz-17}. 
Here we use the connection with Sierpi\' nski graphs. Let $n\in \NN_0$. The {\em Sierpi\' nski triangle graph $ST^n_3$ of dimension $n$} is the graph obtained from $S^{n+1}_3$ by contracting all non-clique edges. 

																						Sierpi\' nski triangle graphs can also be constructed recursively in the following way. Firstly $ST^0_3$ is the triangle $K_3$. For any $n\in \NN_0$, assuming that $ST^n_3$ is given, we obtain $ST^{n+1}_3$ from three copies of $ST^n_3$, by identifying three pairs of extreme vertices, two from each copy, in the manner shown in Fig.~\ref{fig:rekurzija}. Note that extreme vertices of Sierpi\' nski triangle graph are vertices of degree 2 and there are exactly three such vertices in each $ST^n_3$. The three copies of the subgraph isomorphic to $ST^n_3$ that lie in $ST^{n+1}_3$ are denoted by $0ST^n_3$, $1ST^n_3$ and $2ST^n_3$, see Fig.~\ref{fig:rekurzija}.

\begin{figure}[h]
\begin{center}
\resizebox{0.3\linewidth}{!}{
\begin{tikzpicture}%[scale=0.95, style=thick]
\def\vr{2pt}
\def\len{1}
\coordinate (x) at (-2, 0);
\coordinate (y) at (0, 0);
\coordinate (z) at (2, 0);
\coordinate (u) at (-1, 2);
\coordinate (v) at (1, 2);
\coordinate (w) at (0, 4);
\coordinate (a) at (-1, 0.8);
\coordinate (b) at (1, 0.8);
\coordinate (c) at (0, 2.8);
\draw (x) -- (z) -- (w) -- (x);
\draw (y) -- (v) -- (u) -- (y);
\draw(x)[fill=white]circle(\vr);
\draw(y)[fill=white]circle(\vr);
\draw(z)[fill=white]circle(\vr);
\draw(u)[fill=white]circle(\vr);
\draw(v)[fill=white]circle(\vr);
\draw(w)[fill=white]circle(\vr);
% označim
\draw(a)node[below]{$1ST^n_3$};
\draw(b)node[below]{$2ST^n_3$};
\draw(c)node[below]{$0ST^n_3$};
\end{tikzpicture}}
\end{center}
\caption{$ST^{n+1}_3$}
\label{fig:rekurzija}
\end{figure}

%% TRIKOTNIK
\begin{theorem}
\label{th:trikotnik}
If $ST^n_3$, $n \geq 0$, is the Sierpi\'nski triangle graph of dimension $n$, then 
$$\chi_{\rho}(ST^n_3) =
\left\{\begin{array}{ll}
3; & n=0\,,\\
4; & n=1\,,\\
8; & n=2\,.
\end{array}\right.$$
Moreover, for each $n \geq 0:\chi_\rho(ST^n_3) \leq 31$.
\end{theorem}

\begin{proof}
It is clear that $\chi_\rho(ST^0_3)=3$, since $ST^0_3$ is isomorphic to $K_3$.

For any packing coloring of $ST^1_3$, color $1$ can used at most three times and each other color at most once, since the diameter of this graph is $2$. Using also the fact that $|V(ST^1_3)|=6$, colors $1, 2, 3,$ and $4$ are necessary. Thus, $\chi_\rho(ST^1_3) \geq 4$. In Fig.~\ref{fig:trikotnik2}  a $4$-packing coloring of $ST^1_3$ is shown, and hence $\chi_\rho(ST^1_3)=4$.
\begin{figure}[htb!]
\begin{center}
\begin{tikzpicture}%[scale=0.95, style=thick]
\def\vr{2pt}
\def\len{1}
\foreach \j in {0, 1, 2}{
\foreach \i in {0,...,\j}{
\coordinate (x^\j_\i) at (-0.5*\j+1+\i, 2-\j);
}}
\draw (x^2_0) -- (x^2_2) -- (x^0_0) -- (x^2_0);
\draw (x^1_0) -- (x^2_1) -- (x^1_1) -- (x^1_0);
% krogci
\foreach \j in {0, 1, 2}{
\foreach \i in {0,...,\j}{
\draw (x^\j_\i)[fill=white]circle(\vr);
} }
% barvam
\draw(x^0_0)node[left]{\footnotesize{1}};
\draw(x^1_0)node[left]{\footnotesize{2}};
\draw(x^1_1)node[right]{\footnotesize{4}};
\draw(x^2_0)node[below]{\footnotesize{1}};
\draw(x^2_1)node[below]{\footnotesize{3}};
\draw(x^2_2)node[below]{\footnotesize{1}};
\end{tikzpicture}
\caption{The $4$-packing coloring of $ST^1_3$} 
\label{fig:trikotnik2}
\end{center}
\end{figure}
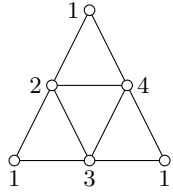

In Fig.~\ref{fig:trikotnik3} an $8$-packing coloring of $ST^2_3$ is presented and thus $\chi_\rho(ST^2_3) \leq 8$. 
Suppose that $\chi_\rho(ST^2_3) = 7$ and denote by $c$ a $7$-packing coloring of $ST^2_3$. 
The following properties of $c$ are easy to observe: $|c^{-1}(1)| \leq 6$, $|c^{-1}(2)| \leq 3$, $|c^{-1}(3)| \leq 3$, $|c^{-1}(4)| \leq 1$, $|c^{-1}(5)| \leq 1$, $|c^{-1}(6)| \leq 1$ and $|c^{-1}(7)| \leq 1$. 

\begin{figure}[htb!]
\begin{center}
\begin{tikzpicture}%[scale=0.95, style=thick]
\def\vr{2pt}
\def\len{1}
\foreach \j in {0, 1, 2, 3, 4}{
\foreach \i in {0,...,\j}{
\coordinate (x^\j_\i) at (2-0.5*\j+\i, 4-\j);
} }
\draw (x^0_0) -- (x^4_0) -- (x^4_4) -- (x^0_0);
\draw (x^2_0) -- (x^4_2) -- (x^2_2) -- (x^2_0);
\draw (x^1_0) -- (x^2_1) -- (x^1_1) -- (x^1_0);
\draw (x^3_0) -- (x^4_1) -- (x^3_1) -- (x^3_0);
\draw (x^3_2) -- (x^4_3) -- (x^3_3) -- (x^3_2);
% krogci
\foreach \j in {0, 1, 2, 3, 4}{
\foreach \i in {0,...,\j}{
\draw (x^\j_\i)[fill=white]circle(\vr); }}
\draw(x^0_0)node[left]{\footnotesize{1}};
\draw(x^1_0)node[left]{\footnotesize{7}};
\draw(x^1_1)node[right]{\footnotesize{3}};
\draw(x^2_0)node[left]{\footnotesize{1}};
\draw(x^2_1)node[below]{\footnotesize{8}};
\draw(x^2_2)node[right]{\footnotesize{1}};
\draw(x^3_0)node[left]{\footnotesize{2}};
\draw(x^3_1)node[right]{\footnotesize{4}};
\draw(x^3_2)node[left]{\footnotesize{5}};
\draw(x^3_3)node[right]{\footnotesize{6}};
\draw(x^4_0)node[below]{\footnotesize{1}};
\draw(x^4_1)node[below]{\footnotesize{3}};
\draw(x^4_2)node[below]{\footnotesize{1}};
\draw(x^4_3)node[below]{\footnotesize{2}};
\draw(x^4_4)node[below]{\footnotesize{1}};
\end{tikzpicture}
\caption{The $8$-packing coloring of $ST^2_3$} 
\label{fig:trikotnik3}
\end{center}
\end{figure}
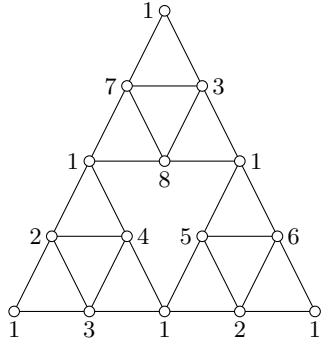

Next, suppose that $|c^{-1}(3)| = 3$. This implies, that all $3$ extreme vertices of $ST^2_3$ (i.e., the vertices of degree 2) get color $3$, because these are the only three vertices of $ST^2_3$, which are pairwise at distance (at least) $4$. 
But then at most four vertices can be colored by color $1$ and this is a contradiction, since $\sum_{i=1}^7{|c^{-1}(i)|} \leq 14$ and $|V(ST^2_3)|=15$.  Hence $|c^{-1}(3)| \leq 2$ and consequently $\sum_{i=1}^7{|c^{-1}(i)|} \leq 15$. This yields, that each number $|c^{-1}(i)|$, $ 1 \leq i \leq 7$, must reach the established upper bound. Hence $|c^{-1}(1)| = 6$ and it is necessary that each of the extreme vertices gets color $1$. But then one cannot have $|c^{-1}(2)|=3$ and at the same time $|c^{-1}(3)|=2$, which implies $\sum_{i=1}^7{|c^{-1}(i)|} \leq 14$, a contradiction. Thus $\chi_\rho(ST^2_3) =8$.

Now, we consider the general case of $ST^n_3$ where $n\ge 4$. Fig.~\ref{fig:trikotnik_barvanje} shows a partial packing coloring of $ST^4_3$. Vertices that are not colored orange and are not labelled receive color 1, while vertices colored orange and have no label receive color $22$ and $23$, which yields $\pch(ST^4_3)\le 23$.

Next we present the packing coloring of the vertices of $ST^5_3$, and in parallel we also consider the general case of $ST^n_3$ for $n\ge 5$.

 %%%%%%%%%%%%% svetlo sivo barvo definiram (1:bela, 0:črna)

% TRIKOTNIK, ST_3(4) in neskončna trikotniška mreža
%
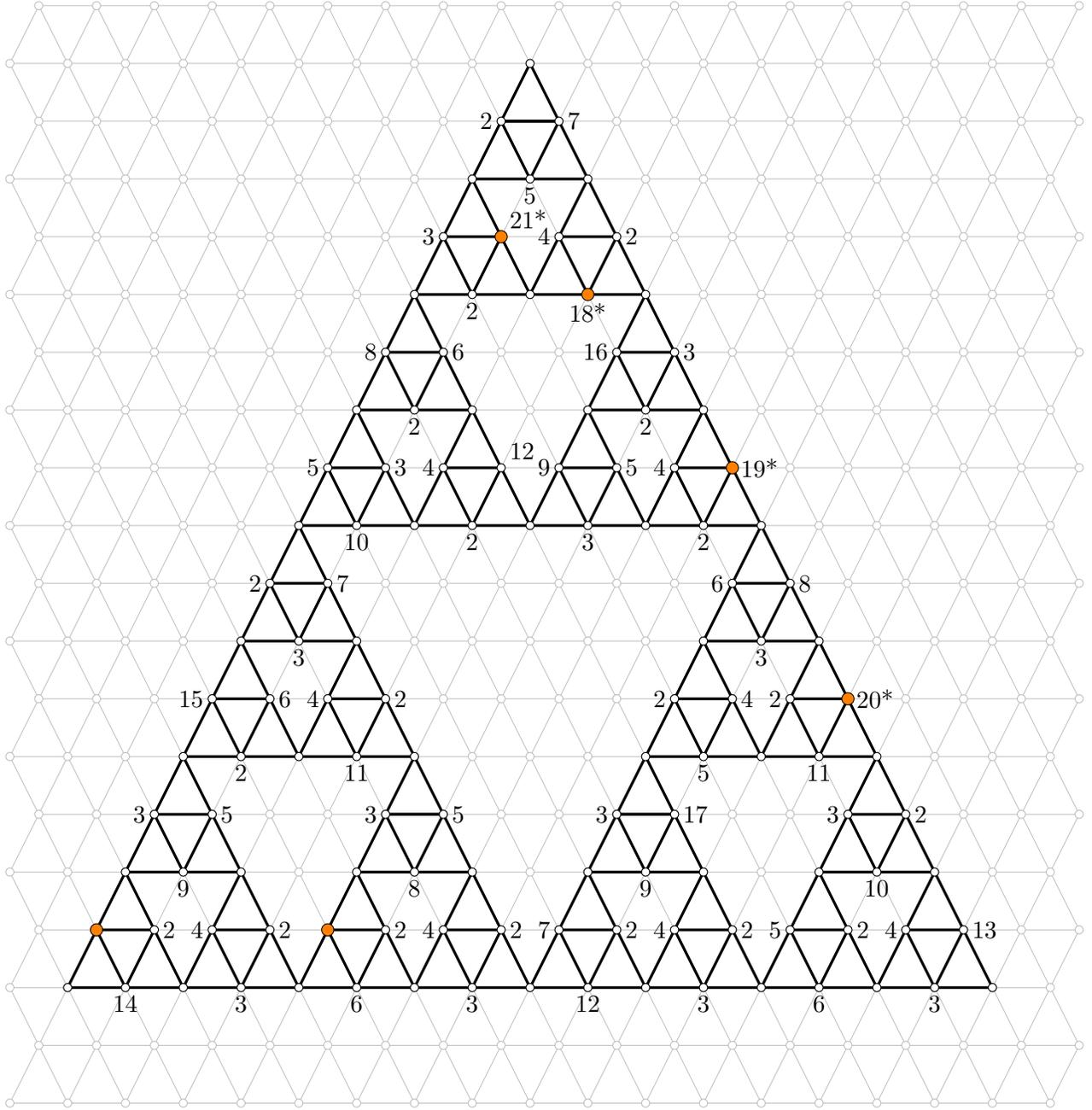
\begin{figure}[htb!]
\begin{center}
\begin{tikzpicture}[scale=0.90] %,style=thick]
\def\vr{2pt}
\def\len{1}
%
% neskončna trikotniška mreža
%
\foreach \l in {-1, 0, 1, 2, 3, 4, 5, 6, 7, 8}{ %višina
\foreach \k in {-1, 0, 1, 2, 3, 4, 5, 6, 7, 8, 9, 10, 11, 12, 13, 14, 15, 16, 17}{
\coordinate (x^\l_\k) at (\k, 2*\l); 
\coordinate (y^\l_\k) at (\k+0.5, 2*\l+1);
 }}
%
% srednje dolge črte
\foreach \k in {-1, 0, 1, ..., 8}{
\draw[very thin, lgray](x^-1_\k) -- (y^8_\the\numexpr\k+9\relax);
\ifthenelse{\k < 8}{\draw[very thin, lgray](x^-1_\the\numexpr\k+10\relax) -- (y^8_\k);}{}
}
% leva stran krajše
\foreach \k in {0, 1, ..., 8}{
\draw[very thin, lgray](x^\k_-1) -- (y^8_\the\numexpr7-\k\relax);
\draw[very thin, lgray](x^-1_\k) -- (x^\k_-1);
\draw[very thin, lgray](x^-1_\the\numexpr\k+9\relax) -- (y^\the\numexpr7-\k\relax_17);
\draw[very thin, lgray](y^\the\numexpr\k-1\relax_17) -- (y^8_\the\numexpr\k+8\relax);
}
\foreach \l in {-1, 0, 1, ..., 8}{
\draw[very thin, lgray](x^\l_-1) -- (x^\l_17);
\draw[very thin, lgray](y^\l_-1) -- (y^\l_17);
\foreach \k in {-1, 0, 1, 2, ..., 17}{
\draw (x^\l_\k)[lgray, fill=white]circle(\vr);
\draw (y^\l_\k)[lgray, fill=white]circle(\vr);
}
}
%
% Najprej trikotnik narišem
% Najprej krogce, ki so še nevidni, ker ne potrebujem vseh
\foreach \j in {0,...,16} { % višina/vrstice, j=0 je zgoraj
\foreach \i in {0,...,\j}{  % dolžina vrstice
\coordinate(x^\j_\i) at (8-0.5*\j+\i, 16-\j);
}
\ifthenelse{\j>0}{\draw[very thick](x^\j_0) -- (x^\j_1); \draw[very thick](x^\j_\the\numexpr\j-1\relax) -- (x^\j_\j); }{} % črne črtice
}
\foreach \j in {0,...,15} { % roza spodnje
\draw[very thick](x^15_\j) -- (x^16_\j);
\draw[very thick](x^15_\j) -- (x^16_\the\numexpr\j+1\relax);
 }
\foreach \j in {4, 6, 8, 10, 12, 14, 16}{  % rdeče črtice
\draw[very thick](x^\j_1) -- (x^\j_2);  % na levi
\draw[very thick](x^\j_\the\numexpr\j-2\relax) -- (x^\j_\the\numexpr\j-1\relax);% na desni
\draw[very thick](x^14_\the\numexpr\j-4\relax) -- (x^15_\the\numexpr\j-4\relax); % roza 2. vrstica od spodaj navzgor
\draw[very thick](x^14_\the\numexpr\j-4\relax) -- (x^15_\the\numexpr\j-3\relax); % roza 2. vrstica, nagnjene v levo
} 
\draw[very thick](x^14_14) -- (x^15_14); % roza 2. vrstica od spodaj navzgor, zadnje
\draw[very thick](x^14_14) -- (x^15_15); % roza 2. vrstica, nagnjene v levo, zadnje
\foreach \j in {7, 8, 11, 12, 15, 16}{  
\draw[very thick](x^\j_2) -- (x^\j_3);  % zelene ze levi
\draw[very thick](x^\j_\the\numexpr\j-3\relax) -- (x^\j_\the\numexpr\j-2\relax); % zelene na desni
\draw[very thick](x^13_\the\numexpr\j-7\relax) -- (x^14_\the\numexpr\j-7\relax); % roza, 3. vrstica od spodaj navzgor
\draw[very thick](x^13_\the\numexpr\j-7\relax) -- (x^14_\the\numexpr\j-6\relax); % roza, 3. vrstica od spodaj navzgor
} 
\draw[very thick](x^13_12) -- (x^14_12); % roza, 3. vrstica od spodaj navzgor, dodatne na koncu
\draw[very thick](x^13_12) -- (x^14_13); % roza, 3. vrstica od spodaj navzgor, na koncu
\draw[very thick](x^13_13) -- (x^14_13); % roza, 3. vrstica od spodaj navzgor, dodatne na koncu
\draw[very thick](x^13_13) -- (x^14_14); % roza, 3. vrstica od spodaj navzgor, na koncu
\foreach \j in {8, 12, 16}{  
\draw[very thick](x^\j_3) -- (x^\j_4);  % še druge zelene črtice na levi
\draw[very thick](x^\j_\the\numexpr\j-4\relax) -- (x^\j_\the\numexpr\j-3\relax); % na desni
\draw[very thick](x^12_\the\numexpr\j-8\relax) -- (x^13_\the\numexpr\j-8\relax); % roza, 4. vrstica od spodaj navzgor
\draw[very thick](x^12_\the\numexpr\j-8\relax) -- (x^13_\the\numexpr\j-7\relax); % roza, 4. vrstica od spodaj navzgor
} 
\draw [very thick](x^12_12) -- (x^13_12); % roza, 4. vrstica od spodaj navzgor, dodatne na koncu
\draw [very thick](x^12_12) -- (x^13_13); % roza, 4. vrstica od spodaj navzgor, na koncu
\foreach \j in {13, 14, 15, 16}{  % modre
\draw[very thick](x^\j_4) -- (x^\j_5);  % na levi
\draw[very thick](x^\j_\the\numexpr\j-9\relax) -- (x^\j_\the\numexpr\j-8\relax); % na levi notranje
\draw[very thick](x^\j_8) -- (x^\j_9); % na desni notranje
\draw[very thick](x^\j_\the\numexpr\j-4\relax) -- (x^\j_\the\numexpr\j-5\relax); % desni
\draw [very thick](x^11_\the\numexpr16-\j\relax) -- (x^12_\the\numexpr16-\j\relax); % roza, 5. vrstica
\draw [very thick](x^11_\the\numexpr24-\j\relax) -- (x^12_\the\numexpr24-\j\relax); % roza, 5. vrstica
\draw [very thick](x^11_\the\numexpr16-\j\relax) -- (x^12_\the\numexpr17-\j\relax); % roza, 5. vrstica
\draw [very thick](x^11_\the\numexpr24-\j\relax) -- (x^12_\the\numexpr25-\j\relax); % roza, 5. vrstica
\draw [very thick](x^7_\the\numexpr16-\j\relax) -- (x^8_\the\numexpr16-\j\relax); % vijolične
\draw [very thick](x^7_\the\numexpr20-\j\relax) -- (x^8_\the\numexpr20-\j\relax); % vijolične
\draw [very thick](x^7_\the\numexpr16-\j\relax) -- (x^8_\the\numexpr17-\j\relax); % vijolične
\draw [very thick](x^7_\the\numexpr20-\j\relax) -- (x^8_\the\numexpr21-\j\relax); % vijolične
\draw [very thick](x^3_\the\numexpr16-\j\relax) -- (x^4_\the\numexpr16-\j\relax); % oranžne zgornje
\draw [very thick](x^3_\the\numexpr16-\j\relax) -- (x^4_\the\numexpr17-\j\relax); % oranžne zgornje, druga smer
\draw [very thick](x^6_\the\numexpr\j*2-26\relax) -- (x^7_\the\numexpr\j*2-26\relax); % oranžne srednji
\draw [very thick](x^6_\the\numexpr\j*2-26\relax) -- (x^7_\the\numexpr\j*2-25\relax); % oranžne srednji
\ifthenelse{\j < 15}{\draw [very thick](x^10_\the\numexpr\j*2-26\relax) -- (x^11_\the\numexpr\j*2-26\relax);}{} % oranžne spodnji
\ifthenelse{\j < 15}{\draw [very thick](x^10_\the\numexpr\j*2-26\relax) -- (x^11_\the\numexpr\j*2-25\relax);}{} % oranžne spodnji
\ifthenelse{\j > 14}{\draw [very thick](x^10_\the\numexpr\j*2-22\relax) -- (x^11_\the\numexpr\j*2-22\relax);}{} % oranžne spodnji
\ifthenelse{\j > 14}{\draw [very thick](x^10_\the\numexpr\j*2-22\relax) -- (x^11_\the\numexpr\j*2-21\relax);}{} % oranžne spodnji
\ifthenelse{\j > 14}{\draw [very thick](x^9_\the\numexpr16-\j\relax) -- (x^10_\the\numexpr16-\j\relax);}{} % rumene
\ifthenelse{\j > 14}{\draw [very thick](x^9_\the\numexpr16-\j\relax) -- (x^10_\the\numexpr17-\j\relax);}{} % rumene
\ifthenelse{\j > 14}{\draw [very thick](x^9_\the\numexpr24-\j\relax) -- (x^10_\the\numexpr24-\j\relax);}{} % rumene
\ifthenelse{\j > 14}{\draw [very thick](x^9_\the\numexpr24-\j\relax) -- (x^10_\the\numexpr25-\j\relax);}{} % rumene
\ifthenelse{\j > 14}{\draw [very thick](x^1_\the\numexpr16-\j\relax) -- (x^2_\the\numexpr16-\j\relax);}{} % rumene
\ifthenelse{\j > 14}{\draw [very thick](x^1_\the\numexpr16-\j\relax) -- (x^2_\the\numexpr17-\j\relax);}{} % rumene
\ifthenelse{\j > 14}{\draw [very thick](x^5_\the\numexpr16-\j\relax) -- (x^6_\the\numexpr16-\j\relax);}{} % rumene
\ifthenelse{\j > 14}{\draw [very thick](x^5_\the\numexpr16-\j\relax) -- (x^6_\the\numexpr17-\j\relax);}{} % rumene
\ifthenelse{\j > 14}{\draw [very thick](x^5_\the\numexpr20-\j\relax) -- (x^6_\the\numexpr20-\j\relax);}{} % rumene
\ifthenelse{\j > 14}{\draw [very thick](x^5_\the\numexpr20-\j\relax) -- (x^6_\the\numexpr21-\j\relax);}{} % rumene
}
\draw [very thick](x^16_5) -- (x^16_6) -- (x^16_7); % dodatne modre
\draw [very thick](x^16_9) -- (x^16_10) -- (x^16_11); % dodatne modre
\draw [very thick](x^1_0) -- (x^0_0) -- (x^1_1); % dodatne rumene
\draw [very thick](x^9_0) -- (x^8_0) -- (x^9_1); % dodatne rumene
\draw [very thick](x^9_8) -- (x^8_8) -- (x^9_9); % dodatne rumene
\draw [very thick](x^3_0) -- (x^2_0) -- (x^3_1); % dodatne rumene
\draw [very thick](x^3_2) -- (x^2_2) -- (x^3_3); % dodatne rumene
\draw [very thick](x^5_0) -- (x^4_0) -- (x^5_1); % dodatne rumene
\draw [very thick](x^5_4) -- (x^4_4) -- (x^5_5); % dodatne rumene

% Še krogci in barvanje
\foreach \j in {0,...,16} {
\draw(x^16_\j)[fill=white]circle(\vr);
\ifthenelse{\j <16}{\draw(x^15_\j)[fill=white]circle(\vr);}{}
\ifthenelse{\j <9}{\draw(x^8_\j)[fill=white]circle(\vr);}{}
\ifthenelse{\j <8}{\draw(x^7_\j)[fill=white]circle(\vr);}{}
\ifthenelse{\j <5}{\draw(x^4_\j)[fill=white]circle(\vr);}{}
\ifthenelse{\j <4}{\draw(x^3_\j)[fill=white]circle(\vr);}{}
\ifthenelse{\j <3}{\draw(x^2_\j)[fill=white]circle(\vr);}{}
\ifthenelse{\j <2}{\draw(x^1_\j)[fill=white]circle(\vr);}{}
\ifthenelse{\j <1}{\draw(x^0_\j)[fill=white]circle(\vr);}{}
\ifthenelse{\j <5}{\draw(x^12_\j)[fill=white]circle(\vr);}{}
\ifthenelse{\j <5}{\draw(x^12_\the\numexpr8+\j\relax)[fill=white]circle(\vr);}{}
\ifthenelse{\j <4}{\draw(x^11_\j)[fill=white]circle(\vr);}{}
\ifthenelse{\j <4}{\draw(x^11_\the\numexpr8+\j\relax)[fill=white]circle(\vr);}{}
}
\foreach \j in {0, 1, 2}{
\draw(x^14_\j)[fill=white]circle(\vr);
\draw(x^14_\the\numexpr4+\j\relax)[fill=white]circle(\vr);
\draw(x^14_\the\numexpr\j+8\relax)[fill=white]circle(\vr);
\draw(x^14_\the\numexpr\j+12\relax)[fill=white]circle(\vr);
\draw(x^10_\j)[fill=white]circle(\vr);
\draw(x^10_\the\numexpr\j+8\relax)[fill=white]circle(\vr);
\draw(x^6_\j)[fill=white]circle(\vr);
\draw(x^6_\the\numexpr4+\j\relax)[fill=white]circle(\vr);
}
\foreach \j in {0, 1}{
\draw(x^13_\j)[fill=white]circle(\vr);
\draw(x^13_\the\numexpr\j+4\relax)[fill=white]circle(\vr);
\draw(x^13_\the\numexpr\j+8\relax)[fill=white]circle(\vr);
\draw(x^13_\the\numexpr\j+12\relax)[fill=white]circle(\vr);
\draw(x^9_\j)[fill=white]circle(\vr);
\draw(x^9_\the\numexpr\j+8\relax)[fill=white]circle(\vr);
\draw(x^5_\j)[fill=white]circle(\vr);
\draw(x^5_\the\numexpr\j+4\relax)[fill=white]circle(\vr);
}
%
% Barvanje
\foreach \j in {0,...,16} {
\ifthenelse{\j <7}{\draw(x^15_\the\numexpr2*\j+1\relax)node[right]{2};}{}
\ifthenelse{\j <2}{\draw(x^8_\the\numexpr4*\j+3\relax)node[below]{2};}{}
\ifthenelse{\j <2}{\draw(x^6_\the\numexpr4*\j+1\relax)node[below]{2};}{}
\ifthenelse{\j <4}{\draw(x^16_\the\numexpr4*\j+3\relax)node[below]{3};}{}
\ifthenelse{\j <4}{\draw(x^13_\the\numexpr4*\j\relax)node[left]{3};}{}
\ifthenelse{\j <2}{\draw(x^10_\the\numexpr8*\j+1\relax)node[below]{3};}{}
\ifthenelse{\j <4}{\draw(x^15_\the\numexpr4*\j+2\relax)node[left]{4};}{}
\ifthenelse{\j <2}{\draw(x^7_\the\numexpr4*\j+2\relax)node[left]{4};}{}
\ifthenelse{\j <2}{\draw(x^12_\the\numexpr8*\j+3\relax)node[below]{11};}{}
}
\draw(x^1_0)node[left]{2};
\draw(x^3_3)node[right]{2};
\draw(x^4_1)node[below]{2};
\draw(x^9_0)node[left]{2};
\draw(x^11_3)node[right]{2};
\draw(x^11_8)node[left]{2};
\draw(x^11_10)node[left]{2};
\draw(x^12_1)node[below]{2};
\draw(x^13_13)node[right]{2};
\draw(x^8_5)node[below]{3};
\draw(x^7_1)node[right]{3};
\draw(x^5_5)node[right]{3};
\draw(x^3_0)node[left]{3};
\draw(x^11_2)node[left]{4};
\draw(x^11_9)node[right]{4};
\draw(x^3_2)node[left]{4};
\draw(x^15_12)node[left]{5};
\draw(x^12_9)node[below]{5};
\draw(x^13_5)node[right]{5};
\draw(x^7_0)node[left]{5};
\draw(x^7_5)node[right]{5};
\draw(x^2_1)node[below]{5};
\draw(x^13_1)node[right]{5};
\draw(x^16_5)node[below]{6};
\draw(x^16_13)node[below]{6};
\draw(x^11_1)node[right]{6};
\draw(x^9_8)node[left]{6};
\draw(x^5_1)node[right]{6};
\draw(x^15_8)node[left]{7};
\draw(x^9_1)node[right]{7};
\draw(x^1_1)node[right]{7};
\draw(x^14_5)node[below]{8};
\draw(x^9_9)node[right]{8};
\draw(x^5_0)node[left]{8};
\draw(x^14_1)node[below]{9};
\draw(x^14_9)node[below]{9};
\draw(x^7_4)node[left]{9};
\draw(x^14_13)node[below]{10};
\draw(x^8_1)node[below]{10};
\draw(x^16_9)node[below]{12};
\draw(x^7_3)node[above right]{12};
\draw(x^15_15)node[right]{13};
\draw(x^16_1)node[below]{14};
\draw(x^11_0)node[left]{15};
\draw(x^5_4)node[left]{16};
\draw(x^13_9)node[right]{17};
\draw(x^3_1)node[above right]{21*};
\draw(x^4_3)node[below]{18*};
\draw(x^7_7)node[right]{19*};
\draw(x^11_11)node[right]{20*};
%
% In še označim nepobarvana vozlišča
\draw(x^15_0)[fill=orange]circle(3pt);
\draw(x^15_4)[fill=orange]circle(3pt);
\draw(x^11_11)[fill=orange]circle(3pt);
\draw(x^7_7)[fill=orange]circle(3pt);
\draw(x^4_3)[fill=orange]circle(3pt);
\draw(x^3_1)[fill=orange]circle(3pt);
\end{tikzpicture}
\caption{Partial packing coloring of the graph $ST^4_3$} 
\label{fig:trikotnik_barvanje}
\end{center}
\end{figure}
   
First, color the vertices of each of the (three) subgraphs of $ST^5_3$ isomorphic to $ST^4_3$ as shown in Fig.~\ref{fig:trikotnik_barvanje}. More precisely, the colors, which are labelled with a star, are used only in the subgraphs $1ST^4_3$ and $2ST^4_3$. Color the remaining vertices of $ST_3(5)$, which correspond to all orange labelled vertices of $ST_3(4)$, with different colors, each by one of the colors from $\{22, 23, \ldots, 31\}$. (Color the others not labelled vertices by color $1$.) 
We claim that in this way we obtain a $31$-packing coloring of $ST^5_3$ (and also of any $ST^n_3$ for $n\ge 5$).

By using the described coloring of the graph $ST^5_3$ (and of $ST^n_3$, which is obtained by just copying $3^{n-5}$ times the colorings of $ST^5_3$), in each subgraph of $ST^5_3$, which is isomorphic to $ST^4_3$, there are exactly six uncolored vertices before applying the colors $18, 19, \ldots , 31$. Since $ST^5_3$ consists of three distinct copies of $ST^4_3$, this means that altogether there are $18$ uncolored vertices in $ST^5_3$. Eight of them are then colored by colors $18, 19, 20, 21$ and the remaining ten vertices by new, distinct colors. Therefore all vertices of the graph $ST^5_3$ (and of $ST^n_3$, $n \geq 5$) are colored by using the described coloring. 

Now, we prove that the described coloring is a packing coloring of $ST^5_3$ (and of $ST^n_3$). For each color $i$, $1 \leq i \leq 17$, used for the packing coloring of $ST^4_3$, which is shown in Fig.~\ref{fig:trikotnik_barvanje}, the following holds: 
\begin{enumerate}[(i)] 
\item
if there exist at least two vertices of $ST^4_3$, both colored by color $i$, then any such two vertices are at distance at least $i+1$; 
\item
if there exist at least two vertices of $ST^4_3$, both colored by color $i$, then for any such two of them, denoted by $a$ and $b$, the sum of the distance from vertex $a$ to the nearest extreme vertex and the distance from vertex $b$ to its nearest extreme vertex is at least $i+1$ (the color of extreme vertices need not be considered, because they are identified when building $ST^5_3$ from three $ST^4_3$'s);
\item
if there exists only one vertex of $ST^4_3$ colored by color $i$, then the sum of the distances from this vertex to any two extreme vertices, is at least $i+1$. \\
\end{enumerate}
This implies that any two vertices of $ST^n_3$, $n \geq 4$, both colored by color $i \in [17]$, are at distance more than $i$. 

It is clear, that two vertices of $ST^5_3$, colored by the same color $i \in \{18, 19, 20, 21\}$, are at distance at least $i+1$. Also the sum of the distance from one such vertex to the nearest extreme vertex and the distance from another such vertex to its nearest extreme vertex, is at least $i+1$. 
Hence any two vertices, colored by the same color $i \in \{18, 19, 20, 21\}$, are at distance more than $i$ in any of the graphs $ST^n_3$, $n \geq 5$. 

Recall that the diameter of $ST^5_3$ is $32$. For any vertex of $ST^5_3$, colored by color $i$, where $22 \leq i \leq 31$, the sum of the distances from this vertex to any two extreme vertices, is at least $i+1$. Therefore any two vertices of $ST^n_3$, $n \geq 5$, colored by the same color $i$, $22 \leq i \leq 31$, are at distance at least $i+1$.  Hence the described coloring is indeed a $31$-packing coloring of $ST^n_3$ for any $n \geq 5$.

\qed \end{proof}

%%%%%%%%%%%%%%%%%%%%%%%%%%%%%%%%%%%%%%
%%%%%%%%%%%%%%%%%%%%
%OPEN PROBLEMS
%%%%%%%%%%%%%%%%%%%%%%%%%%%%%%%%%%
%%%%%%%%%%%%%%%%%%%%%%%%%%%%%%%%%%
\section{Open problems}
One can think of several natural open questions related to packing colorings of Sierpi\' nski-type graphs. We mention three problems that directly arise from results of this paper.

\begin{problem}
\label{p1}
Characterize the graphs $G$ for which the sequence $(\chi_\rho(S^n_G))_{n \in \NN}$ is bounded.
\end{problem}

\noindent  Note that from Corollary~\ref{cor:K4} follows that such graphs $G$ do not contain $K_4$ as a subgraph. Moreover, it can be proven in a similar way as in the proof of Corollary~\ref{cor:K4} that if $H$ is a subgraph of $G$, then $S^n_H$ is a subgraph of $S^n_G$ for any $n$. Hence the graphs for which the sequence $(\chi_\rho(S^n_G))_{n \in \NN}$ is bounded form a hereditary class of graphs. 

\begin{problem}
\label{p2}
Determine the exact values of $\chi_\rho(S^n_{K_4-e})$ and $\chi_\rho(S^n_{paw})$ for all $n$. 
\end{problem}

\noindent As Problem~\ref{p2} might be difficult, it would also be interesting to improve the lower and upper bounds for $\chi_\rho(S^n_{K_4-e})$ and $\chi_\rho(S^n_{paw})$ proven in this paper.

\begin{problem}
\label{p3}
What is the maximum of the set $\{\chi_{\rho}(ST^n_3)\,|\, n\in \NN\}$?
\end{problem}

\noindent  Again, it might already be challenging to improve the currently known upper bound for the packing chromatic number of the Sierpi\' nski triangle graphs, which is $31$ by Theorem~\ref{th:trikotnik}.

\section*{Acknowledgement}
B.B. acknowledges the financial support from the Slovenian Research Agency (research core funding No.\ P1-0297 and the project grant J1-7110).

\end{document}